\documentclass[11pt]{amsart}

\usepackage{amsmath, amssymb, amsthm, color, hyperref, bbm}
\usepackage{tikz-cd, tikz}
\usetikzlibrary[positioning,patterns,shapes]

\newcommand{\bs}{\backslash}
\newcommand{\C}{\mathbb{C}}
\newcommand{\fg}{\mathfrak{g}}
\newcommand{\ftg}{\widetilde{\mathfrak{g}}}
\newcommand{\ft}{\mathfrak{t}}
\newcommand{\ff}{\mathfrak{f}}

\newcommand{\Cx}{{\mathbb{C}^\times}}
\newcommand{\Cxm}{ {\mathbb{C}_{-\mu-\frac{1}{2} \mu_0}^\times} }
\newcommand{\Cxma}{ {\mathbb{C}_{-\mu + \frac{1}{2}\alpha}^\times} }

\newcommand{\Cxz}{ {\mathbb{C}_{-\frac{1}{2} \mu_0}^\times} }
\newcommand{\Cxzf}{ {\mathbb{C}_{[-\frac{1}{2} \mu_0]}^\times} }

\newcommand{\Z}{\mathbb{Z}}
\newcommand{\N}{\mathbb{N}}
\newcommand{\tG}{{\widetilde{G}}}

\newcommand{\A}{\mathcal{A}}
\newcommand{\Az}{ {\mathcal{A}_{-\zeta - [\frac{1}{2} \mu_0], \hbar}} }
\newcommand{\Am}{ {\mathcal{A}_{[-\mu - \frac{1}{2} \mu_0] , \hbar}} }

\newcommand{\hh}{\hbar}
\newcommand{\GK}{{G_{\mathcal{K}}}}
\newcommand{\GO}{{G_{\mathcal{O}}}}
\newcommand{\NK}{{N_{\mathcal{K}}}}
\newcommand{\NO}{{N_{\mathcal{O}}}}
\newcommand{\ttGK}{{\widetilde{G}_{\mathcal{K}}}}
\newcommand{\ttGO}{{\widetilde{G}_{\mathcal{O}}}}
\newcommand{\ttGOK}{{\widetilde{G}^{\mathcal O}_{\mathcal K} \rtimes \Cx}}
\newcommand{\ttGOKnoloop}{{\widetilde{G}^{\mathcal O}_{\mathcal K}}}
\newcommand{\FO}{F_\cO}
\newcommand{\cK}{{\mathcal K}}
\newcommand{\cO}{{\mathcal O}}
\newcommand{\sT}{\mathcal{T}}
\newcommand{\sR}{\mathcal{R}}
\newcommand{\sZ}{\mathcal{Z}}

\DeclareMathOperator{\Hom}{Hom}
\DeclareMathOperator{\val}{val}
\newcommand{\PP}{\mathbb{P}}
\newcommand{\OO}{\mathcal{O}}
\newcommand{\sF}{\mathcal{F}}

\newcommand{\fh}{\mathfrak{h}}
\newcommand{\Sym}{\operatorname{Sym}}
\newcommand{\QM}{\operatorname{QM}}
\newcommand{\Spec}{\operatorname{Spec}}
\newcommand{\Proj}{\operatorname{Proj}}
\newcommand{\Rees}{\operatorname{Rees}}
\newcommand{\id}{\operatorname{id}}
\newcommand{\Gr}{\mathrm{Gr}}
\newcommand{\Sp}{\mathrm{Sp}}
\newcommand{\Vco}{\overset{\circ}{V_c}}
\newcommand{\Lco}{\overset{\circ}{L_c}}
\newcommand{\Xco}{\overset{\circ}{X_c}}

\newcommand{\La}{\mathcal{L}}
\newcommand{\Za}{\mathcal{Z}}
\newcommand{\cF}{\mathcal{F}}

\newcommand{\cW}{\mathcal{W}}

\newcommand{\groupoid}{\mathcal{P}}
\newcommand{\inclu}{\alpha}
\newcommand{\ontu}{b}
\newcommand{\ontuu}{\tilde{b}}
\newcommand{\eqstab}{{L}_c}

\newcommand{\arxiv}[1]{\href{http://arxiv.org/abs/#1}{\tt arXiv:\nolinkurl{#1}}}

\newcommand{\ssl}{\mathfrak{sl}}

\newcommand{\excise}[1]{}

\newcommand{\sslash}{\mathbin{/\mkern-6mu/}}
\newcommand{\sssslash}{\mathbin{/\mkern-6mu/\mkern-6mu/\mkern-6mu/}}

\newtheorem{Theorem}{Theorem}[section]
\newtheorem{Proposition}[Theorem]{Proposition}
\newtheorem{Lemma}[Theorem]{Lemma}

\newtheorem{Corollary}[Theorem]{Corollary}
\newtheorem{Conjecture}[Theorem]{Conjecture}

\theoremstyle{definition}
\newtheorem{Definition}[Theorem]{Definition}
\newtheorem{Example}[Theorem]{Example}
\newtheorem{Remark}[Theorem]{Remark}

\newcommand{\basissigma}{| \sigma \rangle}
\newcommand{\basissigmas}{| [\sigma] \rangle}
\newcommand{\basissigmamax}{| \sigma^{\max} \rangle}

\newcommand{\basislambdasigma}{| \lambda + \sigma \rangle}
\newcommand{\basislambdasigmas}{| [\lambda + \sigma] \rangle}

\title{BFN Springer Theory}
\author{Justin Hilburn}
\email{jhilburn@perimeterinstitute.ca}
\author{Joel Kamnitzer} 
\email{jkamnitz@math.toronto.edu}
\author{Alex Weekes} 
\email{weekes@math.usask.ca}

\begin{document}
\maketitle

\begin{abstract}
Given a representation N of a reductive group G,
Braverman-Finkelberg-Nakajima have defined a remarkable Poisson
variety called the Coulomb branch. Their construction of this space
was motivated by considerations from 3d gauge theories and
symplectic duality. The coordinate ring of this Coulomb branch is
defined as a convolution algebra, using a vector bundle over the
affine Grassmannian of G.

This vector bundle over the affine Grassmannian maps to the space of
loops in the representation N.  We study the fibres of this maps,
which live in the affine Grassmannian. We use these BFN Springer
fibres to construct modules for (quantized) Coulomb branch algebras.
These modules naturally correspond to boundary conditions for the
corresponding gauge theory.

We use our construction to partially prove a
conjecture of Baumann-Kamnitzer-Knutson and give evidence for
conjectures of Hikita, Nakajima, and Kamnitzer-McBreen-Proudfoot.  We
also prove a relation between BFN Springer fibres and quasimap spaces.
\end{abstract}

\section{Introduction}

\subsection{The Coulomb branch and BFN Springer fibres}
Let $G$ be a reductive group and let $M$ denote a quaternionic representation of $G$. Associated to the pair $(G, M)$ physicists have defined a $3$d $\mathcal{N} = 4$ gauge theory. One interesting invariant of the theory is its moduli space of vacua, which is a union of hyperkahler manifolds. For some time, it was an open question how to compute the ring of algebraic functions on a certain irreducible component of this moduli space, known as the Coulomb branch \cite{BDG}. Braverman-Finkelberg-Nakajima \cite{BFN1} solved this problem completely under the assumption that $M $ admits a $G$-invariant Lagrangian splitting $ M = N \oplus N^*$. 

Their construction of this ring begins by defining a certain convolution variety, analogous to the Steinberg variety. Let $ D $ denote the formal disk and $ D^\times $ the punctured disk.  Define,
 \begin{align*}
\sT_{G,N} = &\{(P, \sigma, s) : \text{ $ P $ is a principal $G$-bundle on the $ D $,} \\
&\text{$\sigma$ is a trivialization of $ P $ on $ D^\times $, $s \in \Gamma(D, N_P) $}\}
\end{align*}
This space is an (infinite-rank) vector bundle over the affine Grassmannian and also comes with a map $ \sT_{G,N} \rightarrow \NK := \Gamma(D^\times, N) $.  We can form the fibre product of two copies of this space over $ \NK $, 
$$
\sZ_{G,N} = \sT_{G,N} \times_{\NK} \sT_{G,N} 
$$

Then \cite{BFN1} define $ \A_0 := H_\bullet([\sZ_{G,N}/\GK]) $ to be the homology of the stack quotient of $ \sZ_{G,N}$ by the group $ \GK $.  The homology $ \A_0 $ carries a convolution algebra structure.  Due to infinite-dimensionality issues, the definition of this homology from \cite{BFN1} is a bit subtle, see section \ref{se:defA}.  

Braverman-Finkelberg-Nakajima prove that $ \A_0 $ is a commutative algebra and define the Coulomb branch by $ \mathcal M_C(G,N) := \Spec \A_0 $.  They also define a non-commutative version $ H_\bullet^{\Cx}([\sZ_{G,N}/\GK])$, where $ \Cx $ acts by rotation on the formal disk $ D $. 

There is a natural analogy between $ \sT_{G,N} \rightarrow \NK $ and the Springer resolution $ T^* \mathcal B \rightarrow \mathcal N $ (where $ \mathcal B $ is the flag variety and $ \mathcal N $ is the nilpotent cone).  In classical Springer theory, one considers fibres of the Springer resolution and uses them to construct representations of the homology of the Steinberg variety $ T^* \mathcal B \times_{\mathcal N} T^* \mathcal B $.

For this reason, it is natural to consider fibres of the map $ \sT_{G,N} \rightarrow \NK $, which we call BFN Springer fibres.  For any $ c \in \NK $, this fibre is given by
$$
\Sp_c := \{ (P, \sigma) : \sigma^{-1}(c) \in \Gamma(D, N_P) \} = \{ [g] \in \GK/\GO : g^{-1} c \in \NO \}
$$
For the second description, we undo the usual modular description of the affine Grassmannian $ \Gr_G = \GK/\GO $.

The purpose of this paper is to study these BFN Springer fibres and use them to construct modules for these Coulomb branch algebras.  Our construction can be viewed as a generalization of the construction of Cherednik algebra modules using the homology of affine Springer fibres (see Varagnolo-Vasserot \cite{VV} and Oblomkov-Yun \cite{OY}).

\subsection{Physics Interpretation}
In \cite{BDGH}, Bullimore-Dimofte-Gaiotto-Hilburn studied boundary conditions for these $3$d $\mathcal{N}=4$ gauge theories that preserve $2$d $\mathcal{N}=(2,2)$ supersymmetry.  They showed that from each boundary condition one can construct a module over the Coulomb branch algebra. 

One particular nice class of boundary conditions are parameterized by triples $\mathcal{B} = (N, c, H)$ where $c \in N$, and $H \subseteq \text{Stab}_G(c)$. In the physics literature the choice $H=1$ is referred to as Dirichlet boundary condition for the gauge field. The choice $H=G$ is called the Neumann boundary condition for the gauge field. 

In \cite{DGGH} and \cite{HY}, the first author and his coauthors showed that the compactification of the $A$-twisted $3$d theory on a circle gives rise to the $2$d TQFT associated to the category of $D$-modules on $ N_\cK/\GK $. (See \cite{BF, CG2, CCG} for other perspectives on these ideas.) In particular, each boundary condition can be wrapped on the circle to give an object of this category. The boundary condition $\mathcal{B}$ gives rise to the delta functions supported on $c/H_{\cK} \hookrightarrow N_{\cK}/G_\cK$. Furthermore, it was also shown that each vortex line gives rise to a $D$-module on the loop space. In particular, the trivial vortex line $\mathbbm{1}$ is associated with the structure sheaf of $N_{\cO}/G_{\cO}$. From this perspective, the work of Braverman-Finkelberg-Nakajima shows that $\text{End}(\mathbbm{1}) = \A_0$, as is expected on physical grounds.

In \cite{DGGH}, it is argued that the space of local operators living at the intersection of a half space with boundary condition $\mathcal{B}$ and the trivial line operator $\mathbbm{1}$ oriented transverse to the boundary is given by
\[
\Hom(\mathbbm{1}, \mathcal{B}) \cong H^{H_{\cK}}_{\bullet}(\Sp_c).
\]
When one can make sense of the $H_{\cK}$-equivariant homology, we expect that this is an $\A_0$-module. One can elaborate on this construction by replacing $\mathbbm{1}$ with a flavour vortex line labeled by a cocharacter $\mu: \Cx \to F$, or equivalently by replacing $c$ with $z^\mu c$, to construct a module which we will call $\mathcal{B}[\mu]$. Physically, taking equivariant homology for the loop rotation corresponds to turning on an omega background.

For physical reasons it is expected that these modules coincide with the ones constructed by \cite{BDGH} and we show this in the abelian case. Furthermore, the local operators in $A$-twisted $2$d $\mathcal{N}=(2,2)$ theory are closely related to quantum cohomology so on physical grounds we expect that these modules also coincide with the ones constructed by Teleman in \cite{T}. We will return to this in future work.

\subsection{Geometric properties}
The affine Grassmannian $ \GK/\GO $ has connected components $ \Gr(\sigma) $ labelled by $ \sigma \in \pi_1(G)$ (for us $ G $ is often a product of general linear groups and thus $ \pi_1(G) $ is infinite).  Moreover each connected component $ \Gr(\sigma) $ is usually infinite-dimensional (unless $ G $ is a torus), so potentially $ \Sp_c $ is quite a complicated space.  For example, if $ N $ is the adjoint representation, then $ \Sp_c $ is the usual affine Springer fibre, see for example \cite{Zhiwei}.  However, we consider situations in which $\Sp_c $ behaves more like the following example.

\begin{Example}
	Suppose that $ G = GL_k $ and $ N = \Hom(\C^n, \C^k) $ with $ k \le n $.  Let $ c \in \NK $ be a surjective map.  
	
	If $ c \in N(\C)$, then it is easy to see that $ \Sp_c $ is actually independent of $ c $ and coincides with the positive part of affine Grassmannian of $ GL_k $.  This is the locus of all $ \cO $-lattices in $\cK^k $ which contain $ \cO^k $.  Thus each connected component of $ \Sp_c $ is a finite-dimensional projective variety.
\end{Example}

Our first result (Theorem \ref{th:finitedim}) shows that the components $ \Sp_c(\sigma) = \Sp_c \cap \Gr(\sigma)$ are well-behaved under the assumption that $ c \in \NK $ is GIT-stable.

\begin{Theorem}
	Let $ c \in \NK $.  Assume that $ c$ is $ \chi$-stable for some $ \chi: \GK \rightarrow \cK^\times $, with trivial stabilizer.  Then for each $ \sigma \in \pi_1(G)$, $ \Sp_c(\sigma) $ is a finite-dimensional projective variety.
\end{Theorem}

\subsection{Modules}

The second theorem of this paper is the construction of modules using BFN Springer fibres and more generally certain equivariant sheaves. Since the category of $ \GK$-equivariant $D$-modules on $\NK $ is difficult to work with technically, we will work with $\groupoid$-equivariant sheaves on $ \NO $, where
$$
\groupoid = \{(g, v) : g \in \GK, v \in \NO, gv \in \NO \} 
$$
is the groupoid obtained by restricting the $\GK $-action to $ \NO $.  In section \ref{se:DefAction}, (Theorem \ref{th:action}) we prove the following result.

\begin{Theorem} \label{th:introModule}
	For any $\groupoid$-equivariant sheaf $ \mathcal F$, there is a $ \A_0 $-module structure on $ H^{-\bullet}_{\GO}(\NO, \mathcal F)$.
\end{Theorem}

If $ \GK $ acts with finite stabilizer $\Lco $ on $ c \in \NK $, then there is a $\groupoid$-equivariant sheaf $ \mathcal F $ such that $ H^{-\bullet}_{\GO}(\NO, \mathcal F) = H_\bullet^{\Lco}(\Sp_c)$ and thus we get an $\A_0$-module structure on $ H_\bullet^{\Lco}(\Sp_c)$.

The proof of Theorem \ref{th:introModule} involves adapting the proof from \cite{BFN1} of the associativity of the algebra $ \A_0$ and inserting our sheaf $ \mathcal F $ into the process.

\subsection{Flavour symmetry and weight modules}
In the theory of Coulomb branches, it is frequently important to consider the presence of another symmetry group $ F$, called the \textbf{flavour group}.  Thus assume that we are given an extension
$$
1 \to G \to \tG \to F \to 1
$$
where $ F $ is a torus, and an action of $ \tG $ on $ N $ extending the action of $ G $.

Braverman-Finkelberg-Nakajima \cite{BFN1} used the loop rotation action and the flavour group to define the non-commutative Coulomb branch algebra $ \A = H_\bullet^{\ttGO \rtimes \Cx}(\sR) $ containing $ H^\bullet_{F \times \Cx}(pt) $ in its centre.  We modify Theorem \ref{th:introModule} appropriately to get an action of this larger algebra.  In particular, we get an action of $ \A $ on $  H_\bullet^{\eqstab}(\Sp_c)$, where $ \eqstab $ is the stabilizer of $ c$ in $ \ttGOK $.

The algebra $ \A $ contains a large commutative subalgebra $ H^\bullet_{\tG \times \Cx}(pt) $, called the \textbf{Gelfand-Tsetlin subalgebra}, and a smaller commutative subspace $ H^2_G(pt) $, called the \textbf{Cartan subalgebra}.  In this paper, we will study weight modules for $ \A$, that is modules on which the Cartan subalgebra acts semisimply after specializing the centre (see Definition \ref{def:Weight}).  We prove the following result (Corollary \ref{co:weight}).

\begin{Theorem} \label{th:introweight}
	If $ c \in \NK $ has trivial stabilizer in $ G $, then $ H_\bullet^{\eqstab}(\Sp_c) $ is a weight module with weight spaces $ H_\bullet^{\eqstab}(\Sp_c(\sigma)) $.
\end{Theorem}
In this generality, we cannot describe precisely how the Gelfand-Tsetlin algebra acts on these modules.  However, in Lemma \ref{le:highestweight} and Theorem \ref{thm: GT bases}, under more restrictive assumptions on $ c$, we find eigenvectors for the action of the Gelfand-Tsetlin algebra and descibe the eigenvalues.

Another important class of modules is category $ \OO $.  To specify this category, we need to choose a character $ \chi : G \rightarrow \Cx $.  This collapses the $ \pi_1(G) $-grading on $ \A $ to a $\Z $-grading and thus we may consider category $ \OO$ modules, which are those weight modules with bounded above weight spaces. The following result (Theorem \ref{th:catO}) is a fairly easy application of the definition of stability.

\begin{Theorem}
	Assume that $ c \in \NK $ is $ \chi$-stable.  Then $ H_\bullet^{L_c}(\Sp_c) $ lies in category $ \OO $.
\end{Theorem}

\subsection{Fixed points and changing Lagrangians}

Consider the Hamiltonian reduction $ M = T^* N \sssslash G $; this is called the \textbf{Higgs branch} of the corresponding gauge theory.  An important conjecture of Hikita and Nakajima (see \cite{Hikita} and \cite[\S 8]{moncrystals}) relates the equivariant cohomology of this Higgs branch to the $B$-algebra of $ \A $.  Under the assumption of isolated fixed points, this implies that there is bijection between $ F $--fixed points $ (T^* N \sssslash G)^F $ in the Higgs branch and Verma modules for $ \A $.  More precisely, given a fixed point $ [c]$, we get a splitting $ F \rightarrow \tG $ whose image is the stabilizer of $ c $ in $ \tG $.  The fixed point $ [c] $  should correspond to a highest weight module for $ \A $, such that the action of the Cartan subalgebra of $ \A $ on the highest weight space is determined by the above splitting. 

Consider now a fixed point $ c \in (N \sslash G)^F $.  In this case, our general construction gives a module $ M_c := H_\bullet^{F \times \Cx}(\Sp_c) $.  In section \ref{se:Hikita} we verify that these modules lie in category $ \OO $ and have the desired highest weights.  Thus, we give a conceptual explanation for the Hikita-Nakajima conjecture.  (The reader should note that our method only covers those fixed points which lie in these $N \sslash G $, see section \ref{sec: hwtftsy}.)

Braverman-Finkelberg-Nakajima showed that the algebra $ \A $ depends only on $ M $ and not on $ N $ \cite[\S 6(viii)]{BFN1}.  If we pick a different $ G $-invariant Lagrangian subspace $ N' \subset T^*N $, then we get an isomorphism $ \A(G,N) \cong \A(G,N') $.  Thus our construction can be viewed as giving modules $ M_c $ for all $ c \in N_\cK$.  However, if $c \in N_\cK \cap N'_\cK$ happens to lie in two different Lagrangians, then the modules we obtain can be different, as can be seen in the abelian case.

\subsection{Toric case}
Suppose that $ G = T $ is a torus. In this case, the Springer fibres and the associated modules can be very easily understood.  Using our method, we construct all the simple and Verma modules in category $\mathcal O$.  Furthermore one can encode the structure of these modules using hyperplane arrangements as in \cite{BLPW1, BDGH}.  

\subsection{Quiver case}
Another important case is when the pair $ G, N $ comes from a quiver.  Fix a finite quiver $Q $ with vertex set $I$.  Fix $I$-graded vector spaces $ V, W $.  Let $ G = \prod GL(V_i)$ and let $ N $ be the usual vector space of framed representations of the quiver $ Q $ on the vector spaces $V, W $.  Thus a point $ c \in N $ corresponds to a representation $ V $ of the quiver $ Q $ along with a framing $ W \rightarrow V $. For the character of $ G $ given by the product of determinants, the point $ c $ is stable when the image of $ W $ generates $ V $ as a $ \C Q $-module.  In section \ref{se:SpringerQuiver}, we give the following description of the Springer fibre for stable points.

\begin{Theorem} \label{le:introQuiver}
	Assume that $ c \in N$ is stable.  Then
	$$
	\Sp_c = \{ V_\cO \subseteq L \subset V_\cK : L \text{ is an $ \C Q \otimes \cO $ submodule of }  V_\cK  \}
	$$
\end{Theorem}
In particular, this shows that $ \Sp_c $ is independent of the choice of framing.

In this quiver case, the Coulomb branch is a generalized affine Grassmannian slice.  Thus, Theorem \ref{th:introModule} provides a quasi-coherent sheaf on the generalized affine Grassmannian slice whose space of sections is given by the homology of this space of $ \C Q \otimes \cO$-submodules.  

\begin{Remark}
Motivated by the theory of MV polytopes and Duistermaat-Heckman measures,
Baumann, Knutson, and the second author \cite{BKK} formulated a general conjecture relating preprojective algebra modules and MV cycles.  In particular, given any preprojective algebra module, we conjectured the existence of a quasi-coherent sheaf whose support was a union of MV cycles.  As explained in section \ref{se:GeomSlices}, Theorems \ref{th:introModule} and  \ref{le:introQuiver} establish our conjecture for the case where the preprojective algebra module comes from a quiver representation. This was one of the main motivations for the current paper.
\end{Remark}

\subsection{Quasimap spaces}
The space $ \Sp_c $ can be viewed as the space of maps from $ D $ to the stack $ [N/G] $, whose restriction to $ D^\times $ is given by $ c $.  For this reason it is interesting to relate $\Sp_c $ to spaces of based maps from $ \PP^1$ to $[N/G]$; these are usually called ``based quasimaps'' in the literature.

To explain the relation, let us begin with the special case where $ G = \prod_{i=1}^{n-1} GL_i $ and $ N = \oplus_{i=1}^{n-1} \Hom(\C^i, \C^{i+1}) $ (coming from the quiver in Figure \ref{fig:1}).  In this case, the Coulomb branch $ \mathcal M_C(G,N)$ coincides with the nilpotent cone of $ GL_n $.  Moreover, the quantized algebra $ \A $ coincides with the asymptotic enveloping algebra $ U_\hbar \mathfrak{gl}_n $ (see Theorem \ref{th:sln}).  We choose $ c \in N $ to correspond to the standard embeddings $ \C^i \hookrightarrow \C^{i+1} $.
\begin{figure}
\begin{tikzpicture}[scale=0.6]
\draw  (0,0) circle [radius=1];
\draw (4,0) circle[radius=1];
\draw (12,0) circle [radius=1];
\draw [->] (1.2,0) -- (2.8,0);
\draw [->] (5.2,0) -- (6.8,0);
\draw [->] (9.2,0) -- (10.8,0);
\draw [->] (13.2,0) -- (14.8,0);
\node at (0,0) {$1$};
\node at (4,0) {$2$};
\node at (12,0) {$n-1$};
\node at (8,0) {$\cdots$};
\draw (15,-1) rectangle (17,1);
\node at (16,0) {$n$};
\end{tikzpicture}
\caption{A linearly oriented, type $A_{n-1}$ quiver with framing at the last node} \label{fig:1}
\end{figure}
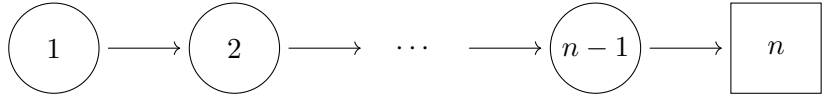
The space $ \La_n := QM_c(\PP^1,[N/G])$ of based maps from $ \PP^1, \infty $ to the stack $ [N/G], c $ is called Laumon's quasiflag space, because it parametrizes flags of locally free sheaves on $ \PP^1 $, which are standard at $ \infty $.  A celebrated result of Kuznetsov \cite{Kuz} shows that Laumon's space $ \La_n $ provides a small resolution of the Drinfeld's Zastava space  $ \Za_n$, which is the space of based maps from $ \PP^1 $ to the stack $ [SL_n \sslash U_n / T_n] $, where $U_n $ the maximal unipotent subgroup of $ GL_n $ and $ T_n $ its torus. 

In this case, the central fibre of the map $ \La_n \rightarrow \Za_n $ is isomorphic to the BFN Springer fibre $ \Sp_c$, and as each component of the Zastava space is contractible, this provides an isomorphism $ H_\bullet(\La_n) \cong H_\bullet(\Sp_c) $.

Now, let $ G, N $ be an arbitrary pair as before and let $ c \in N$. In order to generalize the computation of the equivariant $I$-function of $G/P$ in \cite{B} to a computation of equivariant $I$-functions of Lagrangians in other symplectic resolutions, \cite{BDGH} studied the space of based quasimaps $QM_c(\PP^1, [N/G])$.  In particular, they gave an action of the quantized Coulomb branch algebra on the cohomology of this space.

 In section \ref{se:various}, we generalize the Laumon/Drinfeld setup by defining Zastava-type quasimaps $ QM_c(\PP^1, [N\sslash G'/H]) $, where $  G' $ is the commutator subgroup of $ G $ and $ H = G/G' $ (a torus).  These quasimaps have a degree $ \sigma \in \pi_1(G) $ and we have decompositions $ QM(\dots) = \cup_\sigma QM^\sigma(\dots) $. Under a mild assumption on $c$, we can prove the following theorem.

\begin{Theorem} \label{th:introQM}
Let $ G, N$ be as above and assume $ c \in N $ has trivial stabilizer in $ G$.
	\begin{enumerate}
		\item There is a natural map 
		$$  QM_c(\PP^1, [N/G]) \to QM_c(\PP^1, [N\sslash G'/H]) $$
		\item For each $ \sigma$, $QM_c^\sigma(\PP^1, [N\sslash G'/H]) $ is an affine scheme contractible to the constant map $ c $.
		\item Assume that the scheme theoretic fibre of $ N \rightarrow N \sslash G' $ over the point $ [c] $ is isomorphic to $ G'$.  Then the central fibre of  
$  QM_c^\sigma(\PP^1, [N/G]) \to QM_c^\sigma(\PP^1, [N\sslash G'/H]) $ is isomorphic to $ \Sp_c(\sigma)$.
\end{enumerate}
\end{Theorem}
This theorem is illustrated by the diagram
\begin{equation*}
\begin{tikzcd}
\Sp_c(\sigma) \ar[r, hook] \ar[d] & QM_c^\sigma(\PP^1, [N/G]) \ar[d] \\
\{c\} \ar[r, hook] & QM_c^\sigma(\PP^1, [N\sslash G'/H])
\end{tikzcd}
\end{equation*}
The main step in the proof is a careful study (section \ref{se:QMtorus}) of the space of based quasimaps into a torus quotient.

\begin{Remark}
As a corollary of this theorem, we see that $ H_\bullet(\Sp_c) \cong H_\bullet(  QM_c(\PP^1, [N/G]) $.  Thus by \textit{transport de structure}, $ H_\bullet(QM_c(\PP^1, [N/G]) $ acquires an $ \A$-module structure. It is an open question whether this coincides with the other $\A$-actions on (co)homology of quasimap spaces constructed by \cite{FFFR, BFFR, NakHandsaw, BDGHK}.

For example, $ H_\bullet( \La_n) $ acquires a $ U_\hbar \mathfrak{gl}_n$-module structure by this theorem.  On the other hand, in \cite{FFFR}, Feigin-Finkelberg-Frenkel-Rybnikov constructed a module structure in (localized) equivariant cohomology.  Using results on compatibility with localization from Garner-Kivinen \cite{GK}, it should be possible to relate these two actions.  

More generally, by changing the dimension vectors in Figure \ref{fig:1}, we obtain actions of finite W-algebras on the homology of parabolic Laumon spaces, see Proposition \ref{prop: W algebras}.  These should be related to similar actions defined by Braverman-Feigin-Finkelberg-Rybnikov \cite{BFFR} and Nakajima \cite{NakHandsaw}.
\end{Remark}

\begin{Remark}
The generating function (with respect to the degree $ \sigma $) for the Euler characteristics of the quasimap space $ QM_c(\PP^1, [N/G]) $ is closely related to the $I$-function which appears in the study of the quantum cohomology of the Higgs branch $ T^*N \sssslash_\chi G$ (for the Calabi-Yau specialization).  Thus Theorems \ref{th:introweight} and \ref{th:introQM} relate the characters of modules for $ \A $ to this quantum cohomology.  This provides a conceptual explanation for a conjecture of the second author with McBreen and Proudfoot \cite{KMP}.  

\end{Remark}

\subsection{Other relations to the literature}
This paper represents one idea of how to construct Coulomb branch algebra modules using the geometry of the BFN space.  Another approach has been pursued by Webster \cite{Web}, in which he established a Koszul duality between Higgs and Coulomb branches.  The exact relationship between our approach and Webster's remains an interesting subject for future research.  Yet a third approach (perhaps closer in spirit to Webster's than to ours) has been pursued by Nakajima (unpublished).

One advantage of our approach is that it produces geometric descriptions of modules.  This is very well illustrated by the recent work of Garner and Kivinen \cite{GK} who use our results to give an action of the spherical rational Cherenik algebra on the homology of Hilbert schemes on plane curve singularities.

\subsection{Acknowledgements}
	
	We would like to thank Roman Bezrukavnikov, Alexander Braverman, Sabin Cautis, Tudor Dimofte, Michael Finkelberg, Dennis Gaitsgory, Davide Gaiotto, Niklas Garner, Oscar Kivinen, Allen Knutson, Michael McBreen, Dinakar Muthiah, Hiraku Nakajima, Andrei Negut, Nicholas Proudfoot, Sam Raskin, Ben Webster, and Zhiwei Yun for helpful conversations.

\section{The setup}
Let $ G $ be a reductive group.  Let $ T $ denote a maximal torus of $ G $.

We write  $\mathcal K = \C((z)), \mathcal O = \C[[z]] $ and we consider the groups $ \GK = G((z)), \GO = G[[z]] $.   We will study the affine Grassmannian $ \Gr_G = \GK/\GO$.

Assume now that we have an extension
$$
1 \rightarrow G \rightarrow \tG \rightarrow F \rightarrow 1
$$
where $ F $ is a torus.  Following the physics literature, we call $ F $ the \textbf{flavour torus}.

We define the group $\ttGK^\cO$ to be the preimage of $ F_\cO $ under the map $ \ttGK \rightarrow F_\cK$. We have that $ \Gr = \ttGK^\cO / \ttGO $.  

We have an action of $ \Cx $ on $ \ttGK $ and $ \ttGK^\cO $ by loop rotation and we can form the semidirect products $ \ttGK \rtimes \Cx $ and $\ttGOK $ (and similarly for $ \ttGO $).  We have that $ \Gr = \ttGOK / \ttGO \rtimes \Cx $.

\subsection{GIT and the Higgs branch}

Let $ N $ be a representation of $ \tG $. We assume that there exists a central $ \mu_0 : \Cx \rightarrow \tG $ giving the scaling action on $ N $; i.e. for all $ s \in \Cx $ and $ v \in N$, we have $ s^{\mu_0} v = s v $ and $ s^{\mu_0} $ lies in the centre of $ \tG $.  The existence of this $ \Cx $ will be useful to us later.  In most situations,  there is a natural $ \mu_0 $.  In any case, $ \widetilde G $ can always be enlarged to ensure the existence of such a $ \mu_0 $.

We will be interested in quotients of $ N $ and its cotangent bundle by the action of $ G$.

Fix a homomorphism $\chi : G \rightarrow \Cx $. For each $n \in \N $, we define the $n\chi$-semi-invariant functions on $ N $ by
$$
\C[N]^{G, n \chi} = \{ f \in \C[N] : f(g v) = \chi(g)^n f(v) \text{ for all } g \in G, v\in N\}
$$
These form the components of a graded ring.  Thus we can form the projective GIT quotient $$ N \sslash_\chi G = \operatorname{Proj} \bigoplus_{n\geq 0} \C[N]^{G, n \chi}.$$

Now consider the cotangent bundle $T^* N = N \oplus N^*$.  We have a moment map $ \Phi : T^*N \rightarrow \fg $, defined by $ \Phi(\alpha, v)(X) = \alpha(Xv)$. We can form the symplectic reduction $$ T^* N \sssslash_\chi G = \Phi^{-1}(0) \sslash_\chi G. $$ called the \textbf{Higgs branch} corresponding to the pair $G, N$. 

These quotients carry an action of the flavour torus $ F $ and we have an inclusion $ N \sslash_\chi G \rightarrow T^* N \sssslash_\chi G $.

\begin{Definition}
Let $ c \in N$.  We say that $ c $ is $ \chi$-\textbf{semistable}, if there exist $ n \in \N, f \in \C[N]^{G, n\chi} $ such that $ f(c) \ne 0 $.  We say that $ c $ is $\chi$-\textbf{polystable}, if there exist $ n \in \N, f \in \C[N]^{G, n\chi} $ such that $ f(c) \ne 0 $ and the orbit $ G c $ is closed in $ \{v \in N : f(v) \ne 0 \} $.  Finally, we say that $ c $ is $ \chi$-\textbf{stable}, if it is polystable and if $ c $ has finite stabilizer in $ G $.
\end{Definition}

Consider the action of $G$ on $N\times \C$ defined by $g\cdot (v, a) = (gv, \chi^{-1}(g) a)$. Then an element $c\in N$ is $\chi$--semistable if and only if $N \times \{0\}$ does not meet the closure of the orbit $G\cdot (c, 1)$.

The following result (which we were not able to find in the literature) will be very useful to us later.

\begin{Lemma} \label{le:StableProper}
If $ c \in N$ is $ \chi$-stable, then the map $ G \rightarrow N \times \C $ given by $ g \mapsto (gc, \chi(g)^{-1})$ is proper.
\end{Lemma}

\begin{proof}
Since $ c $ is $ \chi $-stable, there exist $ n \in \N, f \in \C[N]^{G,n\chi} $ such that $ f(c) \ne 0 $ and the orbit $ G c $ is closed in $ U := \{v \in N : f(v) \ne 0 \} $.  

Consider the closed subvariety of $N \times \C $ defined by
$$ 
Y = \{(v, a) : f(v) a^n = f(c) \}
$$  
Note that the above morphism $ G \rightarrow N \times \C $ factors through $ Y$.

We have a commutative diagram
\begin{equation*}
\begin{tikzcd}
\ & & Y \arrow{d}\\
G \arrow{r} \arrow{urr} & Gc \arrow{r} & U
\end{tikzcd}
\end{equation*}
The map $ G \rightarrow Gc $ is finite, hence proper.   Since $Gc$ is closed in $ U $, the map $ G \rightarrow U $ is proper.  Since $ Y \rightarrow U $ is separated, we deduce that $ G \rightarrow Y $ is proper.

As $ Y $ is a closed subvariety of $ N \times \C $, we deduce that the given morphism $ G \rightarrow N \times \C $ is proper. 
\end{proof}

\subsection{The BFN space}
Following Braverman-Finkelberg-Nakajima \cite{BFN1}, we will now define certain spaces which are used to construct the Coloumb branch.

Let $ \NO = N \otimes \C[[z]] $ and $ \NK = N \otimes \C((z))$.  These are the sections (over the disc $D $ and punctured disc $ D^\times$, respectively) of the trivial vector bundle with fibre $ N$.

  We can consider the vector bundle
$$
\sT_{G,N} := \GK \times_{\GO} \NO  \rightarrow \Gr
$$
We may identify $ \sT_{G,N} $ as a moduli space as follows.
\begin{align*}
\sT_{G,N} = &\{(P, \sigma, s) : \text{ $ P $ is a principal $G$-bundle on the disc $ D $,} \\
&\text{$\sigma$ is a trivialization of $ P $ on $ D^\times $, $s \in \Gamma(D, N_P) $}\}
\end{align*}
Here $ N_P = P \times_G N $ denotes the associated $ N $ bundle coming from the principal $ G $-bundle.

The map
$$
p: \sT_{G,N} \rightarrow \NK \quad [g,v] \mapsto gv
$$
is an analog of the Springer resolution.  Its fibres were called generalized affine Springer fibres by Goresky-Kottwitz-MacPherson \cite{GKM}. We will refer to these fibres as \textbf{BFN Springer fibres}.

Considering both of the above maps, we see that
$$ \sT_{G,N} = \{([g], w) \in \Gr \times \NK : w \in g \NO \}$$

Now, form the space
$$
\sZ_{G,N} := \sT_{G,N} \times_{\NK} \sT_{G,N} = \{([g_1], [g_2], w) \in \Gr \times \Gr \times \NK: w \in g_i \NO \}
$$
This is an analog of the Steinberg variety.  It admits a moduli description as follows.

$$
\sZ_{G,N} = \{ (P_1, \sigma_1, s_1), (P_2, \sigma_2, s_2) \in \sT_{G,N}^2 : \sigma_1(s_1) = \sigma_2(s_2) \}
$$
In other words we consider two principal $G$-bundles on $ D$, each trivialized away from $ 0 $ and equipped with a section of the associated $N $ bundle, and we demand that the two sections are equal when regarded (using the trivializations) as rational sections of the trivial $N $ bundle.

We have a diagonal action of $ \GK $ on $ \sZ_{G,N} $.  We would like to consider the $ \GK $-equivariant homology of $ \sZ_{G,N} $, but this is difficult because $ \GK $ is badly infinite-dimensional.  To avoid this problem, we follow \cite{BFN1}.

Let
$$
\sR_{G,N} = p^{-1}(\NO) = \{([g], w) \in \Gr \times \NO :  w \in g \NO \}
$$
The modular description of $ \sR_{G,N} $ is as follows.
$$
\sR_{G,N} = \{ (P, \sigma, s) \in \sT_{G,N} : \sigma(s) \in \NO \}
$$
In other words, we consider a principal $ G $ bundle on $ D$, a trivialization $ \sigma $ away from $ 0 $ and a section $ s $ of the associated $N $ bundle, such that $ \sigma(s) $ extends to a section over $ D $ over the  trivial $ N $ bundle.

Then we have $ \sZ_{G,N} = \GK \times_{\GO} \sR_{G,N} $ and thus $[ \sR_{G,N}/\GO ] = [ \sZ_{G,N} / \GK]$.  This stack is the moduli stack of pairs $ (P, s) $ where $ P $ is a principal $G$-bundles on the raviolo curve $ D \cup_{D^\times} D $ and $ s $ is a section of the associated $N$-bundle.

We write simply $ \sT, \sR, \sZ $, when $ G $ and $ N $ are fixed.

Following \cite[page 6]{BFN1}, we define an action of the group $ \ttGK \rtimes \Cx $ on $ \sT_{G,N}, \sZ_{G,N}, \NK$, where $ \Cx $ acts by a combination of loop rotation and scaling of $ N $ with weight $ 1/2 $.  (As in \cite{BFN1}, this means that we are really using the double cover of this $ \Cx $.) Similarly, the group $ \ttGO  \rtimes \Cx $ acts on $ \sR_{G,N}$.

\subsection{The (quantized) Coulomb branch} \label{se:defA}

Following Braverman-Finkelberg-Nakajima \cite{BFN1}, we can form the convolution algebra
$$ \A :=   H_\bullet^{\ttGO \rtimes \Cx}(\sR_{G,N}) $$
The algebra $ \A $ is called the \textbf{flavoured quantized Coulomb branch algebra}.

In the above definition of $ \A$, Braverman, Finkelberg, and Nakajima use a renormalized dualizing sheaf which is denoted $ \omega_{\sR_{G,N}}[-2 \dim \NO] $ and then $ \A $ is defined to be $ \A = H^{-\bullet}_{\ttGO \rtimes \Cx}(\sR_{G,N},  \omega_{\sR_{G,N}}[-2 \dim \NO]) $.  The actual definition of this homology involves finite-dimensional approximations to $ \sR_{G,N} $ and to the group $ \ttGO $; we refer to section 2(ii) of \cite{BFN1} for the precise definition.  Intuitively, we can think that classes in $ \A $ are represented by cycles in $ \sR_{G,N} $ which have finite-dimensional image under the map $ \sR_{G,N} \rightarrow \Gr $ and are finite codimensional along the fibres.

The (nonquantized) \textbf{Coulomb branch algebra} $ \A_0 $ is defined as $ \A_0 = H_\bullet^{\GO}(\sR_{G,N}) $.  It is a commutative algebra and $ \operatorname{Spec} \A_0 $ is called the \textbf{Coulomb branch}.

\subsection{Components of the affine Grassmannian} \label{se:ses}
The group $ \GK $ and the affine Grassmannian $ \Gr $ are disconnected with connected components labelled by $ \pi_1(G) $.  Note that $ \pi_1(G) = X_*(T)/ Q_*(G) $, where $X_*(T) $ denotes the coweight lattice of $ T $ and $ Q_*(G) $ denotes the coroot lattice.  Let $ \pi = \pi_1(G) $.


We get decompositions $ \GK = \sqcup_{\sigma \in \pi}  \GK(\sigma) $ and $ \Gr = \sqcup_{\sigma \in \pi} \Gr(\sigma)$.

For any dominant coweight $ \lambda $, we may also consider $ \Gr^\lambda = \GO z^\lambda $.  These $ \GO $ orbits partition $ \Gr $ and in particular we have  that $ \Gr(\sigma) = \cup \Gr^\lambda $ where the union is taken over those dominant $ \lambda \in X_*(T) $ such that $ [\lambda] = \sigma$.

\subsection{The Springer fibre and the orbital variety}
Let $ c \in \NK $.  Let $ \Lco \subset \GK $ be the stabilizer of $ c $.

\begin{Definition} \label{def:Springer}
We would like to consider the BFN Springer fibre over $ c$.  This will come in two flavours.

First, we can take the actual \textbf{Springer fibre} at c,
$$
\Sp_c := p^{-1}(c) = \{ [g] \in \Gr :  c \in g\NO \}.
$$
From the action of $ \GK$ on $ \Gr $, we see that $\Lco $ acts on $ \Sp_c $.   We define $ \Sp_c(\sigma) = \Sp_c \cap \Gr(\sigma) $ and we see that $ \Sp_c $ is the disjoint union of these $ \Sp_c(\sigma) $.  Note however that the action of $ \Lco $ may not preserve these pieces.

Second, we can consider an \textbf{orbital variety} analog. Namely, let
$$
\Vco = \GK c \cap \NO
$$
The variety $ \Vco $ has an action of $ \GO$.

Finally, we consider a third space which maps to both of these
	$$
\Xco = \{g \in \GK : gc \in \NO \}
$$
\end{Definition}

We will give $\Vco $ a scheme structure as a closed subscheme of the orbit space $ \GK c $.  It will not typically be a locally closed subscheme of $ \NO $, though there will be a morphism $ \Vco \rightarrow \NO $.

\begin{Lemma} \label{le:isoStacks}
We have an isomorphisms of stacks
 $  [\Lco \bs \Sp_c] \cong [ \GO \backslash \Xco / \Lco ] \cong [\GO \bs \Vco]. $
\end{Lemma}

\begin{proof}
	Note that we have maps $ q_1 : \Xco \rightarrow \Vco $ (given by $ g \mapsto g c$) and $ q_2 : \Xco \rightarrow \Sp_c $ (given by $ g \mapsto [g^{-1}] $) and these maps are $ \Lco $ and $ \GO $ principal bundles, respectively.  So the result follows.
\end{proof}

Now, let us assume that $ \Lco $ is finite-type and each component $ \Sp_c(\sigma) $ is of finite type (in Theorem \ref{th:finitedim} we will show that this holds if $ c $ is $\chi$-stable for some character $ \chi$).  In this case, the dualizing sheaf $ \omega_{\Sp_c} $ exists and we can use the above isomorphism of stacks to transfer it to a $ \GO$-equivariant sheaf on $ \Vco $, which we can denote $ \omega_{\Vco}[-2\dim \GO + 2 \dim \Lco] $.  By definition we have that
$$
H^\bullet_{\GO}(\Vco, \omega_{\Vco}[-2\dim \GO + 2 \dim \Lco]) \cong H^\bullet_{\Lco}(\Sp_c, \omega_{\Sp_c})
$$

\section{Finite-dimensionality}
The goal of this section is to provide a criterion on $ c \in \NK $ which will ensure that each component $ \Sp_c(\sigma) $ is finite-dimensional. 

\subsection{Semisimple groups}
We begin with the case that $ G $ is semisimple.  In this case, $\Gr_G $ has only finitely connected components, though for the moment we will just focus on the identity component $ \Gr(0) $ and the corresponding component $ \Sp_c(0)$.

\begin{Theorem} \label{th:inGrlambda}  Assume that $ G $ is semisimple.
	Let $ c \in \NK $.  Assume that the $ \GK $ orbit of $ c $ is closed in the Zariski topology on $ \NK $ (viewed as a variety over $ \cK $).  Assume also that $ \GK $ acts freely on $ c$.  Then there exists a dominant coweight $ \lambda $ such that $ \Sp_c(0) \subset \overline{\Gr^\lambda_G}$.  In particular $ \Sp_c(0) $ is a finite-dimensional projective variety.
\end{Theorem}

\begin{Lemma} \label{le:valge}
	Let $  G, c$ be as in statement of the above theorem.  For any $a \in \C[G]$, there exists $ n \in \Z $ such that if $ [g] \in \Sp_c $, then $\val a(g) \ge n$.
\end{Lemma}

\begin{proof}
	Consider the map of affine $ \cK$-varieties $ \GK \rightarrow \NK $ given by $ g \mapsto g c $, leading to the $ \cK$-algebra map $ \cK[\NK] \rightarrow \cK[\GK c] $.  Since the orbit is closed, this map is surjective. 
	
	Thus, we can find $ a' \in \cK[\NK] $ such that $ a'(g^{-1}c) = a(g) $ for all $ g \in \GK $.  Since $ \cK[\NK] = \C[N] \otimes_\C \cK $, we can write $ a' = \sum_{i=1}^k a_i \otimes p_i $, where $ p_i \in \cK $.  Let $ n = \min_i \val p_i $.
	
	Suppose that $ [g] \in \Sp_c $.  We have that $ a'(g^{-1}c) = \sum a_i(g^{-1} c) p_i $.  Since $ [g] \in \Sp_c $, we have that $ g^{-1}c \in \NO $ and so $ \val a_i(g^{-1} c) \ge 0 $.  Thus,
	$$
	\val a(g) = \val a'(g^{-1}c) = \sum \val a_i(g^{-1} c) p_i \ge \min_i p_i \ge n
	$$
\end{proof}

\begin{proof}[Proof of Theorem \ref{th:inGrlambda}]
	Let $ \lambda $ be a dominant coweight and let $ [g] \in \Gr(0)$.  We know that $ [g] \in \overline{\Gr^\lambda}$ if and only if $ \val a(g) \ge \langle w_0 \lambda, \mu \rangle $ for all dominant weights $ \mu $ of $ G$ and for all $ a \in V(\mu)^* \otimes V(\mu) $ (see for example \cite{KMW}).  In fact, it suffices to check this for a finite set $ S $ of $ \mu $ which spans the weight lattice.  Fix such a set $ S $. For each $ \mu \in S $, pick a basis for $ V(\mu)^* \otimes V(\mu) $  and apply Lemma \ref{le:valge} to the elements in this basis.  Thus we conclude that there exists $ n_\mu $ such that
	$$
	[g] \in \Sp_c \Rightarrow \val(a(g)) \ge n_\mu \text{ for all $  a \in V(\mu)^* \otimes V(\mu) $}
	$$
	Choose $ \lambda $ such that $  \langle w_0 \lambda, \mu \rangle  \le n_\mu $ for all $ \mu \in S $.  Thus, we see that if $ [g] \in \Sp_c(0) $, then $ [g] \in \overline{\Gr^\lambda}$ as desired.
\end{proof}

\subsection{Reduction to semisimple case}

We start with the following simple observation.  Let $ G $ be a reductive group and let $ G' $ be its commutator subgroup (which is semisimple).  The identity component of the affine Grassmannian of $ G' $ is isomorphic to every connected component of the affine Grassmannian of $  G$.  We will need to consider Springer fibres for both $ G$ and $ G'$, so we write $ \Sp_{c,G}$ and $ \Sp_{c,G'} $ to distinguish them.  The proof of the following lemma is immediate.

\begin{Lemma} \label{le:GtoG'}
	Let $ c \in \NK $.  Fix a coweight $ \tau $ and let $ \sigma = [\tau] \in \pi$.  
	\begin{enumerate}
		\item
		The map $ g \mapsto z^\tau g $ gives an isomorphism
		$$\Gr_{G'}(0) \cong \Gr_G(\sigma) $$
		\item
		This restricts to an isomorphism
		$$
		\Sp_{z^{-\tau} c,G'}(0) \cong \Sp_{c,G}(\sigma)
		$$
	\end{enumerate}
\end{Lemma}

Now we will combine this Lemma with our analysis of the semisimple case.

\begin{Theorem} \label{th:finitedim}
	Let $ c \in \NK $.  Assume that $ c $ is $\chi$-stable for some homomorphism $ \chi : \GK \rightarrow \cK^\times $.  Assume also that $ \GK$ acts freely on $ c $.  Then for each $ \sigma \in \pi $, $ \Sp_c(\sigma) $ is a finite-dimensional projective variety.
\end{Theorem}

\begin{proof}
	Choose any coweight $ \tau $ such that $[\tau] = \sigma $.  Since $c $ is $ \chi$-stable, the map $ \GK \rightarrow \NK \times \cK $ given by $ g \mapsto (gc, \chi(g)^{-1}) $ is proper.  In particular, the image of $ G'_\cK z^{-\tau} $ is closed in $ \NK \times \cK $.
	
	Since $ G' $ is semisimple, $ \chi $ is the identity on $ G'_\cK $ and thus the image of $ G'_\cK z^{-\tau} $ is the set $$ \{ g z^{-\tau} c : g \in G'_\cK \} \times \{ \chi(z^{\tau}) \} $$  Since this second factor is a point, we see that the $ G'_\cK $ orbit of $ z^{-\tau} c $ is closed in $ \NK$.
	
	Thus, by Theorem \ref{th:inGrlambda}, we conclude that $ \Sp_c(\sigma) $ is a finite-dimensional projective variety.
\end{proof}

\section{Modules from equivariant sheaves}

\subsection{Some sheaf theory} \label{section: some sheaf theory}

We work with equivariant derived categories of constructible sheaves on spaces such as $\sT, \sR$ and $\NO$ and $\Vco$.  These are (placid) ind-schemes of infinite type, and some care is required in defining associated categories. This is done by expressing these spaces as limits of finite-dimensional approximations, as in \cite[Section 2(ii)]{BFN1}.  Standard results about constructible sheaves on finite-dimensional spaces naturally extend to this generality. In particular, we will appeal to several compatibilities between functors from \cite[Appendix B]{AHR}.

\begin{Remark}
Very general theories of sheaf categories on ind-schemes are developed in \cite{Raskin}, \cite{BKV}, which should be flexible enough to define categories such as $D_{\GO}(\Vco)$ in full generality.  Since our proofs appeal to basic properties of sheaf functors, they should extend to these general settings.
\end{Remark}

\subsubsection{Pull-back with supports}
\label{section: pull-back with supports}
We recall the definition of pull-back homomorphisms with supports, following \cite[3(ii)]{BFN1}. Consider a Cartesian square
\begin{equation}
\label{diagram: Cartesian}
\begin{tikzcd}
X  \ar[d, "i"'] & Y \ar[l, "g"'] \ar[d, "j"] \\
M  & N \ar[l, "f"]
\end{tikzcd}
\end{equation}
Let $A, B$  be complexes of sheaves on $M,N$, respectively.  Given $\varphi : f^\ast A \rightarrow B$, we define 
\begin{equation}
\label{eq: pbwithsupports}
 i^! A \longrightarrow i^! f_\ast f^\ast A \cong g_\ast j^! f^\ast A \xrightarrow{g_\ast j^! \varphi} g_\ast j^! B
\end{equation}
where the first arrow is the unit of the adjunction $(f^\ast, f_\ast)$, and the middle isomorphism is base-change.  The morphism (\ref{eq: pbwithsupports}) is called \textbf{pull-back with supports} with respect to $ \varphi $. 

Under the $(f^\ast,f_\ast)$ adjunction, $ \varphi : f^\ast A \rightarrow  B$ corresponds to $ \tilde \varphi : A \rightarrow f_\ast B $.  Using $ \tilde \varphi$ we may rewrite the pull-back with supports in a slightly simpler way.
\begin{Lemma}
\label{lemma: alternative pullback}
The pull-back with supports homomorphism is equivalently given by the composition
$$
i^! A \xrightarrow{i^! \tilde \varphi} i^! f_\ast B \cong g_\ast j^! B,
$$
\end{Lemma}

\begin{Remark}
Suppose that $M$ is a smooth variety, and $f:N \hookrightarrow M$ is an embedding of a smooth subvariety of complex codimension $d$.  Then there is a canonical isomorphism $f^\ast \omega_M \cong \omega_N[2d]$.  The pull-back with supports with respect to this isomorphism gives us a morphism 
$$\omega_X = i^! \omega_M \rightarrow g_* j^! \omega_N[2d] = g_* \omega_Y[2d] $$
Pushing forward to a point gives $ H_\bullet(X) \rightarrow H_{\bullet -2d}(Y) $ which corresponds to the usual restriction with supports in Borel-Moore homology, see \cite[\S 8.3.21]{CG}.
\end{Remark}


\subsubsection{Composition} \label{section: composition of pull-back}
Consider the situation of two vertically stacked Cartesian squares:
\begin{equation}
\label{diagram: vertically stacked Cartesian}
\begin{tikzcd}
Z  \ar[d, "k"'] & W \ar[d, "\ell"] \ar[l, "h"']\\
X \ar[d, "i"'] & Y \ar[d, "j"]  \ar[l, "g"'] \\
M  & N \ar[l, "f"]
\end{tikzcd}
\end{equation}
Given $A, B, \varphi$ as above, we can produce elements of
$$\Hom( k^! i^! A, h_\ast \ell^! j^! B) \cong \Hom\big((i \circ k)^! A, h_\ast (j \circ \ell)^! B\big)$$
by iterating the pull-back with supports construction twice, or all at once, respectively.  These two homomorphisms are identified, as is encoded by the commutative diagram
$$
\begin{tikzcd}
k^! i^! A \ar[r,"k^! i^! \tilde{\varphi}"] \ar[d,"\sim"] & k^! i^! f_\ast B \ar[r, "\sim"] \ar[d,"\sim"] & k^! g_\ast j^! B \ar[r,"\sim"] & h_\ast \ell^! j^! B \ar[dl,"\sim"]\\
(i \circ k)^! A \ar[r,"(i\circ k)^! \tilde{\varphi}"]  & (i\circ k)^! f_\ast B  \ar[r,"\sim"] & h_\ast (j\circ \ell)^! B & 
\end{tikzcd}
$$
The rows correspond to the respective pull-back with supports maps (using Lemma \ref{lemma: alternative pullback}), and the commutativity of the right pentagon corresponds to the composition property for base change as in \cite[Figure B.7(c)]{AHR}.

\begin{Remark}
Given horizontally stacked Cartesian squares, we can also produce morphisms by using pull-back with supports in two different ways.  These again agree, this time by  the composition for base change encoded by \cite[Figure B.7 (d)]{AHR}.
\end{Remark}

\subsubsection{Proper direct image}
\label{section: proper direct image}
For a proper morphism $i: X \rightarrow M$ and a sheaf $A$ on $M$, there is a natural map on hypercohomology:
$$ H^\bullet(X, i^! A) = H^\bullet( M, i_\ast i^! A) = H^\bullet(M, i_! i^! A)  \longrightarrow H^\bullet( M, A), $$
induced by the counit of the adjunction $(i_!, i^!)$.  We call this morphism the \textbf{proper direct image}.

Consider the situation where the vertical arrows $i, j$ in (\ref{diagram: Cartesian}) are proper, and we are given $\varphi: f^\ast A \rightarrow B$ as per usual.  Then we have a diagram in hypercohomology:
\begin{equation}
\label{diagram: proper base change hypercohomology}
\begin{tikzcd}
H^\bullet(X, i^!A) \ar[d] \ar[r] & H^\bullet(Y, j^! B)  \ar[d] \\
H^\bullet(M, A) \ar[r] & H^\bullet(N, B)
\end{tikzcd}
\end{equation}
The bottom map is induced by $\tilde \varphi$, the top by the pull-back with supports with respect to $\varphi$, and the vertical arrows are proper direct image.

\begin{Lemma}
\label{lemma: proper base change hypercohomology}
The diagram (\ref{diagram: proper base change hypercohomology}) commutes.
\end{Lemma}
\begin{Remark}
In the case when $A = \omega_X$, the proper direct image corresponds to the proper direct image map in Borel-Moore homology \cite[\S 8.3.19]{CG}.  Moreover the previous lemma generalizes the smooth base change in Borel-Moore homology \cite[Prop 2.7.22]{CG}.
\end{Remark}

\subsection{Definition of the action} \label{se:DefAction} 

Because the definition of the quantized Coulomb branch involves $ \sR_{G,N} $ (and not $ \sZ_{G,N} $), we will define our modules using a diagram involving $ \NO $.   We closely follow \cite[Sections 3(i)-(iii)]{BFN1}.

We begin by considering the groupoid given by restricting the action of $ \GK  $ on $ \NK $ to $\NO$. We define
$$
\groupoid = \{ (g,v) \in \GK \times \NO : gv \in \NO \} 
$$
This is a groupoid over $ \NO $ with source and taget maps $ d_1, d_2 : \groupoid \rightarrow \NO $ defined by $ d_1(g,v) = v, d_2(g,v) = gv $.  We let $ \groupoid^{(2)} $ denote the scheme of composable arrows in this groupoid.  Explicitly, we have
$$
\groupoid^{(2)} = \{ (g_1, g_2, v) \in \GK \times \GK \times \NO : g_2 v, g_1 g_2 v \in \NO \}
$$
This comes with three projections $ d_{12}, d_{13}, d_{23} : \groupoid^{(2)} \rightarrow \groupoid $ given by
\begin{equation} \label{eq:dij}
d_{12}(g_1, g_2, v) = (g_2, v) \quad d_{13}(g_1, g_2, v) = (g_1 g_2, v) \quad d_{23}(g_1, g_2, v) = (g_1, g_2 v) \quad
\end{equation}
We also consider a variation on the unit of this groupoid: the inclusion $u: \GO \times \NO \hookrightarrow \groupoid$ defined by $(g,v) \mapsto (g,v)$.  We denote $u_1,u_2: \GO \times \NO\rightarrow \NO$ the maps $u_1(g,v) = v$, $u_2(g,v) =gv$.
\begin{Remark}
\label{rmk: canonical unit iso}
For any object $\sF$ of the $\GO$--equivariant derived category of $\NO$, there is a canonical isomorphism $u_1^! \sF \cong u_2^! \sF$.
\end{Remark}

\begin{Definition}
Let $ \mathcal F $ be an object of the $ \GO $-equivariant derived category of $ \NO$.  We say that $ \sF $ is $ \groupoid $-\textbf{equivariant}, if we are given an isomorphism $ \beta: d_1^! \sF \rightarrow d_2^! \sF $ (as objects in the $\GO \times \GO$--equivariant derived category of $\groupoid$) such that the following diagram commutes
\begin{equation} \label{eq:hexagon}
\begin{tikzcd}[column sep={1.3cm,between origins}, row sep={1.5cm,between origins}]
& d_{13}^! d_1^! \sF \arrow[rr, "d_{13}^!\beta"] \arrow[ld, equal] && d_{13}^! d_2^! \arrow[rd, equal] &  \\
d_{12}^! d_1^! \sF \arrow[rd, "d_{12}^!\beta"']&  &&  & d_{23}^! d_2^! \sF  \\
& d_{12}^! d_2^! \sF \arrow[rr, equal] && d_{23}^! d_1^! \sF \arrow[ru, "d_{23}^!\beta"'] & 
\end{tikzcd}
\end{equation}
(as objects in the $\GO\times \GO\times\GO$--equivariant derived category of $\groupoid^{(2)}$).  We also require that the diagram
\begin{equation}\label{eq: unit axiom}
\begin{tikzcd}
u^! d_1^! \sF \ar[r,"u^!\beta"] \ar[d,equal] & u^! d_2^! \sF \ar[d,equal] \\
u_1^! \sF \ar[r] & u_2^! \sF
\end{tikzcd}
\end{equation}
commutes, where the bottom arrow is the canonical isomorphism from Remark \ref{rmk: canonical unit iso}.
\end{Definition}

Our primary example of such an $ \sF $ is the sheaf $ \sF =  \omega_{\Vco}[-2\dim \GO + 2 \dim \Lco ] $, which we obtained from the dualizing sheaf from $ F_c $. This sheaf has a natural $\groupoid$--equivariant structure: for $i=1,2$ we have a Cartesian diagram
\begin{equation}
\begin{tikzcd}
\left\{ (g, v) \in \GK \times \Vco : gv \in \Vco \right\} \ar[hook]{r} \ar[d] & \groupoid \ar[d,"d_i"] \\
\Vco \ar[hook]{r} & \NO
\end{tikzcd}
\end{equation}
By base change, $d_i^! \cF$ is given by the $\ast$--pushforward of the (shifted) dualizing sheaf along the top inclusion.  The latter is independent of $i$, and thus gives an isomorphism $d_1^! \cF \cong d_2^!\cF$.

\begin{Theorem} \label{th:action}
If $ \mathcal F $ is a $ \groupoid$-equivariant object of the $ \GO$-equivariant derived category of $ \NO$, then $ H^{-\bullet}_{\GO}(\NO, \sF) $ carries a left module structure for the algebra $ \A_0 $, which is $H_G^\bullet(pt)$--linear in the first variable.

In particular, when $ \Sp_c $ and $ \Lco $ are of finite type, then
$$ M_c := H^{\Lco}_\bullet(\Sp_c) = H_{\GO}^{-\bullet}(\Vco, \omega_{\Vco}[-2 \dim \GO + 2 \dim \Lco]) $$
is an $\A_0$--module.
\end{Theorem}

We will now define the action of $ \A_0 $ on $ H^{-\bullet}_{\GO}(\NO, \sF)  $. To that end, consider the following commutative diagram:

\begin{equation}
\label{eq: diagram for action1}
\begin{tikzcd}
\sR \times \NO \arrow[hook]{d}{i} & \groupoid \arrow{l}[swap]{p}\arrow{r}{q} \arrow[hook]{d}{j}& \sR \arrow{r}{m}  & \NO \\
\sT \times \NO \arrow[leftarrow]{r}{\tilde p}  & \GK \times \NO & &
\end{tikzcd}
\end{equation}
The map $\tilde p$ is given by $(g,v) \mapsto ( [g,v], v)$, and $m$ by $[g, v] \mapsto gv$. The vertical arrows are the inclusion maps, and $p$ is defined by restriction.  Finally $q$ is the quotient map $(g,v) \mapsto [g,v]$, which realizes $\groupoid$ as a principal $\GO$--bundle over $\sR$.  These maps are equivariant for group actions as in \cite[(3.4)]{BFN1}; in particular $\GO \times \GO$ acts on $\GK \times \NO$ by $(h_1, h_2) \cdot (g, v) = (h_1 g h_2^{-1}, h_2 v)$.

\begin{Lemma}
\label{Lemma: morphism to put into restriction with supports}
	There is an isomorphism of $\GO \times \GO $--equivariant complexes
	\begin{equation} \label{eq:defphi}
	\tilde p^\ast \big(\omega_{\sT}[-2\dim \NO ] \boxtimes \sF \big) \xrightarrow{\sim} \omega_{\GK}[- 2\dim \GO] \boxtimes \sF 
	\end{equation}
\end{Lemma}
\begin{proof}
As in the proof of \cite[Lemma 3.5]{BFN1}, we write $\tilde p = (\tilde p_{\sT}, \tilde p_{\NO})$ in terms of its components.  Since $\tilde p_\NO$ is a projection, we have
$$
\tilde p_\NO^\ast \sF \cong \C_\GK \boxtimes \cF
$$
Meanwhile $\tilde p_\sT:\GK \times \NO \longrightarrow \sT$ is a $\GO$--bundle, so 
$$
\tilde p_\sT^\ast (\omega_\sT) \cong \tilde p_\sT^!(\omega_\sT)[- 2 \dim \GO] \cong \omega_\GK \boxtimes \omega_\NO [-2\dim \GO] 
$$
But $\NO$ is smooth, so $\omega_\NO \cong \C_\NO[2 \dim \NO]$.  
\end{proof}

Using the pull-back with supports homomorphism from Section \ref{section: pull-back with supports} with respect to (\ref{eq:defphi}), we get a morphism
\begin{equation} \begin{gathered} \label{eq:restrictionF}
\omega_{\sR}[-2\dim \NO] \boxtimes \sF=  i^! \big(\omega_{\sT}[-2\dim \NO] \boxtimes \sF\big)\\ \longrightarrow p_* j^! \big(\omega_{\GK} [-2 \dim \GO ]\boxtimes \sF\big) = p_\ast d_1^! \sF[-2\dim \GO]
\end{gathered}
\end{equation}
where the last equality comes from the fact that $ d_1 $ equals $ j $ composed with the projection $ \GK \times \NO \rightarrow \NO $.

Next, we use the $\groupoid$-equivariant structure on our sheaf $ \sF $ to obtain
\begin{equation} \label{eq:restrictionF2}
d_1^! \sF[-2\dim \GO] \xrightarrow{\beta} d_2^! \sF[-2\dim \GO] = q^! m^! \sF[-2\dim \GO]
\end{equation}
where we use that $ d_2 = m \circ q $.

Next, we note that the map $ q $ is a principal $ \GO  $-bundle and thus, $  q^! =  q^*[2 \dim \GO] $ and so 
\begin{equation}
H^\bullet_{\GO \times \GO}(\groupoid,  q^!  m^! \sF[-2\dim \GO]) \cong  H^\bullet_{\GO}(\sR, m^! \sF) \label{eq:actionGO}
\end{equation}
Finally, we have the proper direct image for the morphism $ m $ which gives us
\begin{equation} \label{eq:actionproper}
H^\bullet_\GO(\sR, m^! \sF) \rightarrow H^\bullet_\GO(\NO, \sF)
\end{equation}

Combining all these steps together (applying push forward to a point when needed), and remembering that $\A_0 = H^{-\bullet}_\GO(\sR, \omega_\sR[-2\dim \NO] )$, we finally obtain
$$
\A_0 \otimes  H^{-\bullet}_{\GO}(\NO, \sF)  \longrightarrow  H^{-\bullet}_{\GO}(\NO, \sF) 
$$
as desired.

\begin{proof}[Proof of Theorem \ref{th:action}]
To see that this defines a module structure, we closely follow the proof of \cite[Theorem 3.10]{BFN1}, with some minor changes.  

To begin, recall that the action was defined as a composition of four steps given in equations (\ref{eq:restrictionF}) -- (\ref{eq:actionproper}).  We will write these steps in a condensed form as follows
\begin{equation} \label{eq:actionCondensed}
H(\omega_\sR \boxtimes \sF) \xrightarrow{p^*} H(d_1^!\sF) \xrightarrow{\beta} H(d_2^! \sF) \cong H(m^! \sF) \xrightarrow{m_*} H(\sF)
\end{equation}
We have omitted the equivariance groups, we have left out the spaces on which we take these cohomologies, and we have left out the cohomological degrees and shifts.  We have also named the pull-back with supports $ p^*$ and the proper direct image $ m_* $, which will be helpful for keeping track of these later. 

First, we verify that the identity element $r_0 \in \A_0$ acts as the identity, where $r_0$ is the fundamental class of the fibre of $\sR \rightarrow \Gr$ over $[1]$. Consider the commutative diagram
\begin{equation*}
\begin{tikzcd}
\NO \times \NO \ar[d, hookrightarrow] & \GO \times \NO \ar[l,"p_0"'] \ar[d, hookrightarrow] \ar[r,"u_2"] & \NO \ar[d, hookrightarrow] \ar[dr,equal] & \\
\sR \times \NO & \ar[l,"{p}"] \groupoid \ar[r, "{q}"'] & \sR \ar[r,"{m}"'] & \NO
\end{tikzcd}
\end{equation*}
Here $p_0 = (u_2, u_1)$ sends $(g,v) \mapsto(gv, v)$. The vertical arrows are closed embeddings; note that the leftmost vertical arrow uses the map $\NO \hookrightarrow \sR, v \mapsto [1,v]$ which is the inclusion of the fibre over $[1]$. The middle rightward arrows are both quotients by $\GO$ actions.   We use this diagram to add another row to (\ref{eq:actionCondensed}):
\begin{equation}
\label{eq: unit diagram in proof}
\begin{tikzcd}[column sep = scriptsize]
H(\omega_\NO \boxtimes \sF) \ar[r,"p_0^*"] \ar[d] & H(u_1^!\sF) \ar[r,"\sim"] \ar[d] & H(u_2^!\sF)  \ar[d]\ar[rr,"\sim"]& & H(\sF) \ar[d,equal] \\
H(\omega_\sR \boxtimes \sF) \ar[r,"p^*"] & H(d_1^!\sF)\ar[r,"\beta"] & H(d_2^!\sF) \ar[r,"\sim"] &H(m^! \sF) \ar[r,"m_*"] & H(\sF)
\end{tikzcd}
\end{equation}
The vertical arrows are proper pushforwards along the vertical arrows in (\ref{eq: unit diagram in proof}).  The middle square uses the compatibility (\ref{eq: unit axiom}).  In the leftmost square, we define the top edge by pull-back with supports, so that this square commutes by Lemma \ref{lemma: proper base change hypercohomology}.  Using that $\omega_\NO[-2\dim\NO] = \C_\NO$, this pull-back with supports morphism is simply the adjoint of the morphism
$$
p_0^\ast (\C_\NO \boxtimes \sF ) = u_2^\ast( \C_\NO) \otimes u_1^\ast( \sF) = u_1^\ast(\sF) 
$$
Thus the composed top row of (\ref{eq: unit diagram in proof}) is the canonical action $H(\C_\NO \boxtimes \sF) \rightarrow H(\sF)$.  For any $b\in H(\sF)$ the class $r_0 \otimes b$ comes by pushforward along the left vertical arrow in (\ref{eq: unit diagram in proof}), and so we can compute the product $r_0 \ast b$ by tracing the effect through the diagram.  Combined with the above discussion, this proves that $r_0$ acts as the identity.

Second, we will show that $(a_1 \ast a_2) \ast b = a_1 \ast( a_2 \ast b)$ for $a_1, a_2 \in \A_0$ and $b\in H(\sF)$.  Recall the definition of the multiplication from \cite{BFN1}, which makes use of the following diagram \cite[(3.2)]{BFN1} (except that we have swapped the roles of $ \tilde p $ and $ p$, etc).
\begin{equation}
\label{eq: diagram from BFN}
\begin{tikzcd}
\sR \times \sR \arrow[hook]{d}{i} & \tilde p^{-1}(\sR \times \sR) \arrow{l}[swap]{p}\arrow{r}{q} \arrow[hook]{d}{j}& \tilde p^{-1}(\sR \times \sR) /{ \scriptstyle \GO} \arrow{r}{m}  & \sR\\
\sT \times \sR \arrow[leftarrow]{r}{\tilde p}  & \GK \times \sR & &
\end{tikzcd}
\end{equation}

The multiplication in $ \A_0 $ is given by the following sequence
\begin{equation} \label{eq:multCondensed}
H(\omega_\sR \boxtimes \omega_\sR) \xrightarrow{p^*} H(\omega_{\tilde p^{-1}(\sR \times \sR)}) \cong H(\omega_{q(\tilde p^{-1}(\sR \times \sR))}) \xrightarrow{m_*} H(\omega_\sR) 
\end{equation}
where the first arrow is a pull back with supports with respect to an isomorphism $$ \tilde p^*(\omega_\sT [-2\dim \NO] \boxtimes \omega_\sR) \rightarrow \omega_\GK[-2\dim \GO] \boxtimes \omega_\sR) $$
the second arrow is the canonical isomorphism from a $ \GO$-bundle and the last is the proper direct image.

Now we consider the following large  commutative diagram, analogous to \cite[(3.11)]{BFN1}.  It is the ``product'' of the top line of (\ref{eq: diagram for action1}) with the top line of (\ref{eq: diagram from BFN}):
\begin{equation}
\label{eq: large diagram}
\begin{tikzcd}[column sep = scriptsize]
\sR \times \NO & \groupoid  \ar{l}[swap]{p} \ar{r}{q}&\sR \ar{r}{m} & \NO \\
\tilde p^{-1} (\sR \times \sR) /{\scriptstyle \GO} \times \NO  \ar{u}{m\times \id_{\NO}} & \groupoid^{(2)}/{ \scriptstyle \GO \times 1} \ar{l} \ar{u}[swap]{d_{13}} \ar{r} & \groupoid^{(2)}/{ \scriptstyle \GO \times \GO} \ar{u} \ar{r} & \sR \ar{u}[swap]{m} \\
\tilde p^{-1} (\sR \times \sR) \times \NO \ar{u}{q \times \id_{\NO}} \ar{d}[swap]{p\times\id_{\NO}} & \groupoid^{(2)} \ar{d}{p_2} \ar{l}{p_1} \ar{u} \ar{r} & \groupoid^{(2)}/{ \scriptstyle 1\times \GO} \ar{u} \ar{r}{d_{23}} \ar{d} & \groupoid \ar{d}{p} \ar{u}[swap]{q} \\
\sR \times \sR \times \NO & \sR \times \groupoid \arrow{l}{\id_\sR\times p} \arrow{r}[swap]{\id_\sR\times q} &\sR \times \sR \arrow{r}[swap]{\id_\sR\times m} & \sR \times \NO
\end{tikzcd}
\end{equation}
where the maps $p, q, m$ denote either those from (\ref{eq: diagram for action1}) or from (\ref{eq: diagram from BFN}).  The maps $ d_{ij} $ were defined in (\ref{eq:dij}). Finally, the map $ p_1: \groupoid^{(2)} \rightarrow \tilde p^{-1}(\sR \times \sR) \times \NO $ is given by $p_1(g_1, g_2, v) = (g_1, [g_2, v], v)$ and $ p_2 : \groupoid^{(2)} \rightarrow \sR \times \groupoid $ is given by $ p_2(g_1, g_2, v) = ( [g_1, g_2 v], g_2, v) $.

Now, we use (\ref{eq: large diagram}) to construct an even larger diagram of maps of cohomology of sheaves
\begin{equation}
\begin{tikzcd}[column sep =scriptsize, row sep =scriptsize] \label{eq: even larger diagram}
H(\omega_\sR \boxtimes \sF) \arrow[r,"p^*"] & H(d_1^! \sF) \arrow[r,"\beta"] & H(d_2^! \sF) \arrow[r,"m_*"] & H(\sF) \\
& H(d_{13}^! d_1^! \sF) \arrow[u, "(d_{13})_*"'] \arrow[r,"\beta"] & H(d_{13}^! d_2^! \sF) \arrow[d,equal] \arrow[u, "(d_{13})_*"']  &  \\
& & H(d_{23}^! d_2^! \sF)  \arrow[r,"(d_{23})_*"]  & H(d_2^! \sF) \arrow[uu,"m_*"'] \\
 & & H(d_{23}^! d_1^! \sF) \arrow[u,"\beta"] \arrow[r,"(d_{23})_*"] & H(d_1^! \sF) \arrow[u,"\beta"'] \\
H(\omega_{p^{-1}(\sR \times \sR)} \boxtimes \sF) \arrow[uuuu,"m_* \otimes \id"] \arrow[r, "p_1^*"] & H(d_{12}^! d_1^! \sF) \arrow[uuu,equal] \arrow[r,"\beta"] & H(d_{12}^! d_2^! \sF) \arrow[u,equal] &\\
H(\omega_\sR \boxtimes \omega_\sR \boxtimes \sF) \arrow[u,"p^*"] \arrow[r,"\id \otimes p^*"] & H(\omega_\sR \boxtimes d_1^! \sF) \arrow[u,"p_2^*"] \arrow[r,"\id \otimes \beta"] & H(\omega_\sR \boxtimes d_2^! \sF) \arrow[u,"p_2^*"] \arrow[r,"\id \otimes m_*"]  & H(\omega_\sR \boxtimes \sF) \arrow[uu,"p^*"']
\end{tikzcd}
\end{equation}
We are using the same conventions as above regarding ommitting equivariance groups, cohomological degrees and shifts, and spaces.  We have also ommitted the canonical isomorphisms related to the principal $ \GO $-bundles, in order to make the diagram simpler.  On the other hand, the spaces $ \groupoid, \sR \times \groupoid $ each have two cohomology groups attached to them, while $ \groupoid^{(2)} $ has six cohomology groups (the ones in the middle of the diagram).

The morphisms in this diagram are labelled in our usual way, except that we have not defined the pull-back with supports for $ p_1, p_2 $ --- we will do so shortly.

Notice that following the boundary of (\ref{eq: even larger diagram}) up, and then right, from $ H(\omega_\sR \boxtimes \omega_\sR \boxtimes \sF) $ to $ H(\sF) $ is the definition of $ (a_1 * a_2) * b $ (see  (\ref{eq:actionCondensed}) and (\ref{eq:multCondensed})).  On the other hand, following the boundary right, and then up, is the definition of $a_1 * (a_2 *b) $.  Thus, it suffices to prove the commutativity of each face of (\ref{eq: even larger diagram}).

\noindent \textbf{Bottom left square.} First, we consider the bottom left square.  This square involves two pairs of pull-back with supports homomorphisms.  It extends to a commutative cube
\begin{equation}
\begin{tikzcd}[column sep = scriptsize, row sep = scriptsize]
& \GK \times \sR \times \NO \ar[dd, "\tilde p\times \id_{\NO}" near start] & & \GK \times \groupoid \ar[ll, "\id_{\GK}\times p"']  \ar[dd, "\tilde p_2"] \\
p^{-1}(\sR\times \sR) \times \NO \ar[ur] \ar[dd, "p \times \id_\NO"] & & \groupoid^{(2)} \ar[ll, crossing over,"p_1" near start]  \ar[ur] & \\
& \sT \times \sR \times \NO & & \sT \times \groupoid \ar[ll, "\id_{\sT}\times p" near start] \\
\sR \times \sR \times \NO \ar[ur] & & \sR\times \groupoid \ar[ll,"\id_\sR \times p"] \ar[ur] \ar[from = uu, crossing over, "p_2" near start] &
\end{tikzcd}
\end{equation}
Arrows from the front square to the back are closed embeddings.  The map $\tilde p_2$ is analogous to the corresponding map $ p_2 $ in the front square, and sends $(g_1, g_2, v) \mapsto ( [g_1, g_2 v], g_2, v)$.  

We have the following morphisms corresponding to the arrows in the back square
\begin{subequations}
\begin{gather}
(\tilde p \times \id_{\NO})^* A \longrightarrow \omega_\GK[-2\dim \GO] \boxtimes \omega_\sR[-2\dim \NO] \boxtimes \sF \label{eq:sq1} \\
(\id_\GK \times p)^* (\omega_\GK[-2\dim \GO] \boxtimes \omega_\sR[-2\dim \NO] \boxtimes \sF)  \longrightarrow B\label{eq:sq2} \\
(\id_\sT \times p)^* A \longrightarrow \omega_\sT[-2\dim \NO] \boxtimes d_1^! \sF[-2\dim \GO] \label{eq:sq3} \\
\tilde p_2^* (\omega_\sT[-2\dim \NO] \boxtimes d_1^! \sF[-2\dim \GO]) \longrightarrow B \label{eq:sq4}
\end{gather}
\end{subequations}
where \begin{align*} A &= \omega_\sT[-2\dim \NO] \boxtimes \omega_\sR[-2\dim \NO] \boxtimes \sF, \\ B &= \omega_\GK[-2\dim \GO] \boxtimes d_1^! \sF[-2\dim \GO] 
\end{align*}

The morphism (\ref{eq:sq1}) comes from \cite[Lemma 3.5]{BFN1}, the morphisms (\ref{eq:sq2}) and (\ref{eq:sq3}) come from (\ref{eq:restrictionF}), and the morphism (\ref{eq:sq4}) is an isomorphism similar to \cite[Lemma 3.5]{BFN1}.  It is easy to see that the composition of first two morphisms equals the composition of the second two morphisms; both compositions lie in $ \Hom(t^* A, B)$,
where $ t = (\id_\GK \times p) \circ (\tilde p \times \id_\NO) = \tilde p_2 \circ (\id_\sT \times p) $ is the diagonal map across this square.

Since the top, right, left, and bottom faces of the cube are Cartesian, the pull-back with supports with respect to these four morphisms gives the four morphisms in the bottom left square of (\ref{eq: even larger diagram}) --- in particular, we use pull-back with supports with respect to (\ref{eq:sq2}) to define $ p_1^*$ and we use pull-back with supports with respect to (\ref{eq:sq4}) to define $ p_2^*$.  Thus the above equality of compositions in the back square implies the commutativity of the bottom left square of (\ref{eq: even larger diagram}).

\noindent \textbf{Bottom central square.} Now we consider the square
\begin{equation} \label{eq:BottomCentral}
\begin{tikzcd}
H(\groupoid^{(2)}, d_{12}^! d_1^! \sF [-2\dim \NO]) \arrow[r,"\beta"] & H(\groupoid^{(2)}, d_{12}^! d_2^! \sF[-2 \dim \NO]) \\
H(\sR \times \groupoid, \omega_\sR \boxtimes d_1^! \sF) \arrow[r, "\id \otimes \beta"] \arrow[u, "p_2^*"] & H(\sR \times \groupoid, \omega_\sR \boxtimes d_2^! \sF) \arrow[u, "p_2^*"]
\end{tikzcd}
\end{equation}
(We are still simplifying by leaving off equivariance groups and cohomological shifts.)

Recall that the left vertical $ p_2^* $ was defined using pull-back with supports with respect to the map (\ref{eq:sq4})
$$
\tilde p_2^* (\omega_\sT[-2\dim \NO] \boxtimes d_1^! \sF) \longrightarrow \omega_\GK[-2\dim \GO] \boxtimes d_1^! \sF
$$
The right vertical $ p_2^* $ is defined using the pull-back with supports with respect to a similar map
\begin{equation} \label{eq:p2star}
\tilde p_2^* (\omega_\sT[-2\dim \NO] \boxtimes d_2^! \sF) \longrightarrow \omega_\GK[-2\dim \GO] \boxtimes d_2^! \sF
\end{equation}
Moreover these two maps fit into a commutative diagram using $ \beta : d_1^! \sF \rightarrow d_2^! \sF $.  Hence the naturality of pull-back with supports gives us the commutative diagram (\ref{eq:BottomCentral}).

\noindent \textbf{Bottom right pentagon.}  Now we consider the bottom right pentagon, except that we will quotient its left column by the action of $ \GO $. We must prove the commutativity of the following square.
\begin{equation} \label{eq:BottomRight}
\begin{tikzcd}
H(\groupoid^{(2)}/ {\scriptstyle 1 \times \GO}, d_{23}^! d_1^! \sF) \arrow[r, "(d_{23})_*"] & H(\groupoid, d_1^! \sF) \\
H(\sR \times \sR, \omega_\sR \boxtimes m^! \sF) \arrow[r, "m_*"] \arrow[u, "p_2^*"] & H(\sR \times \NO, \omega_\sR \boxtimes \sF) \arrow[u, "p^*"]
\end{tikzcd}
\end{equation}
where the pull-back with supports $p_2^* $ was defined above with respect to (\ref{eq:p2star}).  Here we are using that $ d_{23}^! d_1^! \sF = d_{12}^! d_2^! \sF $.

We would like to apply Lemma \ref{lemma: proper base change hypercohomology}, but we first need to relate the two pull-backs with supports in this diagram.

For this reason, we complete the bottom right square of (\ref{eq: large diagram}) to the following commutative cube.
\begin{equation*}
\begin{tikzcd}[column sep = scriptsize, row sep = scriptsize]
 & \GK \times \sR \ar[dd, "\tilde p_2" near start] \ar[rr, "\id_{\GK} \times m"] & & \GK \times \NO \ar[dd, "\tilde p"] \\
\groupoid^{(2)}/1 \times \GO \ar[ur] \ar[dd,"p_2"] \ar[rr, crossing over, "d_{23}" near end] & & \groupoid \ar[ur]  & \\
 & \sT \times \sR \ar[rr, "\id_\sT \times m" near end] & & \sT \times \NO \\
\sR \times \sR \ar[rr, "\id_\sR \times m"] \ar[ur] & & \sR \times \NO \ar[from = uu, crossing over, "p" near start] \ar[ur] &
\end{tikzcd}
\end{equation*}
where we have slightly abused notation by using the same letters to denote maps after the reduction by $ \GO $.

 The existence of this cube means that we can apply Section \ref{section: composition of pull-back} to conclude that $ p_2^* $ is also the pull-back with supports with respect to the morphism
$$
p^* \omega_\sR[-2\dim \NO] \boxtimes \sF \rightarrow d_1^! \sF[-2\dim \GO]
$$
corresponding to the right vertical edge in (\ref{eq:BottomRight}).

So we are in a position to apply Lemma \ref{lemma: proper base change hypercohomology} to conclude that (\ref{eq:BottomRight}) commutes.

\noindent \textbf{Remaining faces.}
The top left square commutes analogously to the bottom right square.  The middle hexagon (which looks like a rectangle) commutes by the commutativity of (\ref{eq:hexagon}).  The other two squares that involve two parallel $ \beta$s commute analogously to the bottom central square.  Finally, the top right square is just proper direct image for the map $ d_2 \circ d_{23} = d_2 \circ d_{13} : \groupoid^{(2)} \rightarrow \NO $ factored in two different ways and so it commutes.

\end{proof}

\subsection{Auxillary action diagram} \label{se:Auxillary}
Now, we consider the case where our sheaf $ \cF $ comes from another space $ Z$.  More precisely, assume that we have the following data
\begin{enumerate}
	\item a space $ Z $ with an action of $ \GO $ and an equivariant morphism $ r : Z \rightarrow \NO $,
	\item an action of the groupoid $ \groupoid $ on $ Z$, compatible with $ r $ and equivariant for the action of $ \GO $,
	\item a $ \GO$-equivariant sheaf $\cF$ on $ Z$,
	\item a $ \groupoid$-equivariant structure on $ \cF $.
\end{enumerate}
Let us explain these conditions more precisely.  The action of $ \groupoid $ on $ Z $ is a map
$$
c_2 : \groupoid \times_\NO Z = \{(g, v, z) \in \groupoid \times Z : r(z) = v \} \rightarrow Z
$$
such that $r(c_2(g,v,z)) = d_2(g,v) = gv $.  We require that $ c_2 $ is invariant for the ``diagonal'' action of $ \GO $ given by $ h \cdot (g, v, z) = (gh^{-1}, hv, hz) $ and equivariant for the ``left'' action given by $ h \cdot (g, v, z) = (hg, v, z) $.

We also have the projection map $ c_1 : \groupoid \times_\NO Z \rightarrow Z $ given by $ c_1(g,v,z) = z $.  Note that the diagrams 
\begin{equation} \label{eq:ZPsquare}
\begin{tikzcd}
\groupoid \times_\NO Z \arrow[r,"c_i"] \arrow[d,"r^{(1)}"] & Z \arrow[d,"r"] \\
\groupoid \arrow[r,"d_i"] & \NO 
\end{tikzcd}
\end{equation}
commute and are Cartesian for $ i = 1,2$, where $ r^{(1)} $ is projection.

Also, consider 
$$
\groupoid^{(2)} \times_\NO Z = \{(g_1, g_2, v, z) \in \groupoid^{(2)} \times Z : r(z) = v \} 
$$
which comes with three maps $ c_{ij} : \groupoid^{(2)} \times_\NO Z \rightarrow \groupoid \times_\NO Z  $ given by
\begin{gather*}
c_{12}(g_1, g_2, v,z) = (g_2, v, z), \ c_{13}(g_1, g_2, v, z) =  (g_1 g_2, v, z), \\ 
c_{23}(g_1, g_2, v,z) = (g_1, g_2v, c_2(g_2, v, z))
\end{gather*}
There is also a unit morphism $ \GO \times Z \hookrightarrow \groupoid \times_\NO Z$ defined by $(g,z) \mapsto (g,r(z), z)$.

A $\groupoid$-\textbf{equivariant structure} on $ \cF $ is an isomorphism $ \beta : c_1^! \sF \rightarrow c_2^! \sF $ (in the $ \GO \times \GO$-equivariant derived category of $ \groupoid \times_\NO Z $), such that the analog of the hexagon (\ref{eq:hexagon}), with $ d_{ij}, d_k $ replaced by $ c_{ij}, c_k $, commutes. We also require that the obvious analog of the unit axiom (\ref{eq: unit axiom}) holds.

For the remainder of this section, let us fix all this data.

\begin{Lemma}
\label{lem: pushforward sheaf is equivariant}
	$r_* \cF $ is a $\groupoid$-equivariant sheaf on $ \NO $, with equivariant structure given by $ r^{(1)}_* \beta$.
\end{Lemma}

\begin{proof}
	By base change for the Cartesian squares (\ref{eq:ZPsquare}), we see that   $ d_i^! r_* \cF \cong r^{(1)}_* c_i^! \cF$.  Thus $ r^{(1)}_* \beta $ gives the desired isomorphism $ d_1^! r_* \sF \rightarrow d_2^! r_* \sF $.  
	
	To deduce the commutativity of the hexagon (\ref{eq:hexagon}), we begin with the corresponding hexagon containing $ c_{ij}, c_k $.  Then we apply $r^{(2)}_* $, where $r^{(2)}$ is the projection $ r^{(2)} : \groupoid^{(2)} \times_\NO Z \rightarrow \groupoid^{(2)} $.  Finally we use base change as above.
	
	The proof of the compatibility for unit axiom (\ref{eq: unit axiom}) is similar.
\end{proof}

Now consider the following modification of (\ref{eq: diagram for action1})
\begin{equation}
\label{eq: diagram for actionZ}
\begin{tikzcd}
\sR \times Z \arrow[hook]{d}{i_Z} & \groupoid \times_\NO Z \arrow{l}[swap]{p_Z}\arrow{r} \arrow[hook]{d}{j_Z}& \groupoid \times_\NO Z / \GO \arrow{r}{m_Z}  & Z\\
\sT \times Z \arrow[leftarrow]{r}{\tilde p_Z}  & \GK \times Z & &
\end{tikzcd}
\end{equation}
Here  $ \tilde p_Z (g,z) = ([g,r(z)], z)$ and $ p_Z $ is defined by restriction. The map $ m_Z $ is the result of descending $ c_2 $ by the diagonal action of $ \GO $.

\begin{Lemma} \label{Lemma: morphism to put into restriction with supports2}
	We have an isomorphism
\begin{equation}
\label{extended morphism 0}
 \tilde p_Z^*(\omega_\sT[-2\dim \NO] \boxtimes \sF) \xrightarrow{\sim} \omega_{\GK}[-2\dim \GO] \boxtimes \sF
\end{equation} which after applying $ (\id_\GK,r)_* $ fits into the bottom row of the following commutative diagram 
\begin{equation*}
\begin{tikzcd}
\tilde p^*(\omega_\sT[-2 \dim \NO] \boxtimes r_*\sF) \arrow{r} \arrow{d} & \omega_{\GK}[-2\dim \GO] \boxtimes r_* \sF \arrow[equal]{d} \\
(\id_\GK, r)_* \tilde p_Z^*(\omega_\sT[-2 \dim \NO] \boxtimes \sF) \arrow{r} & \omega_{\GK}[-2\dim \GO] \boxtimes r_* \sF
\end{tikzcd}
\end{equation*}
where the left vertical arrow is the base change morphism $ \tilde p^* (\id_\sT, r)_* \rightarrow (\id_\GK, r)_* \tilde p_Z^* $ for the Cartesian square
\begin{equation} \label{eq:square}
\begin{tikzcd}
\sT \times Z \arrow{d}{\id_\sT,r} & \GK \times Z \arrow{l}{\tilde p_Z} \arrow{d}{\id_\GK,r}  \\
 \sT \times \NO & \GK \times \NO \arrow{l}{\tilde p}  
\end{tikzcd}
\end{equation}
and the top horizontal arrow is (\ref{eq:defphi}), for $ r_* \sF$.
\end{Lemma}
\begin{proof}
Since $ \tilde p_Z $ factors as $ \GK \times Z \rightarrow \GK \times \NO \times Z \rightarrow \sT \times Z $, the proof of Lemma \ref{Lemma: morphism to put into restriction with supports} gives the desired result.
\end{proof}
We apply pull-back with supports with respect to (\ref{extended morphism 0}) and we get
$$
\omega_\sR [-2\dim \NO] \boxtimes \sF \rightarrow (p_Z)_* j_Z^!( \omega_{\GK}[-2\dim \GO] \boxtimes \sF) =  (p_Z)_* c_1^!( \sF [-2\dim \GO])
$$
Using the equivariant structure $ c_1^! \sF \xrightarrow{\beta} c_2^! \sF $, the isomorphism from $ \GO$-equivariance, and proper direct image for $ m_Z$, we get
\begin{equation} \label{eq:actionmodified}
H^\bullet_{\GO \times \GO}(\sR \times Z, \omega_\sR[-2\dim \NO] \boxtimes \sF) \rightarrow H^\bullet_\GO(Z, \sF)
\end{equation}

\begin{Proposition}
\label{prop: auxiliary action}
	Under the isomorphism $H^\bullet_\GO(Z, \sF) \cong H^\bullet_\GO(\NO, r_* \sF) $, the action of $ \mathcal A_0 $ on $ H^{-\bullet}_\GO(Z, \sF)$ is given by (\ref{eq:actionmodified}).
\end{Proposition}

\begin{proof}

Using condensed notation as in (\ref{eq:actionCondensed}), consider the diagram
$$
\begin{tikzcd}[column sep=scriptsize]
H(\omega_\sR \boxtimes \sF) \ar[r,"p_Z^*"] \ar[d,"\sim"] & H(c_1^!\sF) \ar[r,"\beta"]\ar[d,"\sim"] & H(c_2^! \sF) \ar[r,"\sim"] \ar[d,"\sim"] & H(m_Z^! \sF) \ar[r,"(m_Z)_*"] \ar[d,"\sim"] & H(\sF) \ar[d,"\sim"]\\
H(\omega_\sR \boxtimes r_\ast\sF) \ar[r,"p^*"] & H(d_1^!r_\ast\sF)\ar[r,"r^{(1)}_\ast\beta"] & H(d_2^! r_\ast\sF) \ar[r,"\sim"] &H(m^! r_\ast\sF) \ar[r,"m_*"] & H(r_\ast\sF)
\end{tikzcd}
$$
The top row encodes the auxiliary action via (\ref{eq: diagram for actionZ}), while the bottom row encodes the usual action (\ref{eq:actionCondensed}) applied to $r_\ast \sF$.  The vertical arrows are the obvious isomorphisms. We wish to show that all squares in the diagram commute. This is straightforward for all squares except for the leftmost one; for example, the square involving $\beta$ is commutative by our definition of the $\groupoid$--equivariant structure on $r_\ast \sF$.

Commutativity for the leftmost square will follow if we show that, when we apply $(\id_{\sR}, r)_* $ to the morphism
	\begin{equation}
	\label{extended morphism 1}
	\omega_\sR [-2\dim \NO] \boxtimes \sF \rightarrow  (p_Z)_* c_1^!( \sF [-2\dim \GO]),
	\end{equation}
	then we obtain the morphism (\ref{eq:restrictionF}) for the sheaf  $r_* \sF $
	\begin{equation}
	\label{extended morphism 2}
		\omega_\sR [-2\dim \NO] \boxtimes r_*\sF \rightarrow  p_* d_1^!( r_* \sF [-2\dim \GO])
	\end{equation}
To prove this we consider the following cube, whose back and front squares were used to define the above morphisms:
\begin{equation}
\label{extended morphism 3}
\begin{tikzcd}[column sep = scriptsize, row sep = scriptsize]
\sR \times Z \ar[dd,"i_Z"] \ar[dr] & & \groupoid \times_{\NO} Z \ar[ll,"p_Z"] \ar[dd, "j_Z", near start] \ar[dr] & \\
& \sR \times \NO \ar[dd, crossing over, "i", near start] & & \groupoid \ar[ll, crossing over, "p", near start] \ar[dd, "j"] \\
\sT \times Z \ar[dr] & & \GK \times Z \ar[ll,"\tilde p_Z",near start] \ar[dr] & \\
& \sT \times \NO & & \GK \times \NO \ar[ll, "\tilde p"]
\end{tikzcd}
\end{equation}
The arrows from the back square to the front are induced from $r: Z\rightarrow \NO$ in the obvious ways, and all of the faces are Cartesian. 

Denoting $A = \omega_\sT [-2\dim \NO] \boxtimes \sF$ and $B = \omega_{\GK}[-2\dim \GO] \boxtimes \sF$, recall that by the adjoint to  (\ref{extended morphism 0}) we have a morphism $A \rightarrow (\tilde p_Z)_\ast B$.  Consider a commutative diagram built out of this morphism by applying functors and natural isomorphisms:
\begin{equation*}
\begin{tikzcd}[column sep = scriptsize, row sep = scriptsize]
& (\id_\sR, r)_\ast i_Z^! A \ar[d] \ar[r,"\sim"] & i^! (\id_\sT, r)_\ast A \ar[d] & \\
& (\id_\sR, r)_\ast i_Z^! (\tilde p_Z)_\ast B \ar[r,"\sim"] \ar[dl,"\sim"'] & i^! (\id_\sT, r)_\ast (\tilde p_Z)_\ast B \ar[dr,"\sim"] & \\
(\id_\sR,r)_\ast (p_Z)_\ast j_Z^! B \ar[dr,"\sim"']& & & i^! \tilde p_\ast (\id_\GK, r)_\ast B \ar[dl,"\sim"] \\
& p_\ast r^{(1)}_\ast j_Z^! B \ar[r,"\sim"] & p_\ast j^! (\id_\GK, r)_\ast B &
\end{tikzcd}
\end{equation*}
The vertical arrows at the top come from $A\rightarrow (\tilde p_Z)_\ast B$. The commutativity of the hexagon is encoded by \cite[Figure B.8(b)]{AHR} applied to the cube (\ref{extended morphism 3}).

Now on the one hand, by passing from the top left to the far left we get precisely the functor $(\id_\sR, r)_\ast$ applied to (\ref{extended morphism 1}).  On the other hand, passing all the way down the right side of the diagram gives the pull-back with supports (for the front square) with respect to the morphism
\begin{equation}
\label{extended morphism 5}
 \tilde p^* (\id_\sT, r)_\ast A \longrightarrow  (\id_\GK, r)_\ast B
\end{equation}
which is adjoint to the result of applying $(\id_\sT,r)_\ast $ to the adjoint of the morphism (\ref{extended morphism 0}).
 
 Finally, observe that the map (\ref{extended morphism 2}) is defined by pull-back with supports of an a priori different morphism $ \tilde p^* (\id_\sT, r)_\ast A \longrightarrow  (\id_\GK, r)_\ast B $ given by (\ref{eq:defphi}).  Thus we must show that the morphisms (\ref{extended morphism 5}) and (\ref{eq:defphi}) are equal.

Consider the following diagram of sheaves on $ \sT \times \NO $.
\begin{equation}
\begin{tikzcd}
(\id_\sT,r)_* A \arrow{r}\arrow{d} & \tilde p_* \tilde p^* (\id_\sT, r)_\ast A \arrow{rrr}{\tilde p_*(\ref{eq:defphi})} \arrow{d}& & &   \tilde p_* (\id_\GK,r)_* B  \arrow[equal]{d} \\
(\id_\sT,r)_* (\tilde p_Z)_* \tilde p_Z^* A \arrow{r} & \tilde p_* (\id_\GK,r)_* \tilde p_Z^* A \arrow{rrr}{\tilde p_* (\id,r)_* (\ref{extended morphism 0})} & & & \tilde p_* (\id_\GK,r)_* B
\end{tikzcd}
\end{equation}
The left square commutes by applying Lemma \ref{le:adjunctionBaseChange} below to the Cartesian square (\ref{eq:square}) and the right square commutes by Lemma \ref{Lemma: morphism to put into restriction with supports2}.

Following this diagram along the top, and then down the right, gives the adjoint of (\ref{eq:defphi}). Following the diagram down, and then along the bottom gives the adjoint of (\ref{extended morphism 5}).  Thus the morphisms (\ref{eq:defphi}) and (\ref{extended morphism 5}) are equal and the result follows.

\end{proof}

\begin{Lemma} \label{le:adjunctionBaseChange}
	Consider a Cartesian square
	\begin{equation*}
	\begin{tikzcd}
	X \arrow{d}{i} & Y \arrow{d}{j} \arrow{l}{g} \\
	M & N \arrow{l}{f}     
	\end{tikzcd}
	\end{equation*}
	Then for any sheaf $ A $ on $ X $, the following diagram commutes
	\begin{equation*}
	\begin{tikzcd}
	i_*A \arrow{r} \arrow{d} & f_* f^* i_* A \arrow{d} \\
	i_* g_*g^*A \arrow{r} & f_* j_* g^*A
	\end{tikzcd}
	\end{equation*}
	where the right vertical morphism is the base change morphism.
\end{Lemma}

\subsection{Examples of the action}
There are some special cases of this auxiliary action which will be of interest.

\subsubsection{Springer fibres}
Let $ c \in \NK $ and assume that the stack $[\Vco/\GO] $ admits a dualizing sheaf.  Then we take $ Z = \Vco, \mathcal F = \omega_{\Vco}[-2\dim \GO + 2\dim \Lco] $.   As discussed before, in this case the module is isomorphic to $ H_\bullet^{\Lco}(\Sp_c) $.

\subsubsection{Losing matter}
\label{sec: losing matter}
Suppose that $N = N' \oplus N''$ splits as a representation of $G$.  Then there is an inclusion of algebras
\begin{equation}
\label{eq: lost matter inclusion}
\A_0(G, N) = \A_0(G, N'\oplus N'') \hookrightarrow \A_0(G, N')
\end{equation}
by \cite[Remark 5.14]{BFN1}. Given a $\groupoid$--equivariant sheaf $\sF$ on $N'_\cO$, we may define an action of $\A_0(G,N)$ on $H_{\GO}^{-\bullet}( N'_\cO, \sF)$ in two a priori different ways:

On the one hand, by Theorem \ref{th:action} the algebra $\A_0(G, N')$ acts on $H_{\GO}^{-\bullet}( N'_\cO, \sF)$.  By restricting this module under the inclusion map (\ref{eq: lost matter inclusion}), we obtain an action of $\A_0(G, N)$.

On the other hand, the space $Z = N'_\cO$ and its sheaf $\cF$ naturally fit into the framework of Section \ref{se:Auxillary}. So Proposition \ref{prop: auxiliary action} provides another action of $\A_0(G,N)$ on $H_{\GO}^{-\bullet}( N'_\cO, \sF)$.

\begin{Proposition}
\label{prop: losing matter}
The action of $\A_0(G,N)$ on $H_{\GO}^{-\bullet}( N'_\cO, \sF)$ via the inclusion (\ref{eq: lost matter inclusion}) agrees with the action defined by the auxiliary action diagram (\ref{eq: diagram for actionZ}).
\end{Proposition}

Before we prove this result,  let us first recall the construction of (\ref{eq: lost matter inclusion}).  For brevity we will omit $G$ from our notation, writing $\sR_{N}$ instead of $\sR_{G,N}$, and so on. Consider the space
$$
\sR_{N', N''} = \left\{ [g, (n_1,n_2)] \in \GK\times_{\GO} (N'_\cO \oplus N''_\cO ) \ : \ gn_1 \in N'_\cO \right\}
$$
We use this space as an intermediary, via the obvious diagram
$$
\begin{tikzcd}
\sR_N \ar[r,hook,"\inclu"] & \sR_{N', N''} \ar[r,twoheadrightarrow,"\ontu"] & \sR_{N'}
\end{tikzcd}
$$
The map $\inclu$ is a closed embedding, and therefore proper.  Applying Section \ref{section: proper direct image} to the sheaf $\omega_{\sR_{N', N''}}[-2\dim \NO]$ we obtain the  proper direct image map $H^\GO_\ast(\sR_N) \rightarrow H^\GO_\ast(\sR_{N', N''})$. 

Meanwhile, $\ontu$ is a vector bundle with fibres isomorphic to $N''_\cO$, so $\ontu^\ast = \ontu^![-2\dim N''_\cO]$. Together  with the unit of the adjunction $(\ontu^\ast, \ontu_\ast)$, we get a morphism $$\omega_{\sR_{N'}}[-2\dim N'_\cO] \longrightarrow \ontu_\ast \omega_{\sR_{N', N''}}[-2\dim \NO]
$$ Taking cohomology gives the Gysin isomorphism $H_\bullet^{\GO}(\sR_{N'}) \xrightarrow{\sim} H_\bullet^\GO(\sR_{N', N''})$.  Composing the inverse of this isomorphism with the above proper direct image defines (\ref{eq: lost matter inclusion}):
$$
\begin{tikzcd}
H^\GO_\ast(\sR_N) \ar[r,hook]& H^\GO_\ast(\sR_{N', N''}) \ar[r,"\sim"] &H_\bullet^{\GO}(\sR_{N'})
\end{tikzcd}
$$

\begin{proof}[Proof of Proposition \ref{prop: losing matter}]
Consider a commutative diagram:
\begin{equation}
\label{eq: big losing matter}
\begin{tikzcd}
\sR_N \times N'_\cO \arrow[swap,hook]{d}{\inclu\times \id} & \groupoid_{N'} \arrow{l} \arrow{r} \arrow[equal]{d}& \sR_{N'} \arrow{r} \arrow[equal]{d} & N'_\cO\arrow[equal]{d}\\
\sR_{N', N''} \times N'_\cO \arrow[hook]{d} & \groupoid_{N'} \arrow[hook]{d} \arrow{r} \arrow{l} & \sR_{N'} \arrow{r} & N'_\cO \\
\sT_N \times N'_\cO \arrow[leftarrow]{r}  & \GK \times N'_\cO & &
\end{tikzcd}
\end{equation}
The arrows in the left squares are just restrictions  of the usual ones from the diagram (\ref{eq: diagram for action1}) for $(G, N)$.  Meanwhile, the rightward pointing arrows are precisely those from the diagram (\ref{eq: diagram for action1}) for $(G, N')$.

The first and third rows form the auxiliary action diagram (\ref{eq: diagram for actionZ}) for $Z = N'_\cO$.  Note that $\groupoid_{N}\times_\NO N'_\cO \cong \groupoid_{N'}$ is the groupoid corresponding to $(G, N')$.

In Lemma \ref{lemma: for losing matter} below, we will show that the bottom two rows of (\ref{eq: big losing matter}) define the action of $\A_0(G,N')$ on $H_{\GO}^{-\bullet}( N'_\cO, \sF)$, under the above Gysin isomorphism for $\ontu$.

Proposition \ref{prop: losing matter} now follows from our results from Section \ref{section: some sheaf theory}, applied to the left-most squares: the pull-back with supports from the bottom row to the top may be computed by first passing to the middle row (section \ref{section: composition of pull-back}), and proper base change in the top left square (Lemma \ref{lemma: proper base change hypercohomology}).
\end{proof}

\begin{Lemma}
\label{lemma: for losing matter}
Under the Gysin isomorphism $\A_0(G,N') = H_\bullet^{\GO}(\sR_{N'}) \xrightarrow{\sim} H_\bullet^\GO(\sR_{N', N''})$ defined by $\ontu$ above, the action of $\A_0(G,N')$ on $H_{\GO}^{-\bullet}( N'_\cO, \sF)$ is defined by the bottom two rows of the diagram (\ref{eq: big losing matter}).
\end{Lemma}
\begin{proof}
It suffices to study the bottom left square in (\ref{eq: big losing matter}).  We extend this square, as the front face of the following commutative prism:
$$
\begin{tikzcd}
\sR_{N'} \times N'_\cO \arrow[hook]{dd}{i_{N'}} & & \\
& \sR_{N', N''} \times N'_\cO \arrow{ul}{\ontu, \id}  & \groupoid_{N'} \arrow[swap]{llu}{p_{N'}} \arrow{l}{p_{N',N''}} \arrow[hook]{dd}{j_{N'}} \\
\sT_{N'} \times N'_\cO \arrow[from=drr, "\tilde p_{N'}"',near end] & & \\
& \sT_{N} \times N'_\cO \ar[from = uu,hook,crossing over,"i_{N',N''}"]\arrow{ul}{\ontuu,\id} & \GK \times N'_\cO \arrow{l}{\tilde p_{N',N''}} 
\end{tikzcd}
$$
In the bottom triangle of the prism, the map $\ontuu: \sT_N \rightarrow \sT_{N'}$ is a vector bundle with fibres isomorphic to $N''_\cO$. Denoting $A = \omega_{\sT_{N'}}[-2\dim N'_\cO] \boxtimes \sF$ and $B = \omega_{\sT_N}[-2\dim \NO] \boxtimes \sF$, there is an isomorphism 
$$
(\ontuu,\id)^\ast A \cong (\ontuu,\id)^! A[-2\dim N''_\cO] \cong B
$$
Denoting $C = \omega_{\GK}[-2\dim \GO]\boxtimes \sF$, we have $ B \rightarrow (\tilde p_{N',N''})_\ast C$ from Lemma \ref{extended morphism 0}.
Finally, Lemma \ref{Lemma: morphism to put into restriction with supports} gives $ A \rightarrow (\tilde p_{N'})_\ast C$. It is not hard to see that the latter morphism is related to the first two: it is equal to the composition
\begin{equation}
\label{eq: prism base}
A \longrightarrow (\ontuu,\id)_\ast (\ontuu,\id)^\ast A \cong (\ontuu,\id)_\ast B \longrightarrow (\ontuu,\id)_\ast (\tilde p_{N', N''})_\ast C \cong (\tilde p_{N'})_\ast C
\end{equation}
We now build another large commutative diagram:
$$
\begin{tikzcd}[column sep = scriptsize, row sep = scriptsize]
i_{N'}^! A \ar[d] \ar[r] & (\ontu,\id)_\ast (\ontu,\id)^\ast i_{N'}^! A \ar[d]  \\
i_{N'}^! (\ontuu,\id)_\ast (\ontuu,\id)^\ast A \ar[d,"\sim"] \ar[r,"\sim"] & (\ontu,\id)_\ast  i_{N',N''}^! (\ontuu,\id)^\ast A \ar[d,"\sim"]\\
i_{N'}^! (\ontuu,\id)_\ast B \ar[r,"\sim"] \ar[d]   & (\ontu,\id)_\ast  i_{N',N''}^! B \ar[d] \\
i_{N'}^! (\ontuu,\id)_\ast (\tilde p_{N', N''})_\ast C \ar[r,"\sim"] \ar[d,"\sim"] &(\ontu,\id)_\ast  i_{N',N''}^! (\tilde p_{N', N''})_\ast C \ar[d,"\sim"] \\
i_{N'}^! (\tilde p_{N'})_\ast C \ar[d,"\sim"]& (\ontu,\id)_\ast (p_{N',N''})_\ast j_{N'}^!  C\\
(p_{N'})_\ast j_{N'}^! C \ar[ur,"\sim"] &
\end{tikzcd}
$$
The left column is the pull-back with supports for the composition (\ref{eq: prism base}). The top square is built out of the units of adjunctions plus base change, and commutes by a variation on Lemma \ref{le:adjunctionBaseChange}. The remainder of the right column comes by  base change of the left column. 

Now consider the pushfoward of this diagram to a point. First note that all maps in top square become isomorphisms after pushfoward to a point, since $\ontuu,\ontu$ are vector bundles. Taking into account these isomorphisms, passage down the left column is a rewriting of the restriction with supports for $p_{N'}^\ast A \cong C$, for the back square of the prism.  Its pushforward to a point is part of the definition of the action of $\A_0(G,N')$ on $H^{-\bullet}_\GO( N'_\cO, \sF)$. (Namely, the part corresponding to the left square in (\ref{eq: diagram for action1}). Similarly, the final three entries of the right column encode the corresponding part of the diagram (\ref{eq: big losing matter}). This proves the claim.
\end{proof}

\subsubsection{Equivariant cohomology of a point}
We take $ Z = \{0\} $ and $ r $ to be the inclusion of $ 0 $ into $ \NO $.  In this case we take $ \sF = \C_{\{0\}} $.  In this case, we see that the module is $ H_G^\bullet(pt) $.  This module is known as the GKLO representation, following the paper of Gerasimov-Kharchev-Lebedev-Oblezin \cite{GKLO}.  

The GKLO representation fits into the framework of Proposition \ref{prop: losing matter}, and in particular we see that it comes by restricting along the inclusion $\A_0(G,N) \hookrightarrow \A_0(G, 0)$. Note that this map is denoted $\mathbf{z}^\ast$ in \cite[Lemma 5.11]{BFN1}.  The action of the latter algebra is easy to describe:

\begin{Proposition}
The action of $\A_0(G,0) = H_\bullet^\GO(\Gr_G)$ on $1 \in H_G^\bullet(pt)$ is given by proper pushfoward along $\Gr_G \rightarrow pt$.
\end{Proposition}

In particular, the action may be computed by using the localization theorem to compute proper pushforward. This tells us that the GKLO representation is compatible with the embedding of the Coulomb branch into a localized ring of difference operators, as in \cite[Appendix A]{BFN2}.  In the quiver case, explicit formulas for this embedding (and thus the GKLO representation) were given in \cite[Appendix B]{BFN2}, cf.~also \cite[Section 4]{klrpaper}.

\subsubsection{Homology of the affine Grassmannian}
Now we take $ Z = \GK $ with its usual action of $ \GO $.  We take $ r $ to be the constant map to the point $ 0 \in \NO $ and we take $ \sF = \omega_{\GK}[-2\dim \GO] $ to be the dualizing sheaf of the affine Grassmannian.  The action of the groupoid $ \groupoid $ is simply given by the left translation action of $ \GK $ on itself.  In this case, the resulting module is given by $H_\bullet(\Gr)$, the homology of the affine Grassmannian.

The diagram (\ref{eq: diagram for actionZ}) simplifies to 
\begin{equation}
\begin{tikzcd}
\sR \times \GK \arrow[hook]{d} & \GK \times \GK \arrow{l} \arrow{r} \arrow[hook]{d} &  (\GK \times \GK) / \GO \arrow{r} & \GK \\
\sT \times \GK \arrow[leftarrow]{r}  & \GK \times \GK & 
\end{tikzcd}
\end{equation}

\subsection{Adding in flavour equivariance}
Now, we would like to extend the action by adding in flavour and loop equivariance.  Recall the notation $ \tG $ and $ \ttGOK $ from before.  And recall that $ \A = H_\bullet^\ttGOK(\sR_{G,N}) $.

Then Theorem \ref{th:action} extends to the following result.

\begin{Theorem}
	If $ \mathcal F $ is a $ \widetilde{\groupoid}$-equivariant object of the $ \ttGO \rtimes \Cx $-equivariant derived category of $ \NO$, then $ H^{-\bullet}_{\ttGO \rtimes \Cx}(\NO, \sF) $ carries an action of the algebra $ \A $.  
\end{Theorem}

Similarly, we can add loop rotation and flavour equivariance into the results from section \ref{se:Auxillary}. We have the following analog of Proposition \ref{prop: auxiliary action}  

\begin{Proposition} \label{pr: auxiliary action equivariant}
	 Suppose that we have a space $ Z $ as in \ref{se:Auxillary}, but carrying actions of $ \ttGO \rtimes \Cx $ and of the groupoid $ \widetilde{\groupoid}$. Let $ \sF $ be an $\ttGO \rtimes \Cx$-equivariant sheaf on $ Z $ equipped with a $ \widetilde{\groupoid}$-equivariant structure.  Then $ H^\bullet_{\ttGO \rtimes \Cx}(Z, \sF) $ carries an $ \A$-module structure described by the analog of (\ref{eq: diagram for actionZ}).
\end{Proposition}

\section{Weight modules and category $ \mathcal O $}

In this section, we will discuss the modules $ M_c $ and how they depend on specific attributes of the point $ c $.  We begin with some general discussion about the algebra $ \A $ and its types of modules.

\subsection{General properties of $ \A $ and its modules}
\subsubsection{Cartier duality}
Let $ \pi $ be a finitely generated abelian group.  Its \textbf{Cartier dual} $ H $ is defined to be the affine algebraic group $ \operatorname{Spec} \C[\pi] $ (where $\C[\pi] $ denotes the group algebra of $ \pi $).  If $ \pi $ is free abelian, then $ H $ is a torus with weight lattice $ \pi $.  In general, $H $ is the product of a torus and a finite group.

The representation category of $ H $ is semisimple with simple objects (all 1-dimensional) indexed by $ \pi $.  Thus, a representation of $ H $ on a complex vector space is equivalent to a $ \pi$-grading on that vector space.

Finally, the Lie algebra of $ H $ is given by $ \fh = (\pi \otimes_\Z \C)^*$.

\subsubsection{Weakly and strongly equivariant}
Let $ \pi, H $ be as above.  Let $\A $ be a $\C[\hbar] $-algebra with an action of $ H $; equivalently we have a grading $ \A = \oplus_{\sigma \in \pi} \A(\sigma)$.  We say that this action is \textbf{Hamiltonian}, if we are given a map $ \mathfrak h \rightarrow \A $, such that for all $ x \in \fh $ and $ a \in \A(\sigma) $, we have
$$
[x, a] =  \langle x, \sigma \rangle \hbar a
$$

Let $ M $ be an $ \A $-module.  We say that $ M $ is \textbf{weakly $H$-equivariant}, if we are given an action of $ H$ on $M $, compatible with its action on $ \A$.  Equivalently, we have a grading $ M = \oplus_{\sigma \in \pi} M(\sigma) $ such that $ A(\sigma_1) M(\sigma_2) \subset M(\sigma_1 + \sigma_2) $.  

Let $ \mathfrak u $ be another abelian Lie algebra and let $ \phi : \mathfrak h \oplus \C\hbar \rightarrow \mathfrak u $.  We say that $ M $ is \textbf{strongly $(H, \phi)$-equivariant}, if we are given a right action of $ \mathfrak u $ on $ M$ such that for $ x \in \fh \oplus \C\hbar $ and $ m \in M(\sigma) $, we have
$$
x \cdot m - m \cdot \phi(x) = \langle x, \sigma \rangle \hbar \cdot m
$$
(In particular, after specializing $ \hbar = 1 $, $ M$ is a Harish-Chandra bimodule for the group $ H$.)

\subsubsection{Equivariance for Coulomb branch algebras}
Recall that we have a short exact sequence of reductive groups
$$
1 \rightarrow G \rightarrow \tG \rightarrow F \rightarrow 1
$$
where $ F $ is a torus. Let $\widetilde T $ be a maximal torus of $ \tG $ containing $ T $.

This leads to a short exact sequence of fundamental groups
$$
0 \rightarrow \pi \rightarrow \widetilde \pi \rightarrow \tau \rightarrow 0
$$
and thus a short exact sequence of their Cartier duals denoted 
$$
1 \leftarrow F^! \leftarrow \tG^! \leftarrow G^! \leftarrow 1
$$
\begin{Remark}
	Note that $ F^! = \Hom(G, \Cx) \otimes_\Z \Cx $ is defined using $ G $ and $ G^! = \Hom(F, \Cx) \otimes_\Z \Cx$ is defined using the group $ F$, so the reader might find our naming convention here a bit bizarre.  The reader should keep in mind that $ F^! $ acts Hamiltonianly on the Coulomb branch, much as $ F $ acts Hamiltonianly on the Higgs branch.  This is one motivation for this choice of names.  Another motivation is that in the toric case, we end up with the Gale dual sequence of tori.
\end{Remark}

We also get a short exact sequence of abelian Lie algebras
\begin{equation}
\begin{tikzcd}  \label{eq:Cartan}
0 \ar{r} & H^2_{F}(pt) \ar{r} \ar[equal]{d} & H^2_{\tG}(pt) \ar{r} \ar[equal]{d} & H^2_{G }(pt) \ar{r} \ar[equal]{d} &  0 \\
0 \ar{r} &\fg^! = \mathfrak f^*  \ar{r} & \widetilde \fg^! = (\widetilde \ft^*)^W \ar{r} & \ff^! \ar{r} = (\ft^*)^W &  0
\end{tikzcd}
\end{equation}

Similar to the decomposition of the affine Grassmannian from section \ref{se:ses}, we get a disjoint decomposition $\sR_{G,N} = \sqcup_{\sigma \in \pi} \sR_{G,N}(\sigma)$ where
$$
\sR_{G,N}(\sigma) = \{ ([g], w) : [g] \in \Gr(-\sigma) \}
$$
and we have $ \A = \oplus_{\sigma \in \pi} \A(\sigma) $ where $ \A(\sigma) = H_\bullet^{\ttGO}(\sR_{G,N}(\sigma))$.  

This $\pi$-grading gives us an action of $\tG^! $ on $ \A $ which factors through $ F^!$.

We have the commutative subspace $ \widetilde \fg^! = H^2_{\tG}(pt) \subset \A $, which we refer to as the \textbf{Cartan subalgebra} of $ \A $.  In this way we differ from \cite{BFN1}, since they refer to the full $ H^\bullet_{\tG \times \Cx}(pt) $ as the Cartan subalgebra; we will instead call $ H^\bullet_{\tG \times \Cx}(pt) $ the \textbf{Gelfand-Tsetlin algebra}.

The following Lemma is equivalent to Lemma 3.19 from \cite{BFN1} and shows that the action of $ \tG^! $ on $ \A $ is Hamiltonian.
\begin{Lemma} \label{le:decompTorus}
	Let $ \sigma \in \pi $ and let $ x \in \widetilde \fg^! $.  Then for all $ a \in \A(\sigma) $, we have $ [x, a] = \hh \langle \sigma,  x \rangle a $.
\end{Lemma}  

\begin{Remark} \label{re:GvsF}
	Though the action of $ \tG^! $ on $ \A $ factors through $ F^!$, we don't naturally have a Hamiltonian $ F^! $-action on $ \A $, since we don't have a natural way to split (\ref{eq:Cartan}).
\end{Remark}

\subsubsection{Central specialization}
\label{section: central specialization}
In the decomposition $ \A = \oplus \A(\sigma) $,  $ \sigma $ ranges over $ \pi$, and so $ \fg^! \oplus \C\hbar = \mathfrak f^* \oplus \C\hbar $ is central in $ \A $. By \cite[\S 2]{BFN1}, $ \A $ is free over $ H^\bullet_{F\times \Cx}(pt) = \Sym \mathfrak f^* [\hbar]$.  

If $ \xi : \Sym \mathfrak f^*[\hbar] \rightarrow R $ is any algebra map (for example $ R = \C[\hbar] $ or $ \C$ are possible common choices), then we can specialize $ \A $, and we define $\A_\xi := \A \otimes_{\Sym  \mathfrak f^*[\hbar]} R$.

In particular, if $ \zeta \in \ff $, then we will consider the map $  \Sym \mathfrak f^*[\hbar] \rightarrow \C[\hbar] $, dual to the map of varieties $ \C \rightarrow \ff \oplus \C $, defined by $ a \mapsto (a\zeta, a) $.  We denote resulting specialization by 
$$ \A_{\zeta, \hbar} := \A \otimes_{\Sym  \mathfrak f^*[\hbar]} \C[\hbar] $$ 

Another special case will be if $ \xi \in \ff \oplus \C $, then we have $\xi : \Sym \mathfrak f^*[\hbar] \rightarrow \C $ given by evaluation at $ \xi $ and we will denote resulting specialization by 
$$ \A_\xi := \A \otimes_{\Sym  \mathfrak f^*[\hbar]} \C $$ 

Similarly, if $ M$ is any $ \A $-module, then we can specialize $M$ and form $ M_\xi :=  M \otimes_{\Sym  \mathfrak f^*[\hbar]} R$ which will be a module for the specialized algebra.

\subsubsection{Equivariant modules for Coulomb branch algebras}
As we have an action of $ F^!$ on $ \A $ and a Hamiltonian action of $ \tG^! $ on $ \A$, it makes sense to speak about weakly $ F^!$ and strongly $ \tG^! $-equivariant $\A$-modules. 

\begin{Definition} \label{def:Weight}
Let $ M $ be an $ \A$-module.  We say that $M $ is a \textbf{weight module}, if it is strongly $(\tG^!, \phi)$-equivariant for some  $ \phi : \widetilde \fg^! \oplus \C\hbar  \rightarrow \mathfrak u $, such that the action of $ \tG^! $ on $ M $ factors though $ F^! $, and the restriction $ \phi : \mathfrak f^* \oplus \C\hbar = \fg^! \oplus \C\hbar \rightarrow \mathfrak u $ is surjective.
\end{Definition}

\begin{Remark}
To see why this deserves the name ``weight module'', suppose that we choose a linear map $ \gamma: \mathfrak u \rightarrow \C $.  Composed with $ \phi : \mathfrak f^* \oplus \C\hbar \rightarrow \mathfrak u $ this gives us  $ \xi \in \ff \oplus \C $.  Thus we can form the specialized algebra $ \A_\xi $ and the specialized module $M_\xi $.

The action of $\tG^! $ on $ M$ (which factors through $F^!$) gives us a grading $ M = \oplus_{\sigma \in \pi} M(\sigma) $.  Since it is strongly equivariant, any  $ x \in \widetilde \fg^! $ will act on $ M_\xi(\sigma) $ as multiplication by the complex number $ \gamma(\phi(x)) + \xi(\hbar) \langle \sigma, \bar{x} \rangle$.  Thus $ M_\xi(\sigma)$ consists of $ \widetilde \fg^! $-eigenvectors with eigenvalue $ \gamma \circ \phi + \xi(\hbar) \sigma $.  
\end{Remark}

\subsubsection{Category $\OO$}
Fix a character $ \chi : G \rightarrow \C^\times $, equivalently a map of $ \Z$-modules, $ \chi : \pi \rightarrow \Z $.  We use $ \chi $ to collapse our $ \pi $-grading into a $ \Z $-grading, and define
$$ \A(n) := \bigoplus_{\sigma : \langle \chi, \sigma \rangle = n} \A(\sigma)
$$
 and let $\A_+ = \oplus_{n > 0} \A(n)$.  Similarly, if we have a weakly $ F^!$-equivariant (i.e. $\pi$-graded) module $ M $, we define the $ \Z$ grading $ M = \oplus M(n) $.

 \begin{Definition}\label{def:catO}
 We say that a weakly $F^!$-equivariant module $ M $ lies in \textbf{category} $ \OO $ if there exists $ N $ such that $ M(n) = 0 $ if $n > N $.  Note that $ \A_+ $ acts locally nilpotently on any category $ \mathcal O $ module.
 \end{Definition}

\begin{Remark} \label{re:attracting}
Inside the commutative algebra $ \A_0 $, we can consider the ideal $ I_+ $ generated by $ \oplus_{n > 0} \A_0(n) $.  We call $ \Spec \A_0 / I_+ $ the \textbf{attracting locus} of the Coulomb branch.  If $ M $ is a category $\OO$ module for $ \A $, then the specialization $ M_0 $ will give a quasi-coherent sheaf of the Coulomb branch which is set-theoretically supported on the attracting locus. 
\end{Remark}

\subsubsection{Example of $ \mathfrak{gl}_n$} \label{eg:sln}
	To help make sense of these definitions, we consider a familiar example.  Let $ G = \prod_{i = 1}^{n-1} GL_i $ and $ N = \oplus_{i = 1}^{n-1} \Hom(\C^i, \C^{i+1}) $ and $ F = (\Cx)^n $, as in Figure \ref{fig:1}.
	
	The resulting Coulomb branch algebra is closely related to the Lie algebra $ \mathfrak{gl}_n $.  To make a precise statement, we need to fix some notation. Write $ \fh$ for the Cartan subalgebra of diagonal matrices in $ \mathfrak{gl}_n$. Also recall the Harish-Chandra isomorphism $ Z(U \mathfrak{gl}_n) \cong (\Sym \fh)^{S_n}$ and the extended asymptotic enveloping algebra $$ \widetilde{ U}_\hbar \mathfrak{gl}_n :=  U_\hbar \mathfrak{gl}_n\otimes_{Z(U_\hbar \mathfrak{gl}_n)} \Sym \fh[\hbar]$$ which contains two copies of $ \mathfrak h$, one embedded as $ \fh \otimes \C $ and the other as $ \C \otimes \mathfrak h $.

	\begin{Theorem} \label{th:sln}
		With the above $ G, N, F $, the following hold.
\begin{enumerate} \item	There is an isomorphism $ \A \cong \widetilde{ U}_\hbar \mathfrak{gl}_n $.
	\item The Cartan subalgebra $\widetilde \fg^!$ of $ \A $ coincides with $ \C \otimes \mathfrak h  \oplus \mathfrak h \otimes \C  $ and (\ref{eq:Cartan}) becomes
	$$
	0 \rightarrow   \fg^!  = \C \otimes \fh \rightarrow \widetilde \fg^!  \rightarrow \ff^! = \fh \otimes \C \rightarrow  0
	$$
	\item The Gelfand-Tsetlin subalgebra $ H^\bullet_{\tG \times \Cx}(pt) $ coincides with the usual Gelfand-Tsetlin subalgebra of $ \widetilde{U}_\hbar \mathfrak{gl}_n $.
\end{enumerate}
\end{Theorem}

\begin{proof}
	All three statements follow by combining \cite[Theorem 4.3(a)]{WWY} and \cite[Corollary B.28]{BFN2}.
	
	A direct proof of the first two statements can also be found in \cite[Corollary 2.79]{FT}.
\end{proof}
	
	An example of a weight module for $\widetilde{U}_\hbar \mathfrak{gl}_n $ is the universal Verma module $ M = \widetilde{U}_\hbar \mathfrak{gl}_n \otimes_{U_\hbar \mathfrak n} \C[\hbar] $.  This is a weight module with $ \mathfrak u = \fh \oplus \C\hbar $, letting $ \phi : \fh \oplus \fh \oplus \C\hbar \rightarrow \fh \oplus \C\hbar $ be given by $ \phi(x_1, x_2, a\hbar) = (x_1 + x_2,a\hbar)$.  
	
	We choose $ \chi : G \rightarrow \Cx $ to be given by the product of the determinants. Then our definition of category $ \OO $ coincides with that from \cite[Section 3.2]{BLPW}, except that we don't  require that the module be finitely-generated over $ \A $. Our definition of $\OO$ is thus slightly different from the classical BGG category $\OO$, see \cite[Remark 3.11]{BLPW}.

\subsection{Modules from Springer fibres}

Let $ c \in \NK $ and let $ \eqstab \subset \ttGOK $ be the stabilizer of $ c $.  We have a left action of $ L_c $ on $ \Sp_c $.

Let $ K \subset L_c $ be subgroup.  We define $ M_{c,K} = H_\bullet^K(\Sp_c) $.  Our default choice for $ K $ will be $ K = L_c $ and so we will write $ M_c = M_{c, L_c } $.

We will study $ M_{c,K} $ as an $ \A$-module under various hypotheses on $ c, K$.

\begin{Proposition} \label{pr:ModuleExists}
		Assume that the stack $ [K \bs \Sp_c] $ admits a dualizing sheaf.  Then $ M_{c,K} $ carries an $ \A$-module structure. 
\end{Proposition}
For the rest of this section we will assume the existence of this dualizing sheaf.

To establish this result we will need flavour versions of $ \Xco, \Vco $ (from Definition \ref{def:Springer}):
$$
X_c := \{ g \in \ttGOK : gc \in \NO \} \quad V_c := (\ttGOK) c \, \cap \NO
$$
Note that we have a right action of $ L_c $ on $ X_c $ with $  X_c / L_c = V_c$.

\begin{proof}

As in Lemma \ref{le:isoStacks}, we have a map $ X_c \rightarrow \Sp_c $ given by $ g \mapsto [g^{-1}] $ leading to an isomorphism of stacks
\begin{equation} \label{eq:isoStacks}
  [ \ttGO \rtimes \Cx \bs X_c / K ] \cong [K \bs \Sp_c ]
\end{equation}
Using this isomorphism, we get a dualizing sheaf $ \mathcal F = \omega_{X_c}[-2\dim \ttGO \rtimes \Cx ] $ on $ X_c $.  Thus by Proposition \ref{pr: auxiliary action equivariant}, we deduce that we have an $ \A $-module structure on
$$
M_{c,K} = H_\bullet^{K}(\Sp_c) \cong H^\bullet_{\ttGO \rtimes \Cx \times K}(X_c, \mathcal F)
$$
\end{proof}

We note here the following simple but useful result, showing that the module $M_{c,K}$ only depends on the orbit of $c, K$ under $\ttGO\rtimes \Cx$:

\begin{Lemma} \label{le:orbitModule}
	For any $c\in \NK$ and $g\in \ttGOK$, we have an isomorphism $X_c / K \cong X_{gc} / gK g^{-1}$ and the $\A$--modules $M_{c,K}$ and $M_{gc, gKg^{-1}}$ are naturally isomorphic.  In particular $ M_c \cong M_{gc} $.
\end{Lemma}

\subsection{Equivariant modules and weight modules}
We will now impose other assumptions on $ c, K $ and see how this affects the resulting modules.

For $ \sigma \in \pi $, define $ X_c(\sigma) := X_c \cap \ttGOKnoloop(-\sigma) \rtimes \Cx $.  Under the map (\ref{eq:isoStacks}), $ X_c(\sigma) $ maps to $ \Sp_c(\sigma) $.

\begin{Proposition} \label{pr:weakEquiv}
	 Assume that $ K \subset \ttGOKnoloop(0) \rtimes \Cx $, the connected component of the identity.  Then $ M_{c,K} $ is weakly $ F^!$-equivariant with $$ M_{c,K}(\sigma) = H_\bullet^K(\Sp_c(\sigma))$$
\end{Proposition}

\begin{proof}
	Because $ K \subset \ttGOKnoloop(0) \rtimes \Cx $, the action of $ K $ preserves the components $ X_c(\sigma)$.
	
	Consider the first line of the auxiliary action diagram (\ref{eq: diagram for actionZ}) for the case $ Z =  X_c / K$.  
	\begin{equation*}
	\sR \times  X_c / K \leftarrow \groupoid \times_\NO  X_c / K \rightarrow  X_c / K
	\end{equation*}
	We can decompose this diagram into components as follows
	\begin{equation*}
	\sR(\sigma_1) \times  X_c(\sigma_2) /K  \leftarrow  \{ (g_1, g_2) : g_i \in \ttGOKnoloop(-\sigma_i) \rtimes \Cx,  g_2c, g_1 g_2 c \in \NO \} / K \rightarrow X_c(\sigma_1 + \sigma_2) / K
	\end{equation*}
	Thus $ \A(\sigma_1) $ maps $ M_{c, K}(\sigma_2) $ into $ M_{c, K}(\sigma_1 + \sigma_2) $ as desired.
\end{proof}

\begin{Remark}
	This proof shows that we get a weakly $F^!$-equivariant module for any $ Z, \mathcal F $, as in section \ref{se:Auxillary}, such that $ Z $ admits a disjoint decomposition labelled by $ \pi $.
\end{Remark}


Assume now that $ K \subset \ttGO \rtimes \Cx $.
In this case, we obtain a map 
\begin{equation}
\label{eq: weight map assumpt2}
\phi : \widetilde \fg^! \oplus \C \hbar = H^2_{\ttGO \rtimes \Cx}(pt) \rightarrow \mathfrak u := H^2_K(pt) 
\end{equation}
We have a right action of $ \mathfrak u $ on $ M_{c,K} =  H_\bullet^K(\Sp_c) $.

\begin{Proposition} \label{pr:strongEquiv}
	With the above assumptions, $ M_{c,K} $ is strongly $ (\widetilde G^!, \phi)$-equivariant.
\end{Proposition}

In order to prove this proposition, we will make a brief discussion about certain equivariant line bundles on $ \ttGOK $.

 Let $ \theta \in H^2_{\tG}(pt, \Z) $; we regard $ \theta $ as a group homomorphism $ \theta : \tG \rightarrow \Cx $.  We also get a group homomorphism, $ \bar \theta : \ttGO \rightarrow \Cx $, by composing
$$
\ttGO \xrightarrow{\theta} \cO^\times \rightarrow \Cx
$$
where the second arrow is evaluation at $ t = 0 $.

We also have the group homomorphism $ \hbar : \ttGO \rtimes \Cx \rightarrow \Cx $ given by projection onto the second factor.

We have left and right actions of $ \ttGO \rtimes \Cx $ on $ \ttGOK $.  The above group homomorphisms define equivariant line bundles $ \OO[\bar \theta, 0], \OO[0,\bar \theta - p \hbar] $ on $ \ttGOK $, for any $ p \in \Z $.

\begin{Lemma} \label{le:isoLine}
	Let $ \sigma \in \pi $ and $ \theta \in H^2_{\tG}(pt, \Z) $.  On $ Y:= \ttGOKnoloop(\sigma) \rtimes \Cx $, we have an isomorphism of equivariant line bundles
	$$
	\OO_Y[\bar \theta,0] \cong \OO_Y[0, \bar \theta - p \hbar]
	$$
	where $ p = \langle \sigma, \theta \rangle$.
\end{Lemma}

This lemma is probably well-known to experts (in fact it is equivalent to Lemma \ref{le:decompTorus}).  We will provide a simple explicit proof for convenience.

\begin{proof}
	To construct this isomorphism, we need a function $ f : Y \rightarrow \Cx $ such that 
	\begin{equation} \label{eq:lhg}
	f(lgr) = \bar \theta(l)  f(g) (\bar \theta - p\hbar)(r) \ \text{ for } l,r \in \ttGO \rtimes \Cx, g \in \ttGOKnoloop(\sigma) \rtimes \Cx
	\end{equation} 
	
	In order to define $ f$, we first fix some notation.  Given $ a \in \cK $, we write $ a_p $ for the coefficient of $ z^p$. In particular $ \bar \theta(h) = \theta(h)_0 $ for $ h \in \ttGO $.
	
	Also, given $ a \in \cK $ and $ s \in \Cx $, we write $ a^s $ for the action of $ s $ on $ a$ by loop rotation.  Note that $ (a^s)_p = s^p a_p $.  More generally, given $ g \in \ttGOK $ we write $ g^s $ for the action by loop rotation.  Note that $ \theta(g)^s = \theta(g^s) $.
	
	The character $ \theta : \ttGK \rightarrow \cK^\times $ restricts to a map $ \tG_\cK^\cO(\sigma) \rightarrow z^p \cO^\times $.  We define $ f $ by $ f(g_1, g_2) = \theta(g_1)_p g_2^{-p} $ where $ (g_1, g_2) \in \tG_\cK^\cO(\sigma) \rtimes \Cx$.
	
	Now let $ l = (l_1, l_2) \in \ttGO \rtimes \Cx $.  Then 
	\begin{gather*}
	f(lg) = f(l_1 g_1^{l_2}, l_2 g_2) = \theta(l_1 g_1^{l_2})_p l_2^{-p} g_2^{-p} \\ = \theta(l_1)_0 (\theta(g_1)^{l_2})_p l_2^{-p} g_2^{-p} = \theta(l_1)_0 \theta(g_1)_p g_2^{-p} = \bar \theta(l) f(g) 
	\end{gather*}  
	Similarly for $ r = (r_1, r_2) \in \ttGO \rtimes \C^\times $, we have
	\begin{gather*}
	f(gr) = f(g_1 r_1^{g_2}, g_2 r_2) = \theta(g_1 r_1^{g_2})_p g_2^{-p} r_2^{-p} \\ = \theta(g_1)_p (\theta(r_1)^{g_2})_0 g_2^{-p} r_2^{-p} = \theta(g_1)_p \theta(r_1)_0 g_2^{-p} r_2^{-p} =  f(g) (\bar \theta - p\hbar)(r) 
	\end{gather*}  
	Thus (\ref{eq:lhg}) holds and $ f $ provides the desired isomorphism of line bundles on $ Y$.
\end{proof}

\begin{proof}[Proof of Proposition \ref{pr:strongEquiv}]

	Note that $ K \subset \ttGO \rtimes \Cx $ and $ X_c(\sigma) \subset \ttGOKnoloop(-\sigma) \rtimes \Cx $.  So we have a map $$ [\ttGO \rtimes \Cx \bs X_c(\sigma) / K] \rightarrow  [\ttGO \rtimes \Cx \bs \ttGOKnoloop(-\sigma) \rtimes \Cx  / \ttGO \rtimes \Cx]$$
	
		Let $ \theta \in H^2_{\tG}(pt, \Z) $.  We need to show that for $ m \in M_{c,K}(\sigma) = H^\bullet_{\ttGO \rtimes \Cx \times K}(X_c(\sigma), \mathcal F) $, we have $ \theta \cdot m - m \cdot \theta = \hbar \langle \sigma, \theta \rangle m$.  The left and right actions of $ \theta $ on $ M_{c,K}$ come from first Chern classes of equivariant line bundles.
	
	Moreover these equivariant line bundles are pulled back from the corresponding left and right $\ttGO \rtimes \Cx$-equivariant line bundles on $ \ttGOKnoloop(-\sigma) \rtimes \Cx$.  Thus the result follows from the isomorphism of equivariant line bundles from Lemma \ref{le:isoLine}.
\end{proof}

\begin{Corollary} \label{co:weight}
	Assume that $ K \subset \ttGO \rtimes \Cx $ and that the natural map $ K \rightarrow \FO \rtimes \Cx $ has finite kernel. Then $ M_{c,K} $ is a weight module.
\end{Corollary}

\begin{proof}
	From Proposition \ref{pr:strongEquiv}, $ M_{c,K} $ is strongly $ \tG^!$-equivariant.  Since $  K \rightarrow \FO \rtimes \Cx $ has finite kernel, the map $ \ff^* \oplus \C\hbar \rightarrow \mathfrak u = H^2_K(pt) $ is surjective.
\end{proof}

\begin{Example}
The most important class of examples will be when $ c \in \NK $ is $\chi$-stable for some character $ \chi : \GK \rightarrow \cK^\times $, with $ \GK $ acting with trivial stabilizer.  Then by Theorem \ref{th:finitedim}, each component of $ \Sp_c $ is of finite type.  Also $ L_c $ maps injectively to $ \FO \rtimes \Cx $.  Hence the dualizing sheaf exists.  By Proposition \ref{pr:ModuleExists}, we have an $ \A$-module structure on $ M_c = H^{L_c}_\bullet(\Sp_c) $.  By Corollary \ref{co:weight}, this is a weight module.
\end{Example}
	
\begin{Example}
A more degenerate example is $ c = 0, K = \Cx $.  Then 
$$ X_c = \ttGOK, \quad X_c/ K = \ttGOKnoloop, \quad M_{c,K} \cong H_\bullet^{\Cx}(\Gr) $$  
In this case, $ \phi : \widetilde \fg^! \oplus \C \hbar  \rightarrow \C \hbar $ is the projection.  The theorem implies that $ x \in \widetilde \fg^! \oplus \C \hbar $ acts on $ H_\bullet^\Cx(\Gr(\sigma)) $ by $ \langle \sigma, x \rangle \hbar $.  This is a weight module by Corollary \ref{co:weight}.
\end{Example}

\begin{Example}
Another example is to take $ c = 0, K = \ttGO \rtimes \Cx $.  
$$X_c = \ttGOK, \quad X_c/ K = \Gr, \quad M_{c,K} \cong H_\bullet^{\ttGO \rtimes \Cx}(\Gr) $$ 
which carries left and right actions by $ H^\bullet_{\ttGO \rtimes \Cx}(pt) $.
In this case, $ \phi $ is the identity and the theorem shows that the left and right actions of $ x \in \widetilde \fg^! \oplus \C\hbar $ are equal up to a shift by $ \langle \sigma, x \rangle \hbar $ on each component $ H_\bullet^{\ttGO \rtimes \Cx}(\Gr(\sigma))$.

Moreover, the module in this case is $ \A(G,0)$ with the action of $ \A(G,N) $ coming from the natural inclusion $ \A(G,N) \hookrightarrow \A(G,0)$, as discussed in section \ref{sec: losing matter}.

This module is strongly $ (\tG^!, \phi) $-equivariant, but not a weight module.
\end{Example}

 \subsection{Quantum Hamiltonian reduction} 
 \label{section: QHR}
 Recall that we have a short exact sequence
 \[
 1 \to G \rightarrow \tG \rightarrow F \to 1
 \]
 where $F$ is the flavour torus.
 
 In this section, we will consider $ \tG $ as a gauge group.  So we consider the affine Grassmannian $ \Gr_\tG $, the BFN space $ \sR_{\tG, N} $ and the Coulomb branch algebra
 \[\widetilde{\A} = H^{\tG_{\cO} \rtimes \Cx }_{\bullet}(\sR_{\tG, N} ) \]
We will study the relation between Springer fibres in $ \Gr_\tG$ and those in $ \Gr_G $, and the corresponding $ \widetilde \A$-modules and $ \A$-modules.

The $ \widetilde \pi $-decompositions of $ \Gr_\tG, \sR_{\tG, N} $ can be reduced to  $ \tau$-decompositions
$$
\Gr_\tG = \bigsqcup_{\zeta \in \tau} \Gr_\tG(\zeta), \quad  \sR_{\tG, N} = \bigsqcup_{\zeta \in \tau} \sR_{\tG, N}(\zeta)
$$
This leads to a $ \tau$-grading of $ \widetilde \A $.  


It is clear from the definitions that $ \sR_{\tG, N}(0) = \sR_{G,N} $ and this leads to an identification $\widetilde \A(0) = H^{\ttGO \rtimes \Cx}(\sR_{G,N}) = \A $ (see \cite[Proposition 3.18]{BFN1}).

Let $ c \in \NK $.  We have the $ \tG$-Springer fibre $ \widetilde{\Sp}_c \subset \Gr_\tG $ and we consider the corresponding $\tau$-decomposition $ \widetilde{\Sp}_c = \sqcup_{\zeta \in \tau} \widetilde{\Sp}_c(\zeta)$.  

Let $ K $ be a subgroup of the stabilizer of $ c $ in $ \ttGK $.  Assume that the dualizing sheaf of $  K \bs \widetilde{\Sp}_c $ exists, so that $ \widetilde M_{c,K} := H^{K}_\bullet(\widetilde{\Sp}_c) $ is an $ \widetilde \A $-module by Proposition \ref{pr:ModuleExists}.

 \begin{Proposition} \label{pr:reduction} Assume that $K\subset \ttGOK $.
 	\begin{enumerate}
 		\item We have a decomposition $$  \widetilde M_{c,K} = \oplus_{\zeta \in \tau} \widetilde M_{c,K}(\zeta), \text{ where } \widetilde M_{c,K}(\zeta) = H^K_\bullet(\widetilde{\Sp}_c(\zeta)) $$
 		\item The module $ \widetilde M_{c,K} $ is strongly $ (G^!, \phi)$-equivariant, where 
 		$$
 		\phi : \fg^! \oplus \C\hbar = H^2_{F \times \Cx}(pt) \rightarrow H_K^2(pt) = \mathfrak u
 		$$
 		comes from $ K \subset \ttGOK \rightarrow \FO \rtimes \Cx $.

 		\item For each $ \zeta \in \tau$, the graded component $\widetilde{M}_{c, K}(\zeta)$ is an $\A$-module.
 		
 		\item Let $ \mu : \Cx \rightarrow \tG $. We have an isomorphism
 		\begin{align*}
 		\widetilde{\Sp}_c([-\mu]) &\cong \Sp_{z^\mu c} \\
 		[g] &\mapsto [z^\mu g]
 		\end{align*}
 		which induces an isomorphism of $\A$-modules
 		\[
 		M_{z^\mu c, K_\mu} \cong \widetilde{M}_{c, K}([-\mu])
 		\]
 		where $K_\mu = z^{\mu} K z^{-\mu}$ and $ [-\mu] $ denotes the image of $ -\mu $ in $ \tau $.
 	\end{enumerate}
 \end{Proposition}
 
 \begin{proof}
 	Parts (1), (2) follow similar to Propositions \ref{pr:weakEquiv} and \ref{pr:strongEquiv}.  Part (3) follows immediately from the fact that $ \A = \widetilde \A(0) $. 
 	
 	For Part (4), it is easy to see the isomorphism of Springer fibres.  To check that it gives an isomorphism of modules, we just note the action diagrams for $ \widetilde \A $ restricts to the action diagram for $ \A $, as in the proof of Proposition \ref{pr:weakEquiv}.
 \end{proof}

Now, let $ c \in N $. Recall that we assumed the existence of $ \mu_0 : \Cx \rightarrow \tG $ which acts by scaling on $ N $.  Thus
\[
\Cxz = \{ (s^{-\tfrac{1}{2}\mu_0}, s) \, | \, s \in \Cx \} \subseteq \tG \times \Cx \}
\]
fixes $c$. We write $\Cxzf \subset F \times \Cx$ for the image of this subgroup in $ F \times \Cx $. Also, recall that for $ \zeta \in \tau \subset \ff $, we defined the specialization $ \A_{\zeta, \hbar} $ in section \ref{section: central specialization}.

The previous proposition specializes to give the following result.
 \begin{Proposition}   Let $ c \in N $ and assume that $ K $ lies in the preimage of $ \Cxzf$ under the map $ \tG_\cK \rtimes \Cx \rightarrow F_\cK \rtimes \Cx $.  For each $ \zeta \in \tau $, the graded component $\widetilde{M}_{c,K}(\zeta)$  is a module over $\Az$.
 	
 \end{Proposition}
\begin{proof}
Note that because $ K $ lies in the preimage of $\Cxzf$, the map $ K \rightarrow \FO \rtimes \Cx $ factors through $\Cxzf$.  Hence the resulting map $ \phi  : H^2_{F \times \Cx}(pt) \rightarrow H_K^2(pt) $ factors through $  H^2_{F \times \Cx}(pt) \rightarrow H^2_{\Cxzf}(pt) $.  Since $ H^\bullet_{F \times \Cx}(pt) $ is generated in degree 2, the result follows from Proposition \ref{pr:reduction}.(2).
 \end{proof}

 \subsection{Category O condition}
 We will now give a condition for when the Springer fibre module lies in category $ \OO$.  Recall that we fixed $ \chi : G \rightarrow \Cx $ which we can regard as a map $ \chi : \pi \rightarrow \Z $.  We can use $ \chi $ to collapse the grading on $ \Gr $ into a $ \Z $-grading, with 
 $$\Gr(n) :=  \bigsqcup_{\sigma : \langle \chi, \sigma \rangle= n} \Gr(\sigma)$$
 Note that $ [g] \in \Gr(n) $ if and only if $ \val \chi(g) = n $.
 
 Similarly, we use $ \chi $ to collapse the $ \pi$-gradings on $ \A, M_c $ into $ \Z$-gradings and to collapse the decompositions of $ \Sp_c $ into a $\Z$-decomposition, so
 $$
 \Sp_c(n) := \bigsqcup_{\sigma : \langle \chi, \sigma \rangle= n} \Sp_c(\sigma)
 $$
 
 
 
 \begin{Theorem} \label{th:catO} Assume that $ c \in \NK $ is $ \chi$-semistable and that $ L_c \subset \ttGOKnoloop(0) \rtimes \Cx$. Then $ M_c $ is weakly $F^!$-equivariant  and lies in category $ \OO $.  
 	
 \end{Theorem}

 \begin{proof}
 	From Proposition \ref{pr:weakEquiv}, we know that $ M_c $ is weakly $ F^!$-equivariant.
 	
 	Since $ c $ is $ \chi$-semistable, there exists $r \in \N , f \in \cK[N]^{G, r\chi} $ such that $ f(c) \ne 0 $. 
 	
 	Since $ \cK[N] = \C[N] \otimes \cK $, we can write $ f = \sum f_i \otimes p_i $ where $ f_i \in \C[N] $ and $ p_i \in \cK $.  Let $ m = \min_i \val p_i $.
 	
 	Suppose that $ [g] \in \Sp_c(n) $.  Then $ g^{-1}c \in \NO $ and thus $ f_i(g^{-1}c) \in \cO $ for $ i $.  Hence 
 	\begin{equation}
 	\label{eq:valN}
 	\val f(g^{-1}c) = \val \bigl( \sum_i f_i(g^{-1}c) p_i \bigr) \ge \min_i \val(f_i(g^{-1}c)) +\val(p_i) \ge m
 	\end{equation}
 	On the other hand, $ f(g^{-1}c) = \chi(g^{-1})^r f(c) $.  Since $ [g] \in \Sp_c(n) $, we see that $ \val \chi(g^{-1}) = -n $ and so $ \val f(g^{-1}c) = -nr + \val f(c) $.  Combining with (\ref{eq:valN}), we deduce that $ n \le \frac{ -m - \val f(c)}{r} $.  Thus the module $ M_c $ is bounded above as desired.
 \end{proof}

\subsection{Verma-like modules} \label{se:Verma}

Throughout this section, we will assume that $ c\in N $, is $ \chi$-stable and that the natural map $Stab_{\widetilde G}(c) \subset \widetilde G \rightarrow F $ is an isomorphism.

Given such a $ c $, let $ \psi : F \rightarrow Stab_{\widetilde G}(c) \subset \tG $ be the inverse isomorphism.  Recall the coweight $ \mu_0 : \Cx \rightarrow \tG $ which gives a scaling action on $ N$.

\begin{Lemma} \label{le:stabShift}
	In this situation, there is an isomorphism $ \widehat \psi :\FO \rtimes \Cx \rightarrow \eqstab $ given by 
	\begin{equation*} 
	 (g, s) \mapsto (\psi(g)s^{-\tfrac{1}{2}\mu_0}, s)
	 \end{equation*}
\end{Lemma}


\begin{Theorem} \label{th:Springerpoint1}
	Let $ c $ be as above. Then $\Sp_c(n) = \emptyset $ for $ n > 0 $ and $ \Sp_c(0) = \{ z^0 \} $.  
\end{Theorem}

\begin{proof}
Assume $ n \ge 0 $.  Let $ g \in \GK $ be such that $ [g] \in \Sp_c(n) $.  Then $ g^{-1} c \in \NO $.  Also since $ \val \chi(g) = n $, we see that $ \chi(g) \in \C[[z]] $.  Thus, we see that $ (g^{-1}c, \chi(g))  \in \NO \times\C[[z]] $.

Consider the map $ G \rightarrow N \times \C $ given by $ h \mapsto (h^{-1}c, \chi(h)) $.  Since $ c $ is $ \chi$-stable, this map is proper by Lemma \ref{le:StableProper}. Consider the diagram
\begin{equation*}
\begin{tikzcd}
\Spec \cK \arrow{r} \arrow{d}{g} & \Spec \cO \arrow{d}{(g^{-1}c, \chi(g))} \\
G \arrow{r} & N \times \C 
\end{tikzcd}
\end{equation*}
By the valuative criterion of properness, we deduce that $ g $ extends to a map $ \Spec \cO \rightarrow G $.  In other words, $ g \in \GO $ and thus $ [g] = z^0 $.  This proves the desired assertions about $ \Sp_c(n)$.
\end{proof}

Let $ v_0 \in M_c(0) =  H_\bullet^{L_c}(\{z^0\})$ denote the fundamental class. In Corollary \ref{co:weight}, we determined the action of the Cartan subalgebra $ H^2_{\tG}(pt) $ on $v_0 $.  However, this component of the Springer fibre is just a point by Theorem \ref{th:Springerpoint1}, so we can determine the action of the whole Gelfand-Tsetlin algebra.

Let $ \widehat \psi^* :  H^\bullet_{\tG \times \Cx}(pt) \rightarrow H^\bullet_{F \times \Cx}(pt) $ be the map dual to $ \widehat \psi $.

\begin{Lemma} \label{le:highestweight}
The vector $ v_0 $ is a highest weight vector (annihilated by $ \A_+ $) and the action of the Gelfand-Tsetlin algebra $ H^\bullet_{\tG \times \Cx}(pt) $ on $ v_0 $ is via the map $ \widehat \psi^* : H^\bullet_{\tG \times \Cx}(pt) \rightarrow H^\bullet_{F \times \Cx}(pt) $.
\end{Lemma}

\begin{proof}
	We consider the preimage of $ z^0 $ inside $ X_c$.  This preimage is simply $ \ttGO \rtimes \Cx $, which carries a left action of $  \ttGO \rtimes \Cx$ and a right action of $ \FO \rtimes \Cx $, which acts through $\widehat \psi$.  Thus the stabilizer for the action of $ \ttGO \rtimes \Cx \times \FO \rtimes \Cx $ is $$\{(\widehat \psi(g), g) : g \in \FO \rtimes \Cx \} $$ 
	Thus we conclude that the action of  $ H^\bullet_{\tG \times \Cx \times F \times \Cx}(pt) $ on $  H_\bullet^{\eqstab}(\{z^0\})$ is given by
	$$
H^\bullet_{\tG \times \Cx} \otimes H^\bullet_{F \times \Cx} \xrightarrow{\widehat \psi^* \otimes id} H^\bullet_{F \times \Cx} \otimes H^\bullet_{F \times \Cx} \rightarrow H^\bullet_{F \times \Cx}
$$
where the second map is multiplication.  In particular, we see that $ x \in H^\bullet_{\tG \times \Cx}(pt) $ acts the same way as $ \widehat \psi^*(x) $ as desired.
\end{proof}

\begin{Remark}
The highest weight space of $ M_c $ is free of rank 1 over the equivariant parameters.  For this reason, we will call the resulting modules, and those of the form $ M_{z^\mu c}$, ``Verma-like''.  However, in general they will not be Verma modules in any sense, and in particular they will often not be generated by the highest weight vector $ v_0 $.  In the abelian case, it is possible to completely analyze these modules and determine which are actually Vermas (see Proposition \ref{pr:AbelianVerma})
\end{Remark}

\subsection{Fixed points and the Hikita conjecture} \label{se:Hikita}
We would like to relate these Verma-like modules to a conjecture of Hikita and Nakajima.

Recall the Higgs branch $ Y = T^* N \sssslash_\chi G = \Phi^{-1}(0) \sslash_\chi G $.  We have an action of $ \tG \times \Cx $ on $\Phi^{-1}(0) $ where $ \Cx $ acts by weight $ 1/2 $ scaling on $ T^* N $.  This action leads to an action of $ F \times \Cx $ on $ Y $ and a Kirwan map 
$$ 
H^\bullet_{\tG \times \Cx}(pt) \rightarrow H^\bullet_{F \times \Cx}(Y) 
$$ 
which is known to be surjective in many examples.

On the other hand, define the $ B$-algebra of the algebra $ \A$, $ B(\A) $ by setting
$$
B(\A) := \A(0) / \langle ab : a \in \A(-n), b \in \A(n), n > 0 \rangle
$$

Since we have a map $ H^\bullet_{\tG \times \Cx}(pt) \rightarrow \A(0) $, we obtain a map 
\begin{equation} \label{eq:HtoB}
H^\bullet_{\tG \times \Cx}(pt) \rightarrow B(\A)
\end{equation}
 which is often surjective.

The following conjecture is an extension of conjectures by Hikita and Nakajima, c.f~\cite{Hikita} and \cite[Section 8]{moncrystals}.

\begin{Conjecture}
\label{hikita conjecture}
There is an algebra isomorphism $ H^\bullet_{F \times \Cx}(Y) \cong B(\A) $ which is compatible with the map from $ H^\bullet_{\tG \times \Cx}(pt) $ to both sides.
\end{Conjecture}

Note that if the maps from $ H^\bullet_{\tG \times \Cx}(pt)$ to both sides are surjective, then this conjecture is equivalent to requiring both sides are the same quotient (i.e. there is no data in the isomorphism).

We will now examine the relationship between this Conjecture and Lemma \ref{le:highestweight}.

Let $ c $ be as in section \ref{se:Verma}.  Then we have a point  $ [c] \in N\sslash_\chi G \subset Y  $ which is fixed by $ F\times \C^\times $.  The stabilizer of $c $ in $ \Phi^{-1}(0) $ is given by $ \{ (\psi(g) s^{-\tfrac{\mu_0}{2}}, s) : g \in F, s \in \Cx \} $ as in Lemma \ref{le:stabShift}.

  Also $ c $ gives rise to a restriction map $ H^\bullet_{F \times \Cx}(Y) \rightarrow H^\bullet_{F \times \Cx}(\{c\}) $.  The resulting composition $ H^\bullet_{\tG \times \Cx}(pt)\rightarrow H^\bullet_{F \times \Cx}(Y) \rightarrow H^\bullet_{F \times \Cx}(pt) $ agrees with $ \widehat \psi^* $.

On the other hand, the action of $ \A(0) $ on the vector $ v_0 \in M(0) \cong H_{F \times \Cx}^\bullet(pt)  $ gives an algebra map $ \A(0) \rightarrow  H_{F \times \Cx}^\bullet(pt)  $.  This map factors through $ B(\A) $ since $ \A_+ $ annihilates $ v_0 $.  Thus we obtain an algebra morphism $B(\A) \rightarrow H^\bullet_{F\times \Cx}(pt) $.

Thus Lemma \ref{le:highestweight} shows the commutativity of the diagram
\begin{equation*}
\begin{tikzcd}
H^\bullet_{\tG \times \Cx}(pt) \arrow{r} \arrow{d} & H^\bullet_{F \times \Cx}(Y) \arrow{d} \\
B(\A) \arrow{r} &  H^\bullet_{F \times \Cx}(pt)
\end{tikzcd}
\end{equation*}
as both sides give $ \widehat \psi^* $.
Now specialize the equivariant parameters at a generic point $ \xi \in \mathfrak f \oplus \C $, as in Section \ref{section: central specialization}.  Then we obtain 
\begin{equation*}
\begin{tikzcd}
H^\bullet_{\tG \times \Cx}(pt) \otimes_{H^\bullet_{F \times \Cx}(pt)} \C_\xi   \arrow{r} \arrow{d} & H^\bullet_{\xi}(Y) \arrow{d} \\
B(\A_\xi) \arrow{r} &  H^\bullet_{\xi}(pt) = \C
\end{tikzcd}
\end{equation*}
where $ \A_\xi = \A \otimes_{H^\bullet_{F \times \Cx}(pt)} \C_\xi $. All algebra morphisms $ H^\bullet_\xi(Y) \rightarrow \C $ are given by connected components of $ Y^{F\times \C^\times} $.  

Now assume Kirwan surjectivity and surjectivity of (\ref{eq:HtoB}), so that $ \operatorname{Spec} H^\bullet_{\xi}(Y) $ and $ \operatorname{Spec} B(\A) $ are both subschemes of 
$$
\operatorname{Spec} H^\bullet_{\tG \times \Cx}(pt) \otimes_{H^\bullet_{F \times \Cx}(pt)} \C_\xi \cong (\ft + \xi) / W.
$$  
Then the above commutative square shows that all closed points of $ \operatorname{Spec} H^\bullet_{\xi}(Y) $ which come from fixed points $ (N \sslash_\chi G)^{F\times \C^\times} $, actually lie in $ \operatorname{Spec} B(\A) $. 

\subsection{Eigenbases for GT algebras}
Recall the Gelfand-Tsetlin subalgebra $ H^\bullet_{\tG\times\Cx}(pt) $ of $  \A$.  Under some special circumstances, we can find a basis for $ M_c $ consisting of eigenvectors for the action of the Gelfand-Tsetlin algebra.

To begin, recall the flavour exact sequence and restrict this to an exact sequence of maximal tori
\begin{gather*}
1 \rightarrow G \rightarrow \tG \rightarrow F \rightarrow 1\\
1 \rightarrow T \rightarrow \widetilde T \rightarrow F \rightarrow 1
\end{gather*}
Recall that $ H^\bullet_\tG(pt) \cong (\Sym \widetilde \ft)^W $.

Assume that we are in the setting of section \ref{se:Verma}. Recall, that we have  $ \psi : F \rightarrow \tG $ whose image is the stabilizer of $ c$.  Assume that $ \psi $ actually lands in the maximal torus $ \widetilde T $.  Assume that the fixed points of $ \psi(F) $ acting on $ \Gr_G $ coincides with $ \Gr_G^T $.  Finally, assume that $ M_c $ is free over $ H^\bullet_{F \times \Cx}(pt) $.

Let $ \Sp_{c,T} = \Sp_c \cap \Gr_G^T $.  This is the same thing as the Springer fibre for $ c$ in $ \Gr_T $.  

Let $ \bar \psi : \ff \rightarrow \widetilde \ft $ be the Lie algebra map derived from $ \psi $.  For $ \lambda \in \ft $ define 
\begin{align*} \bar \psi_\lambda : \ff \oplus \C &\rightarrow \widetilde \ft \oplus \C  \\
(u, a) &\mapsto (\bar \psi(u) + a(-\lambda - \tfrac{1}{2} \mu_0), a)
\end{align*}
We get a resulting algebra map, denoted $ \bar \psi_\lambda^* : H_{\tG \times \Cx}^\bullet(pt) \rightarrow H^\bullet_{F \times \Cx}(pt) $, by
$$
H_{\tG \times \Cx}^\bullet(pt) = (\Sym \tilde \ft^*)^W[\hbar] \rightarrow \C[\widetilde \ft \oplus \C] \rightarrow \C[\ff \oplus \C] = \Sym \ff^*[\hbar] = H^\bullet_{F \times \Cx}(pt)
$$ 

The following result shows that under the above assumptions, we can describe our modules using the homology of the fixed point set.  This result was also obtained by Garner-Kivinen \cite[Proposition 4.16]{GK}.
\begin{Theorem}
\label{thm: GT bases}
	We have a surjective map 
	$$
	H_\bullet^{L_c}(\Sp_{c,T}) \rightarrow H_\bullet^{L_c}(\Sp_c)
	$$
	which is an isomorphism for generic values of the flavour parameters.
	
	Let $ \lambda $ be a coweight of $ G $ such that $ z^\lambda \in \Sp_c$.
	  Then for $ x \in H_{\tG \times \Cx}^\bullet(pt)$, we have $ x \cdot [\{z^\lambda\}] = \bar \psi_\lambda^*(x) [\{z^\mu\}]$.
	
	In particular at a generic flavour parameter $\gamma \in \mathfrak f \oplus \C$, $ M \otimes_{\Sym \ff^*[\hbar]} \C$  has a basis of Gelfand-Tsetlin eigenvectors.
\end{Theorem}

\begin{proof}
 The first statement follows from equivariant formality.
	
	The proof of the second statement is similar to the proof of Lemma \ref{le:highestweight}.  We consider the diagram
	\begin{equation*}
	\begin{tikzcd}
	\{ h z^\lambda : h \in \ttGO\rtimes \Cx \} \ar[r] \ar[d] & X_c \ar[d] \\
	\{z^\lambda \} \ar[r] & \Sp_c 
\end{tikzcd}
\end{equation*}
where the vertical arrows are $ \ttGO \rtimes \Cx $ principal bundles.  On the space $\{ h z^\lambda  : h \in \ttGO \rtimes \Cx \} $, we have an action of $ \ttGO\rtimes \Cx $ by  left multiplication, and an action of $ \FO \rtimes \Cx $ by right multiplication (composed with $\psi$).

 Now the stabilizer of $ z^\lambda $ under this action is 
 $$ \{( \psi(r)s^{-\lambda - \tfrac{1}{2}\mu_0},s), (r, s) : r \in \FO, s \in \Cx \}$$ 
 The result then follows by the same argument as the proof of Lemma \ref{le:highestweight}.
\end{proof}

Note that the set $ \Sp_{c,T} $ is very easy to analyze.  Write $ c = \sum_{j=1}^m c_j $ where each $ c_j \in N_{\gamma_j} $ is a (non-zero) weight vector.  Then it is easy to see that
$$
\Sp_{c,T} = \{z^\mu : \langle \gamma_j, \mu \rangle \le 0 \text{ for $j = 1, \dots, m$ } \} 
$$

\begin{Example} \label{eg:sln2}
	As in Example \ref{eg:sln}, we consider the case $$ G = \prod_{i = 1}^{n-1} GL_i, \  N = \oplus_{i = 1}^{n-1} \Hom(\C^i, \C^{i+1}), \ F = (\Cx)^n $$  
	We choose $ \chi : G \rightarrow \Cx $ by $g = (g_i) \mapsto \prod_i \det(g_i)^{-1}$. 	In this case the $\chi$-stable points of $N $ are the injective homomorphisms and $ N\sslash_\chi G \cong Fl_n $, the variety of full flags in $ \C^n$.
	
	We let $ c \in N $ correspond to the standard inclusions $ \C^i \hookrightarrow \C^{i+1} $.  Then we see that the resulting map $$ \psi : F = (\Cx)^n \rightarrow \widetilde T = \prod_{i=1}^n (\Cx)^i 
	$$
	is given by $$ (t_1, \dots, t_n) \mapsto (t_1, (t_1, t_2), \dots, (t_1, \dots, t_n)) 
	$$
	In particular, $F$ projects surjectively onto each component torus $ (\Cx)^i \subset GL_i $.  Thus we see that $ \Gr_G^{\psi(F)} = \Gr_G^T $ does hold in this example.
	
	Using the lattice model of the affine Grassmannian for $ G $, we see that the Springer fibre is 
	$$
	\{ L_1 \subset \cdots \subset L_{n-1} \subset \cO^n : L_i \text{ is a $\cO$-lattice in $ \cK^i$ } \}
	$$
	
	Moreover, it is easy to analyze $ \Sp_{c,T}$ in this case and we find
	$$
	\Sp_{c,T} = \{(\mu^1, \dots, \mu^{n-1}) : \mu^{i-1}_j \le \mu^i_j \text{ for $ i = 1, \dots, n-1$ and $ 1 \le j < i $} \}
	$$
	In other words, these points are labelled by triangular arrays of numbers with inequalities down each column.  These are easily seen to be in bijection with Gelfand-Tsetlin patterns.  This matches the description of the fixed points of Laumon spaces in \cite[Section 2.2]{FFFR} (see Section \ref{se:Laumon}).
\end{Example}

\subsection{Case of closed orbit}
For this section, fix $ c_0 \in N $ and assume that the $ G $ orbit of $ c_0 $ is closed in $ N $.  We also assume that $ c_0 $ has trivial stabilizer in $ G $. The corresponding Springer fibre is just a point.

\begin{Lemma} 
The Springer fibre $ \Sp_{c_0} $ is the point $ [z^0]$.
\end{Lemma}

\begin{proof}
Consider the map $ G \rightarrow N $ given by $ g \mapsto g^{-1}c $.  Since the orbit has no stabilizer and is closed, this map is proper.  Thus, applying the valuative criterion for properness, as in the proof of Theorem \ref{th:Springerpoint1}, we obtain the desired conclusion.
\end{proof}
So $ M_{c_0} $ is a free rank 1 module over $ H^\bullet_{ L_{c_0}}(pt)$.  In particular, we obtain 1-dimensional modules for any specialization of $ \A $ which factors through $ H^\bullet_{F \times \Cx} \rightarrow H^\bullet_{ L_{c_0}}(pt) $.  In particular, the specialization $ \A_{-[\frac{1}{2}\mu_0], \hbar} $ has a module which is free of rank 1 over $ \C[\hbar]$, since $ \Cxz \subset L_{c_0} $.

Now, let us choose a cocharacter $ \mu : \Cx \rightarrow \tG $.  Then define $ c = z^{\mu} c_0 $. 

\begin{Proposition}
The Springer fibre $ \Sp_c $ has finitely many connected components and thus $H_\bullet^{L_c}(\Sp_c) $ is finitely generated over $ H^\bullet_{L_c}(pt) $. 
\end{Proposition}

\begin{proof}
For any character $ \chi : G \rightarrow \Cx$, we see that $ c_0 $ is $ \chi$-stable.  Thus by Theorem \ref{th:catO}, there exists $N_\sigma $ such that $ \Sp_c(\sigma) = \emptyset $ for $ \langle \sigma, \chi \rangle > N_\sigma $.

Thus those $ \sigma $ such that $ \Sp_c(\sigma) $ is non-empty are contained inside a bounded polytope and thus form a finite set.
\end{proof}

In particular, we obtain finite-dimensional modules for any specialization of $ \A $ which factors through $ H^\bullet_{F \times \Cx} \rightarrow H^\bullet_{L_c}(pt) $, such as $ \Am $.

\begin{Example}
\label{ex: closed orbit quiver}
	Fix three vector spaces $V, W_1, W_2 $ of the same dimension $n$.
	
	We consider $ G = GL(V)$ and $ N = \Hom(V, W_1) \oplus \Hom(W_2, V)$.  Moreover let $ F = T(W_1) \times T(W_2) $ be the product of the maximal tori acting on $ W_1 $ and $ W_2$.  In this case, the resulting Coulomb branch algebra is $ Y^{2n}_0 $, a truncated Yangian for $\mathfrak{sl}_2$, which is isomorphic to a 2-row $W$-algebra (see \cite[Theorem 4.3]{WWY}).  	
	
	Let $ c_0 = (A, B) $ be given by a pair of isomorphisms.  Then the $ G $ orbit of $ c_0 $ is closed and free and equals $ \{(A', B') : A'B' = AB \} $.
	
	Let us choose some $ \mu : \Cx \rightarrow T(W_2) $ and form $ c= z^{-\mu} c_0= (A, B z^{\mu}) $.  Then we find
	$$
	\Sp_c = \{ L \subset V \otimes K : z^\mu (V \otimes \OO) \subseteq L \subseteq V \otimes \OO \}
	$$
	This variety is empty, unless $ \mu_i \ge 0 $ for all $ i $.  When non-empty, it has connected components labelled by $ r = \dim V\otimes \OO / L$.  Each connected component is isomorphic to a ``big Spaltenstein variety''
	$$\Sp_{c}(r) \cong \{ 0 \subseteq U \subseteq \C^m : \dim \C^m / U = r, XU \subset U \}
	$$
	where $ X $ is a type $ \mu $ nilpotent operator on $\C^m $.  (To obtain this we identify $\C^m = V\otimes\OO / z^\mu V\otimes \OO$, and let $ U = L / z^\mu V \otimes \OO $ and let $ X = z $.) 
	
We thus obtain an action of $Y^{2n}_0$ on the homology of this Springer fibre.  We will analyse this module in greater detail in Section \ref{section: closed orbit quiver}.

\end{Example}

\section{Abelian Gauge Theories}

Let $\tG = (\Cx)^n$ acting on $N = \C^n$ and $M = T^*N$ in the obvious way and let 
\[1 \to G \rightarrow \tG \rightarrow  F \to 1\]
be an exact sequence of algebraic tori. For simplicity we assume that the integer matrices representing the induced maps of fundamental groups are all totally unimodular, i.e., that the determinant of every square submatrix is $0$, $1$, or $-1$.

\subsection{The hypertoric enveloping algebra}
When $G$ is a torus, the quantized Coulomb branch algebra $\A = H^{\tG_{\cO} \rtimes \Cx}_\bullet(R_{G,N})$ is known as the hypertoric enveloping algebra. The Gelfand-Tsetlin subalgebra is
\[
H^\bullet_{ \tG_{\cO}\rtimes \Cx}(pt) \cong \Sym \widetilde \fg^*[\hbar] = \C[\hbar, x_1, \ldots, x_n]
\]
where the $x_i$ are the weights of the $\tG$-action on $N$ and $\hbar$ corresponds to loop rotation. As shown in \cite[Section 4(ii)]{BFN1} the algebra $\A$ is free as both a left or right module over the Gelfand-Tsetlin subalgebra with basis given by monopole operators $r^{\lambda} = [R(\lambda)]$ for $\lambda \in \pi$. The multiplication is given by the relations
\[
[x_i, r^\lambda] = \lambda_i \hbar r^\lambda
\]
and
\[
r^{\sigma} r^{\lambda} = 
\prod_{\substack{\lambda_i \cdot \sigma_i < 0 \\ |\lambda_i| \geq |\sigma_i|}} [x_i]^{\sigma_i}
\,\,\, r^{\sigma+\lambda}
\prod_{\substack{\lambda_i \cdot \sigma_i < 0 \\ |\sigma_i| > |\lambda_i|}} [x_i]^{-\lambda_i}.
\]
where
\[ \lambda_i = \langle \bar x_i, \lambda \rangle \]
and
\[
[x_i]^n = \begin{cases} 
\prod_{j=1}^{n} (x_i - (j - \frac{1}{2})\hbar) & n \geq 0   \\
\prod_{j=1}^{|n|} (x_i + (j - \frac{1}{2}) \hbar) & n < 0
\end{cases}.
\]
Here $\bar{x_i} \in \fg $ denotes the restriction of $ x_i $ to $ \fg$.

\subsection{Springer Fibers} By Lemma \ref{le:orbitModule}, the modules associated to a Springer fiber $Sp_c$ depend only on the $\tG_{\cK}^{\cO}  \rtimes \Cx$-orbit of $c \in N_{\cK}$. Thus we will start by finding a list of orbit representatives.

As a warmup, notice that the $\tG$-orbits on $N$ are in bijection with subsets $S \subseteq \{1, \ldots, n\}$. In particular, we define representatives $ c_S \in N $ by
\[ (c_S)_i = \begin{cases} 
1, & \text{when $i \not \in S$} \\ 
0, & \text{when $i \in S$}.
\end{cases} \]
Define the stabilizer groups
\begin{align*}
\tG_S &= \text{Stab}_{\tG}(c_S)  & G_S &= \text{Stab}_{G}(c_S)
\end{align*}
and consider the restriction
\[1 \to G_S \rightarrow \tG_S \rightarrow  F_S \to 1\]
of our short exact sequence of tori. The structure of these groups is given by the following lemma.

\begin{Lemma}
 \begin{align*}
\tG_S &= (\Cx)^S  = \bigcap_{i \not\in S} \ker x_i & G_S &= \bigcap_{i \not\in S} \ker \bar{x}_i
\end{align*}
\end{Lemma}

Now we are ready for the main result of this section.

\begin{Proposition}\
The orbits of the $\tG_{\cK}^{\cO}$-action on $N_{\cK}$ are labeled by pairs $(S, [\mu])$ where $S \subseteq \{1, \ldots, n\}$ and $[\mu]: \Cx \to F$. In particular, each orbit has a representative of the form $c_{S, \mu} = z^\mu c_S$ where $\mu: \Cx \to \tG$ is a lift of $[\mu]: \Cx \to F$. The corresponding stabilizer group is
\[
L_{S,\mu} = \text{Stab}_{\tG_{\cK}^{\cO} \rtimes \Cx}(c_{S,\mu}) \cong (\tG_S)^{\cO}_{\cK} \rtimes \Cxm
\]
where
\[
\Cxm = z^\mu \Cxz z^{-\mu} = \{( s^{-(\mu + \frac{1}{2} \mu_0)}, s) \in  \tG^{\cO}_{\cK}  \rtimes \Cx \, | \, s \in \Cx \}.
\]
\end{Proposition}

\begin{proof}
The statement about orbit representatives is clear. To compute the stabilizer notice that the previous lemma implies that
\[
\text{Stab}_{\tG_{\cK}^{\cO}}(c_{S,\mu}) = (\tG_S)^{\cO}_{\cK}
\]
and that since $\Cx$ acts on $N_{\cK}$ by the combination of the weight $\frac{1}{2}$ action on $N$ and loop rotation we have
\[
s \cdot c_{S,\mu} = s^{\mu + \frac{1}{2}\mu_0} c_{S,\mu}
\]
for $s \in \Cx$. 
\end{proof}

It is easy to give a complete description of the Springer fibre $ \Sp_S[\mu] := \Sp_{c_{S,\mu}} $.  Note that since $ G$ is a torus, we have $ \Gr_G = \pi $.
\begin{Proposition} \label{pr:AbelianSpringer}
	We have
	$$
	\Sp_S[\mu] = \{ \sigma \in \pi : \sigma_i \le \mu_i \text{ for all } i \notin S \}
	$$
In particular, we see that $ \Sp_S[\mu] $ is the set of lattice points inside the polytope in $ \fg_{\mathbb R} $ defined by the linear functionals $ \bar x_i $, for $ i \notin S$, translated by the constants $ \mu_i $.
\end{Proposition}

\begin{proof}
	This follows immediately from the computation \[
	(z^{-\sigma} \cdot c_{S, \mu})_i = \begin{cases}
	z^{\mu_i-\sigma_i} & i \not \in S \\
	0 & i \in S
	\end{cases}
	\]
\end{proof}

Recall the auxiliary space
\[
X_{S}[\mu] := X_{c_{S,\mu}} = \{g \in \tG_{\cK}^{\cO} \rtimes \Cx \, : \, g \cdot c_{S,\mu} \in N_{\cO} \}
\]
Proposition \ref{pr:AbelianSpringer} tells us that
\[
X_{S}[\mu](\sigma) = \begin{cases} 
\tG^{\cO}_{\cK}(-\sigma) \rtimes \Cx & \text{if $\sigma_i \leq \mu_i$ for all $i \not \in S$,} \\
\emptyset & \text{else.}
\end{cases}
\]

For each subgroup $K \subseteq L_{S,\mu}$ one can form the generalized orbital variety
\[
V_{S, K}[\mu] = X_{S}[\mu] / K.
\]

By Proposition \ref{pr:ModuleExists}, one can form a module 
\[
M_{K, S}[\mu] = H^{K}_{\bullet}(\Sp_{S}[\mu]) \cong H^{-\bullet}_{\tG_{\cO} \rtimes \Cx}(V_{S, K}[\mu], \mathcal{F})
\]
and by Proposition \ref{pr:strongEquiv}, $M_{S, K}[\mu]$ is a strongly $\tG^!$-equivariant module whenever
\[
K \subseteq K_{S, \mu} :=  (L_{S,\mu} \cap (\tG_{\cO} \rtimes \Cx)) \cong  (\tG_S)_{\cO} \rtimes \Cxm.
\]
It is not hard to see that in this situation $M_{S,K}[\mu]$ is obtained from $M_{S, K_{S,\mu}}[\mu]$ by restriction of scalars. Thus in the following sections we will examine the cases when $K = L_{S,\mu}$ and $K = K_{S,\mu}$.

\subsection{Full symmetry group} Assume that $K = L_{S,\mu}$ and consider the quotient stack
\[
[\tG_{\cO} \rtimes \Cx \backslash \tG^{\cO}_{\cK} \rtimes \Cx / (\tG_S)_{\cK}^{\cO} \rtimes \Cxm].
\]
Using the exact sequence 
\[
1 \to \tG_S \to \tG \to \tG^S \to 1
\]
and the analogous sequence for $G$ we can see
\begin{align*}
[\tG_{\cO} \rtimes \Cx \backslash \tG^{\cO}_{\cK} \rtimes \Cx / (\tG_S)_{\cK}^{\cO} \rtimes \Cxm]  & \cong [\tG_{\cO} \rtimes \Cx  \backslash (\tG^S)^{\cO}_{\cK} \rtimes \Cx / \Cxm]  \\
&\cong  [(\tG_S)_{\cO} \rtimes \Cx \backslash \text{Gr}_{G^S} \rtimes \Cx /  \Cxm] \\
& \sim  [\tG_S \backslash \text{Gr}_{G^S} /  \Cxm] 
\end{align*}
where the second to last identification uses the splitting $\tG^S \cong \tG_{S^c} \subseteq \tG$ and the last is a Borel-Moore homology equivalence. Thus taking Borel-Moore homology gives us
\[
H^{\tG_S \times \Cxm}_{\bullet}(\text{Gr}_{G^S}) = \bigoplus_{[\sigma] \in \pi^S} R_S [\text{Gr}_{G^S}([\sigma])]
\]
where $\pi^S = \pi_1(G^S)$ and $R_S = H_{\tG_S \times \Cxm}^{\bullet}(pt) = \C[\hbar][x_i \, | \, i \in S]$.

In the last section we saw that $X_{S}[\mu]$ is a union of connected components of $\tG^{\cO}_{\cK}$. We claim that this implies that 
\[
[\tG_{\cO} \rtimes \Cx \backslash V_{S,K}[\mu] ] \cong [\tG_{\cO} \rtimes \Cx \backslash X_{S,K}[\mu] / (\tG_S)_{\cK}^{\cO} \rtimes \Cxm]
\]
is homology equivalent to a union of connected components of $[\tG_S \backslash \text{Gr}_{G^S} /  \Cxm]$. The components of the latter are labeled by $[\sigma] \in \pi^S$ and of the form
\[
[\tG_S \backslash \text{Gr}_{G^S}([-\sigma]) / \Cxm] \sim [\tG_{\cO} \rtimes \Cx \backslash \tG^{\cO}_{\cK}([-\sigma])  \rtimes \Cx / (\tG_S)_{\cK}^{\cO} \rtimes \Cxm]
\]
where on the right hand side we have collapsed the $\pi$-grading on $\tG^{\cO}_{\cK}$ to a $\pi^S$-grading. When $i \not \in S$ the function $[\sigma]_i = \sigma_i$ is well defined and hence our results from the last section tell us that $[\tG_{\cO} \rtimes \Cx \backslash V_{S,K}[\mu] ]$ consists of exactly those connected components where $[\sigma]_i \leq \mu_i$. Thus we have proved:

\begin{Proposition}
The module $M_{S, K}[\mu]$ is a free $R_S$-module generated by the classes
\[
\basissigmas = [V_{S, K}[\mu]([\sigma])] = [\text{Gr}_{G^S}([\sigma])] 
\]
where $[\sigma]_i \leq \mu_i $ for all $i \not \in S$.  For all other $ [\sigma] \in \pi^S $, we define $\basissigmas$ to be $0$.
\end{Proposition}

The first step to understanding the $\A$-action on $M_{S,K}[\mu]$ is understanding the action of the Gelfand-Tsetlin subalgebra $\C[\hbar, x_1, \ldots, x_n]$. Unfortunately the subgroup $K = L_{S,\mu}$ does not always satisfy the hypotheses from Propositions \ref{pr:weakEquiv} and \ref{pr:strongEquiv} and hence $M_{S, K}[\mu]$ is not always a weight module.

\begin{Proposition} \text{ }
\begin{enumerate}
\item The action of $R_S$ on $M_{S, K}[\mu]$ factors through the obvious inclusion of $R_S$ into the Gelfand-Tsetlin subalgebra. In particular we cannot simplify 
$ x_i \basissigmas $
when $i \in S$.
\item If $i \not \in S$ we have
$
x_i \basissigmas = ([\sigma]_i - \mu_i - \tfrac{1}{2}) \hbar \basissigmas.
$
\end{enumerate}
\end{Proposition}

\begin{proof} \text{ }
\begin{enumerate}
\item
For the $x_i$ with $i \in S$ this is essentially the observation that $\tG_S$ is acting through the inclusion into $ \tG_{\cO}\rtimes \Cx$. In particular, we are not re-writing these equivariant parameters in terms of equivariant parameters for $K=L_{S,m}$ as we do in a weight module. However for $\hbar$ one does have to argue as in the proof of Proposition \ref{pr:strongEquiv} and show that the line bundles associated to $\Cx$ and $\Cxm$ are equivariantly isomorphic.
 \item One can argue as in the proof of Proposition 5.10 to show one can rewrite $x_i$ for $i \not \in S$ in terms of the equivariant parameter for $\Cxm$, which we have identified with $\hbar \in R_S$. The map $\phi$ in Proposition  \ref{pr:strongEquiv} comes from the inclusion
 \[
 \Cxm \to (\tG)_{\cO}\rtimes \Cx \to  (\tG^S)_{\cO} \rtimes \Cx
\]
\end{enumerate}
\end{proof}

\begin{Proposition} \label{pr:abpres1} 
The action of the monopole operators $r^{\lambda} \in \A$ is given by

\begin{align*}
r^{\lambda} \basissigmas &= \prod_{\lambda_i > 0} [x_i]^{\lambda_i} \basislambdasigmas \\
	&= \left( \prod_{\substack{\lambda_i > 0 \\ i \in S}} \prod_{j=1}^{\lambda_i} (x_i - (j - \tfrac{1}{2}) \hbar) \right) \left( \prod_{\substack{\lambda_i > 0 \\ i \not\in S}} \prod_{j=1}^{\lambda_i} (([\lambda + \sigma]_i -\mu_i -j) \hbar) \right) \basislambdasigmas.
\end{align*}
\end{Proposition}

\begin{proof}
The diagram 	
\begin{equation*}
\begin{tikzcd}
	\sR \times V_{S, K}[\mu] \arrow[hook]{d}{i} & p^{-1}(\sR \times V_{S, K}[\mu[) \arrow{l}[swap]{\widetilde{p}} \arrow{r}{\widetilde{q}} \arrow[hook]{d}{j}& q(p^{-1}(\sR \times V_{S, K}[\mu])) \arrow{r}{\widetilde{m}}  & V_{S, K}[\mu] \\
	\sT \times V_{S, K}[\mu] \arrow[leftarrow]{r}{p}  & (\tG^{\cO}_{\cK} \rtimes \Cx) \times V_{S, K}[\mu] & &
\end{tikzcd}
\end{equation*}
defines the action of $\A$ on $M_{S,m}$. 

We would like to compute 
\[
	r^{\lambda} \basissigmas := (\tilde{m} \circ \tilde{q})_*(\tilde{p}^*([\sR(\lambda)] \times [V_{S,K}[\mu]([\sigma])])).
\]
At the level of sets it is clear that
\[
	V_{S,K}[\mu]([\lambda + \sigma])  = (\tilde{m} \circ \tilde{q})(\tilde{p}^{-1}(\sR(\lambda) \times V_{S,K}[\mu]([\sigma]))
\]
so this action must be compatible with the $\pi$-grading on $A$ and the $\pi^S$-grading on $M_{S,m}$. In particular we must have that
\[
	r^{\lambda} \basissigmas   = f(\hbar, x_1, \ldots, x_n) \basislambdasigmas
\]
for some element $f(\hbar, x_1, \ldots, x_n)$ of the Gelfand-Tsetlin subalgebra. Geometrically this prefactor is precisely the excess intersection of $R(\lambda) \times V_S([\sigma])$ and the image of $\tilde{p}$ inside $\sT \times V_S$. But by construction the image of $\tilde{p}$ is contained in $\sR(\lambda) \times V_S([\sigma])$ so the excess is
	\[
	ch_{\tG_{\cO} \rtimes \Cx}(\sT(\lambda) / \sR(\lambda)) = ch_{\Cx \ltimes \tG_{\cO}}(z^{-\lambda} N_{\cO} / (z^{-\lambda} N_{\cO} \cap N_{\cO})) = \prod_{\lambda_i > 0} [x_i]^{\lambda_i}.
	\]
Now using our computation of action of the Gelfand-Tsetlin subalgebra we see that

	\begin{align*}
	\prod_{\lambda_i > 0} [x_i]^{\lambda_i} \basislambdasigmas
	&= \prod_{\lambda_i > 0 } \prod_{j=1}^{\lambda_i} (x_i - (j - \tfrac{1}{2}) \hbar) \basislambdasigmas \\
	&= \left( \prod_{\substack{\lambda_i > 0 \\ i \in S}} \prod_{j=1}^{\lambda_i} (x_i - (j - \tfrac{1}{2}) \hbar) \right) \left( \prod_{\substack{\lambda_i > 0 \\ i \not\in S}} \prod_{j=1}^{\lambda_i} (([\lambda + \sigma]_i -\mu_i -j) \hbar) \right) \basislambdasigmas.
	\end{align*}

\end{proof}

\subsection{Restricted symmetry group}  Assume that $K = K_{S,\mu}$. 
Using similar reasoning to the last section we see that
\[
[\tG^{\cO} \rtimes \Cx \backslash V_{S,K}[\mu]]
\]
is a union of connected components of
\[
[\tG_{\cO} \rtimes \Cx  \backslash \tG^{\cO}_{\cK} \rtimes \C) / (\tG_S)_{\cO} \rtimes \Cxm] \cong  [\text{Gr}_{G}/ (\tG_S)_{\cO} \rtimes \Cxm].
\]
In particular, the module $M_{S,K}[\mu]$ is a free $U_S$-submodule of
\[
H^\bullet_{\tG_S \times \Cxm}(\text{Gr}_G)) \cong  \bigoplus_{[\sigma] \in \pi^S} U_S \cdot [\text{Gr}_{G}(\sigma)]
\]
where $U_S = H^{\bullet}_{\tG_S \times \Cxm}(pt) = \C[\hbar][y_i \, | \, i \in S]$.

\begin{Proposition} \label{pr:abpres2} \text{ }
\begin{enumerate}
\item
As a   $U_S$-module $M_{S,K}[\mu]$ is free and generated by the classes 
\[
\basissigma = [V_{S,K}[\mu](\sigma)] = [\text{Gr}_{G^S}(\sigma)]
\]
where $\sigma_i \leq \mu_i$ for all $i \not \in S$. For all other $ \sigma \in \pi $, we define $\basissigma$ to be $0 $.

\item The action of the Gelfand-Tsetlin subalgebra is given by 
\[
		x_i \basissigma  = (y_i  \delta_{i, S} + (\sigma_i - \mu_i - \tfrac{1}{2})\hbar) \basissigma.
\]
where $\delta_{i,S} = 1$ if $i \in S$ and $\delta_{i,S} = 0$ otherwise.

\item The action on the monopole operators is given by
 \begin{align*} 
		r^{\lambda} \basissigma &= \prod_{\lambda_i > 0} [x_i]^{\lambda_i} \basislambdasigma \\
		&= \prod_{\lambda_i > 0 } \prod_{j=1}^{\lambda_i} (y_i  \delta_{i, S} + (\lambda_i + \sigma_i - j - \mu_i)\hbar) \basislambdasigma
		\end{align*}

\end{enumerate}
\end{Proposition}

\begin{proof}
By Proposition \ref{pr:strongEquiv} the $\A$-module $M_{S,K}[\mu]$ is a generalized weight module with coefficients in $U_S$ and
	 $\phi_{S,m}: \C[\hbar, x_1, \ldots, x_n] \to \C[\hbar][y_i \, | \, i \in S]$ given by
	\begin{align*}
	\phi_{S,m}(\hbar) &= \hbar \\
	\phi_{S,m}(x_i) &= y_i \delta_{i, S} - (\mu_i + \tfrac{1}{2}) \hbar
	\end{align*}
The action of the monopole operators is computed exactly as in the last section.
\end{proof}

\subsection{Changing Lagrangians}
As mentioned in the introduction, one can construct additional modules by considering other $\tG$-equivariant Lagrangian splittings of $M$ than $M \cong N \oplus N^*$. In the abelian case these modules are especially easy to understand.

First note that $\tG$-equivariant Lagrangian splittings of $M$ are in bijection with sign vectors $\alpha \in \{+,- \}^n$. To see this let $v_1, \ldots v_n$ be the standard coordinates on $N$ and let $w_1, \ldots w_n$ be the dual coordinates on $N^*$.  We define $N_{\alpha}$ to be the vanishing set of all $v_j$ with $\alpha_j = +$ and all $w_j$ with $\alpha_j = -$. In particular, we can and will identify $\alpha$ with the cocharacter of $\tG$ that acts as inverse scaling on $N_{\alpha}$.

For each $t^\mu c_S \in (N_{\alpha})_{\cK}$ and choice of subgroup $K \subset L_{S,\mu}$ one gets a $\A_{G, N_{\alpha}}$-module $M^{\alpha}_{S, K}[\mu]$. Since the torus $\tG$ does not act on $N_{\alpha}$ in the standard way, this does not fit into the framework we have been using. However, we can fix this by inverting the $j$th $\Cx$-factor of $\tG$ whenever $\alpha_j = +$. Thus the presentations for $\A_{G, N_{\alpha}}$ and $M^{\alpha}_{S, K}[m]$ are the same as the ones we computed earlier with $x_i$ replaced by $-\alpha_i x_i$. 

Using the Fourier transform
\begin{align*}
\A_{G, N} &\cong \A_{G, N_{\alpha}} \\
r^{\lambda}&\mapsto  \prod_{\alpha_i \lambda_i < 0} (-\alpha_i)^{|\lambda_i|}  r^{\lambda} \\
\end{align*}
constructed in \cite[\S 4(v)]{BFN1} one can make $M^K_{\alpha, S}[\mu]$ into an $\A_{G,N}$-module. Using this procedure can find more general versions of Propositions \ref{pr:abpres1} and \ref{pr:abpres2}.  Recall that we have identified $\alpha$ with the cocharacter of $\tG$ that acts as inverse scaling on $N_{\alpha}$.

\begin{Proposition} \label{pr:abpresn} Suppose $K = (\tG_S)_{\cK}^{\cO} \rtimes \Cxma$. Then $M^{\alpha}_{S, K}[\mu]$ is a free $R_S$-module generated by the classes
\[
\basissigmas = [V^K_{\alpha, S}[\mu]([\sigma])] = 
[\text{Gr}_{G^S}(\sigma)] \text{ if $\alpha_i [\sigma]_i \geq \alpha_i [\mu]_i $ for all $i \not \in S$,}
\]
   For all other $ [\sigma] \in \pi^S $, we define $\basissigmas$ to be $0 $. 
   
   The $\A$-module structure is determined by the formulas
\begin{align*}
[x_i]^{\alpha_i} \basissigmas &=  \begin{cases} 
[x_i]^{\alpha_i} \basissigmas, & i \in S  \\
([\sigma]_i - [\mu]_i) \hbar \basissigmas, & i \not\in S
\end{cases}
\\
r^{\lambda} \basissigmas &= \prod_{\alpha_i \lambda_i < 0 } [x_i]^{\lambda_i} \basislambdasigmas \\
\end{align*}
\end{Proposition}

\begin{Proposition} \label{pr:abpresd}  Suppose $K = (\tG_S)_{\cO} \rtimes \Cxma$. Then $M^{\alpha}_{S, K}[\mu]$ is a free $U_S$-module generated by the classes
\[
\basissigma = [V^K_{\alpha, S}[\mu](\sigma)] = 
[\text{Gr}_{G^S}(\sigma)] \text{ if $\alpha_i \sigma_i \geq \alpha_i \mu_i $ for all $i \not \in S$,} 
\]  For all other $ \sigma \in \pi $, we define $\basissigma$ to be $0 $.

 The $\A$-module structure is determined by the formulas
\begin{align*}
[x_i]^{\alpha_i} \basissigma &= (y_i \delta_{i,S} + \sigma_i - \mu_i) \hbar \basissigma \\
r^{\lambda} \basissigma &= \prod_{\alpha_i \lambda_i < 0 } [x_i]^{\lambda_i} \basislambdasigma \\
&= \prod_{\alpha_i \lambda_i < 0 } \prod_{j=1}^{|\lambda_i|} (y_i \delta_{i,S} + \lambda_i + \sigma_i + \alpha_i j - \mu_i) \hbar \basislambdasigma
\end{align*}
\end{Proposition}

\subsection{Comparison with Physics}
Recall from the introduction that \cite{BDGH} introduced a boundary condition $\mathcal{B}$ for each triple $(L, c, H)$ or equivalently in the abelian case each triple $(\alpha, S, H)$. The boundary condition is Dirichlet when $H=1$, denoted $\mathcal{D}_{\alpha, S}$,  and Neumann when $H = G$, denoted $\mathcal{N}_{\alpha, S}$.

In \cite{BDGH} only modules with $H = G_S$ were considered. Furthermore, particular interested was paid to generic Dirichlet boundary conditions with $F_S = 1$ and exceptional Dirichlet boundary conditions with $F_S = F$. The following proposition gives a combinatorial classification of such boundary conditions which matches the one in \cite{BDGH}.

\begin{Proposition} \label{pr:basis} \text{ }
\begin{enumerate} 
\item $G_S = 1$ and $F_S= 1$ if and only if $S = \emptyset$.
\item $G_S = G$ if and only if $S = \{1, \ldots, n \}$.
\item $G_S=1$ and $F_S = F$ if and only if $\{ \bar{x}_i \, | \, i \not\in S \}$ is a basis for $\fg^*$
\end{enumerate}
\end{Proposition}

\begin{proof}
The first statement is obvious and the second follows immediately from our assumption that all the $G$ weights of $N$ are non-zero. To prove the third statement is enough to show that $\{\bar{x}_i | i \not\in S\}$ is a basis of $\fg^*$ if and only if $F \cong \tG_S$. To see this note that such an $S$ gives a right splitting $F \cong \tG_S \hookrightarrow \tG$ of the short exact sequence $1 \to G \rightarrow \tG \rightarrow  F \to 1$. Because the groups are tori, this is equivalent to the data of a left splitting $\tG \to G$ that induces an isomorphism between $(\Cx)^{S^c}$ and $G$. Because of our total unimodularity assumptions this occurs precisely when $\{\bar{x}_i \, | \, i \not \in S\}$ is a basis for $\ff^! = (\fg)^*$.
\end{proof}

Given a vortex line $[\mu]: \Cx \to F$ one can produce a module over the quantized Coulomb branch algebra $\A_{[\mu]}$ which will be also be denoted by $\mathcal{B}[\mu]$. In \cite{BDGH} presentations for these modules were given in the case that  $\mathcal{B}$ is Neumann, Generic Dirichlet, or Exceptional Dirichlet. These are identical to specialization of the presentations in Propositions \ref{pr:abpresn} and \ref{pr:abpresd}.

\begin{Proposition} \label{pr:bdyCond}
When $\mathcal{B} = (\alpha, S, H)$ is Neumann, Generic Dirichlet, or Exceptional Dirichlet we have
\[
\mathcal{B}[\mu] \cong (M^{\alpha}_{S, K}[\mu])_{[-\mu+\tfrac{1}{2}\alpha], \hbar}
\]
where $K = H_{\cK} \rtimes \Cxma$. 
\end{Proposition}

\begin{Remark}
Recall that a quantization parameter $\zeta \in \ff$ is integral if it is in the same $\tau = \pi_1(F)$-orbit as $[\frac{1}{2} \mu_0]$. In particular all of the quantization parameters $\zeta = [-\mu+\tfrac{1}{2}\alpha]$ are integral. One often wants to consider a fixed quantization so it is common to label modules as $\mathcal{B}(\zeta)$ instead of $\mathcal{B}[\mu]$.
\end{Remark}

\begin{Remark}
Recall that notion of weight module we use in this paper is stricter than the one from \cite{BLPW}: we require the Gelfand-Tsetlin subalgebra to act semi-simply instead of just locally finitely. There are extensions of weight modules contained in category $\mathcal{O}$ that we cannot produce directly using Springer fibres.
\end{Remark}

One puzzling fact noted in \cite{BDGH} is that the Higgs branch image of the Neumann boundary condition is a simple module, but the Coulomb branch image is free over the Gelfand-Tsetlin subalgebra, and hence cannot be in hypertoric Category $\mathcal{O}$. In particular, it is not a projective module as is predicted by symplectic duality \cite{BLPW1}. One can apply the geometric Jacquet functor to remedy this situation \cite{H}, but physical considerations imply that the module $\mathcal{N}_{\alpha}[\mu]$ is more natural: the Neumann boundary condition is engineered by coupling our $3$d theory to the $2$d $\mathcal N=(2,2)$ linear sigma model with target $N_{\alpha}$ using the flavour symmetry $G$.  The Coulomb branch module $\mathcal{N}_{\alpha}[\mu]$ encodes the Hori-Vafa superpotential, and hence also the quantum D-module for the toric variety $N_{\alpha} / G$. A similar observation is made by Teleman in \cite{T}.

\subsection{Hyperplane arrangements}
After specialization at $([-\mu + \tfrac{1}{2} \alpha],1)$ weight modules over $\A$\footnote{Here we mean ordinary weight modules over $\C$.}  have a nice description in terms of the combinatorics of a hyperplane arrangement in $(\ff^!)^* = \fg$. The starting point is the following proposition.
\begin{Proposition} \text{ }
\begin{enumerate}
\item The Gelfand-Tsetlin subalgebra of $\A_{([-\mu + \tfrac{1}{2}\alpha],1)}$ can be identified with functions on the fiber of the map
\[ (\ftg^!)^* = \ftg \xrightarrow{q} \ff = (\fg^!)^*\] above the point $[-\mu + \tfrac{1}{2} \alpha]$.
\item For each representative $\mu: \Cx \to \tG$ one has an isomorphism
\begin{align*}
(\ff^!)^* = \fg & \cong q^{-1}([-\mu + \tfrac{1}{2} \alpha]) \\
\xi &\mapsto \xi - \mu + \tfrac{1}{2} \alpha
\end{align*}
and under this isomorphism $x_i$ restricts to $\bar{x}_i - \mu_i +   \tfrac{1}{2} \alpha_i$.
\end{enumerate}
\end{Proposition}

Thus the vanishing set $V(x_i)$ of each affine function $x_i$ gives a hyperplane in $(\ff^!_{\mathbb R})^*$ which we will call $H_i$. By restricting a weight module over $\A_{[\mu],1}$ to the Gelfand-Tsetlin subalgebra one gets a collection of skyscraper sheaves on $(\ff^!)^*$. The structure of this module is often easy to describe using the hyperplanes $H_1, \ldots, H_n$ and Proposition \ref{pr:abpresd}.

\begin{Proposition} \label{pr:abpreshyp}
The specialization of 
\[
D_{\alpha, S}[\mu] = M^{\alpha}_{S, K}[\mu]
\]
where $K = \Cxma$ at $\hbar = 1$ is a module over $\A_{([-\mu + \tfrac{1}{2}\alpha],1)}$
\begin{enumerate}

\item As a module over the Gelfand-Tsetlin the specialization of $D_{\alpha, S}[\mu]$ is a sum of skyscraper sheaves supported on the intersection of the integer lattice $\pi \subseteq  s{(\ff^!)^* = \fg}$ and the polytope
\begin{align*}
P_{\alpha, S}[\mu] &= \{ \xi \in (\ff^!_{\mathbb R})^* \, | \,  \text{$\alpha_i x_i(\xi) > 0$ for  $i \not \in S$} \} \\
&= \{\xi \in (\ff^!_{\mathbb R})^* \, | \,  \text{$\alpha_i \xi_i - \alpha_i \mu_i + \tfrac{1}{2} > 0$ for  $i \not \in S$} \}
\end{align*}
Moreover, the skyscraper sheaf at $\sigma$ is generated by $\basissigma$.

\item Suppose that $\basissigma \neq 0$. Then $r^{\lambda} \basissigma = 0$ if and only if there exists an $i$ such that $\alpha_i x_i(\sigma) > 0$ and $\alpha_i x_i(\sigma + \lambda) < 0$.
\end{enumerate}
\end{Proposition}

\begin{proof}
To prove the second statement notice that
\begin{align*}
r^{\lambda} \basissigma &= \prod_{\alpha_i \lambda_i < 0 } \prod_{j=1}^{- \alpha_i \lambda_i} (\lambda_i + \sigma_i + \alpha_i j - m_i) \basislambdasigma
\end{align*}
can only vanish if $j = - \alpha_i (\lambda_i + \sigma_i - \mu_i)$.
\end{proof}

Choose a character $\chi$ of $G$. We will assume that that $\mu$ and $\chi$ are generic so that our arrangement is simple and $\chi$ is not constant on any $1$-dimensional flat. Under these assumptions category $\mathcal{O}$ is especially well behaved.  Using the results in this section one can match the modules $D_{\alpha, S}[\mu]$ to specific indecomposable weight modules in category $\cO$. 

\begin{Lemma} \label{lem:poly}
A module $D_{\alpha, S}[\mu]$ is contained in category $\cO$ for $\chi$ if and only if $\chi$ is bounded above on $P_{\alpha, S}[\mu]$. 
\end{Lemma}

\begin{proof}
By Proposition \ref{pr:strongEquiv}  $D_{\alpha, S}[\mu]$ is strongly $F^!$-equivariant. Up to an overall constant the element $\basissigma$ has degree $\chi(\sigma)$ for the $\mathbb{Z}$-grading from Definition \ref{def:catO}. Thus the degree is bounded iff and only if $\chi$ is bounded $P_{\alpha, S}[\mu]$.
\end{proof}

For each $\alpha$ we have a unique generic Dirichlet $D_{\alpha}[\mu] = D_{\alpha, \{1, \ldots, n\}}[\mu]$. 
\begin{Proposition}
Each $D_{\alpha}[\mu]$ is a simple $\A_{([-\mu + \tfrac{1}{2}\alpha], 1)}$-module.
\end{Proposition}
\begin{proof}
We claim that $D_{\alpha}[\mu]$ is generated by $\basissigma$ for any $\sigma \in (P_{\alpha}[\mu] \cap \pi)$. To see this, note that by Proposition \ref{pr:abpreshyp} we have that $r^{\lambda} \basissigma = 0$ only when $\sigma$ and $\lambda + \sigma$ are on opposite sides of some hyperplane. However, the intersection of any hyperplane $H_i$ with $P_{\alpha}[\mu]$ must be a face of $P_{\alpha}[\mu]$ or empty. Therefore if $\lambda + \sigma \in P_{\alpha}[\mu]$ we have $r^{\lambda} \basissigma = c \basislambdasigma$ for some non-zero scalar.
\end{proof}

\begin{Proposition} \label{pr:AbelianVerma}
Let $D_{\alpha, S}[\mu]$ be an exceptional Dirichlet module that is contained in category $\mathcal{O}$. If $\alpha_i x_i(\sigma^{\max}) < 0$ for all $i \in S$, then $D_{\alpha, S}[\mu]$ is a Verma module. 
\end{Proposition}

\begin{proof}
Since $F_S = F$ this is a Verma-like module in the sense of section \ref{se:Verma}. Since the set $\{\bar{x}_i \,|, i \not\in S \}$ is a basis for $\ff^! = \fg^*$ one can consider the dual basis $\{ \lambda^i \, | \, i \not\in S \}$ of $(\ff^!)^* = \fg$. By our total unimodularity assumptions these vectors actually give an integer basis for $\pi$. It is not hard to check that we have
\[
P_{\alpha, S}[\mu] \cap \pi = \sigma^{max} + \N \cdot \{ \alpha_i \lambda^i \, | \, i \not\in S\}
\]
where $(\sigma^{\max})_i = \mu_i$ for all $i \not \in S$ (i.e. $ \sigma^{\max} $ is the image of $ \mu $ under the right splitting $F \to \tilde{G}$ determined by $S$ in the proof of Proposition \ref{pr:basis}). 

We claim that $D_{\alpha, S}[\mu]$ is generated by $\basissigmamax$. To see this it suffices to show that $r^{\lambda} \basissigmamax \neq 0$ for all $\lambda \in \N \cdot \{ \alpha_i \lambda^i \, | \, i \not\in S\}$. However, this is immediate from Proposition \ref{pr:abpreshyp} since the only hyperplanes that can pass through the interior of $P_{\alpha, S}[\mu]$ are $H_i$ for $i \in S$.
\end{proof}

\begin{Remark}
The Hamiltonian reduction approach from Section \ref{section: QHR} works very well for abelian theories. An interested reader can implement the construction of modules in \cite{BLPW1} and \cite{H} using Springer fibers. In particular, one can use \cite[Proposition 4.3.1]{H} to show that the exceptional Dirichlet modules $D_{\alpha, S}[\mu]$ are exactly the twisted Verma modules.
\end{Remark}

\section{Quiver gauge theories}
Fix a quiver with vertex set $ I $ and edge set $ Q $.  We assume that both $ I $ and $ Q$ are finite, but we allow loops and multiple edges.

\subsection{Stable points}

Choose $ I $-graded vector spaces $ V, W $ with $ \dim V_i = v_i, \dim W_i = w_i $.  We define
$$ G = \prod_{i \in I} GL(V_i), \   N = \bigoplus_{(i,j) \in Q} \Hom(V_i, V_j) \oplus  \bigoplus_{i \in I} \Hom(W_i, V_i), \  F = \prod_{i\in I} (\Cx)^{w_i} $$
Then $ G $ acts on $ N $ by acting on the vector spaces $ V_i $ and $ F $ acts on the vector spaces $ W_i $.  For the remainder of this section, fix $ \chi : G \rightarrow \Cx $ to be given by $ \chi(g) = \prod \det(g_i) $.

The following result is well-known, see \cite[Section 3.ii]{NakKM}, \cite[Corollary 5.1.9]{Ginzburg}:
\begin{Lemma}\label{le:stabquiver}  Let $ c = ((c_h), (c_i)) \in N $.  The following are equivalent.
	\begin{enumerate}
		\item $ c $ is $\chi$-stable.
		\item $ c $ is $\chi$-semistable.
		\item There are no $I$-graded subspaces $ V' \subsetneq V $ containing the image of $ (c_i) $ and invariant under $ (c_h) $.
	\end{enumerate}

Moreover any such $c$ has trivial stabilizer in $ G $.
\end{Lemma}
Let  $ c = ((c_h), (c_i)) \in N^s $ be a $\chi$-stable point.   Then $ (c_h) $ defines a $ \C Q$-module structure  structure on the vector space $ V $.  By the above Lemma \ref{le:stabquiver}, we have a linear map $ (c_i) : W\rightarrow V  $ whose image generates $ V $ as a $\C Q$-module (so $ W $ is called a \textbf{framing} of $ V$). 

\subsection{Description of the Springer fibre} \label{se:SpringerQuiver}
  We have the standard interpretation of the affine Grassmannian of $ GL_G $ using lattices.

\begin{Proposition} \label{prop:LatticeModel}
	There is an isomorphism
	$$
	\Gr_G \cong \{  L = \oplus_{i \in I} L_i  \subset V \otimes \cK :  L_i  \text{ is an $\mathcal O$-lattice in $ V_i \otimes \cK $} \}
	$$
	given by $ g \mapsto g V \otimes \cO $.
\end{Proposition}

Let $ c \in N $ be $ \chi$-stable.  The main result of this section describes the Springer fibre $ \Sp_c $ in these terms.  As before, $ c $ gives a $\C Q$-module structure on $ V $, which gives a $ \C Q \otimes \cK $-module structure on $V\otimes \cK$.
\begin{Theorem} \label{th:SpQuiver}
	 Let $ c \in N $ be $ \chi$-stable.  Under the isomorphism from Proposition \ref{prop:LatticeModel}, we have an identification
	 $$
	\Sp_c = \{ L \in \Gr : V \otimes \cO \subseteq L, \text{ and $L $ is a $ \C Q $-submodule of $V \otimes \cK$ }\}
	$$
\end{Theorem}
Recall that $ c$ consists of two pieces of data; the $\C Q$-module structure $ (c_h) $ and the framing $ (c_i) $.  This Theorem shows that the Springer fibre only depends on the underlying $ \C Q$-module.  

\begin{proof}
	Write $ gV \otimes \cO = \oplus_i L_i $. The condition $  (c_h) \in g\Hom_Q(V, V)_\OO $ is equivalent to $ \oplus L_i $ being a $\C Q$-submodule of $ V((z)) $.
	
	On the other hand, the condition that $ (c_i) \in g\Hom_I(W, V)_\OO $ implies that $ L_i $ contains the image of $ W_i \otimes \cO $ for each $ i $.  Thus, $ L_i \cap V_i \otimes \cO $ contains the image of $ W_i \otimes \cO $.  Consider the subspace
	$$
	V' := L \cap V \otimes \cO / L \cap z V \otimes \cO \subset V \otimes \cO / z V\otimes \cO = V
	$$
	We see that $ V' $ contains the image of $ W $ and is invariant under the $\C Q $-module structure.  So by Lemma \ref{le:stabquiver}, we have $ V' = V$.  From this, we conclude that $ V\otimes \cO \subseteq L $ as desired.
\end{proof}

\begin{Remark} \label{re:filt}
Let $ \Gr^+ := \{ L \in \Gr : V \otimes \cO \subseteq L \} $; we refer to this locus as the positive part of the affine Grassmannian.  It is naturally filtered by $$ \Gr^{(k)} :=  \{ L \in \Gr : V \otimes \cO \subseteq L \subseteq z^{-k} V\otimes \cO \} $$
for $ k \in \N $.

The theorem implies that $ \Sp_c $ is contained in $ \Gr^+ $ and thus it is reasonable to consider the truncated Springer fibres $ \Sp_c^{(k)} := \Sp_c \cap \Gr^{(k)} $.
These form an exaustive filtration of $ \Sp_c $ (since for any fixed $\sigma$, for large enough $ k $, $ \Sp_c(\sigma) \subset \Sp_c^{(k)} $ by Theorem \ref{th:finitedim}) and we get a map $ H_\bullet(\Sp_c^{(k)}) \rightarrow H_\bullet(\Sp_c) $ for each $ k $.  In what follows, we will assume that this map is injective and thus we have a filtration on the vector space $ H_\bullet(\Sp_c) $.

By forming the quotient $ L/ V \otimes \cO $, we see that $\Sp_c^{(k)} $ is the variety of all $ \C Q \otimes \C[z] $ submodules of $ V \otimes \C[z] / z^k$.  (In particular, when $ k =1 $, we just get the variety of all $\C Q$-submodules of $ V $.)

\end{Remark}

\subsection{Generalized affine Grassmannian slices} \label{se:GeomSlices}
Assume now that $ Q $ is an orientation of a finite-type simply-laced Dynkin diagram.  There is an associated simple Lie algebra $ \fg_Q $ and we let $ G_Q $ be the corresponding group of adjoint type.

Let $ \lambda, \mu $ be the coweights of $ G_Q $, defined by  $$ \lambda = \sum w_i \omega_i^\vee, \ \lambda - \mu = \sum v_i \alpha_i^\vee, $$ where $\{\varpi_i^\vee\}$ are the fundamental coweights of $G_Q$ and $\{\alpha_i^\vee\}$ the simple coroots.

The Coulomb branch and its quantization admit descriptions using the affine Grassmannian of $ G_Q $.  Let $ U_Q $ denote the unipotent radical of a Borel in $ G_Q $ and let $ U^-_Q $ be the unipotent radical of the opposite Borel.

We will need to recall certain subvarieties of the affine Grassmannian $\Gr_{G_Q}$ of $ G_Q $ and related objects.  In \cite{BFN2}, a certain subvariety $ \cW_\mu \subset G_Q((z^{-1})) $ was defined (we will not recall its definition here) and also the generalized affine Grassmannian slice $\cW^\lambda_\mu := \overline{G_Q[z] z^\lambda G_Q[z]} \cap \cW_\mu $.  Moreover, in \cite[Section B(viii)]{BFN2}, we defined the truncated shifted Yangian $Y^\lambda_\mu $ (this version of $ Y^\lambda_\mu $ contains a central subalgebra isomorphic to $ H^\bullet_F(pt) $).

\begin{Theorem}
We have isomorphisms
\begin{enumerate}
	\item $\A_0(G,N) \cong \C[\cW^\lambda_\mu] $
	\item $\A(G,N) \cong Y^\lambda_\mu $
\end{enumerate}
These isomorphisms are compatible with the actions of the torus $ F^! \cong T_Q $.
\end{Theorem}

\begin{proof}
	The first part is \cite[Theorem 3.10]{BFN2} and the second part is \cite[Theorem B.18]{BFN2} along with \cite{Alex}.
\end{proof}

Recall also the semi-infinite orbits $ S_{\pm}^\mu := U^\pm_Q((z)) z^\mu \subset \Gr_{G_Q}$.  Recall also from Remark \ref{re:attracting}, the definition of the attracting locus $ \Spec \A_0 / I_+ \subset \Spec \A_0(G,N) $ inside the Coulomb branch.  The following result is due to Krylov \cite[Theorem 3.1(1)]{Krylov}.

\begin{Theorem}
	Under the above isomorphism, the attracting locus is equal to $ \overline{\Gr_{G_Q}^\lambda} \cap S_-^\mu $.
\end{Theorem}

Using Theorem \ref{th:catO} and Remark \ref{re:attracting}, we deduce the following results.
\begin{Corollary} \label{co:catOQuiver}
	Let $ ((c_h), (c_i)) $ be a framed $\C Q$-module structure on $ V $.  Then, 
	\begin{enumerate}
		\item $ H_\bullet(\Sp_c) $ is a $\C[\cW^\lambda_\mu]$-module set theoretically supported on $ \overline{\Gr_{G_Q}^\lambda} \cap S_-^\mu $.
		\item $ H_\bullet^{\eqstab}(\Sp_c) $ is a $ Y^\lambda_\mu $-module and lies in category $ \mathcal O$.
	\end{enumerate}
\end{Corollary}

Let $ d \in \N^I \subset \Z^I = \pi $. Then by Theorem \ref{th:SpQuiver},
$$
\Sp_c(d) = \{ V \otimes \cO \subseteq L \subset V \otimes \cK : \dim L_i / V_i \otimes \cO = d_i \text{ for all $ i \in I $} \}
$$
Thus we can rephrase Corollary \ref{co:catOQuiver}.(1) to say that there exists a quasi-coherent sheaf $ \mathcal F $ on $ \Gr_{G_Q} $ such that $$
\Gamma(\Gr_{G_Q}, \mathcal F)_d = H_\bullet(\Sp_c(d))$$
where the left hand side denotes the weight space for the action of $ T_Q $.

\begin{Remark}
In \cite[Conjecture 12.5]{BKK}, the second author and collaborators, motivated by the theory of MV polytopes and biperfect bases, conjectured that for each preprojective algebra module $V $ there is a quasi-coherent sheaf on the affine Grassmannian whose sections are related to cohomology of spaces of submodules of $ V \otimes \cO $.  Corollary \ref{co:catOQuiver}.(1) proves a weak form of this conjecture for those $ V$ which come from a $ \C Q$-module.
\end{Remark}

\begin{Remark}
Let $ \mathcal R^+_{G,N},  \mathcal R^{(k)}_{G,N} $ denote the preimages of $ \Gr^+_G, \Gr^{(k)}_G $ in $ \mathcal R_{G,N}$ (see Remark \ref{re:filt}).  Then $ H_\bullet^{\GO}(\mathcal R^+_{G,N}) $ is a subalgebra of $ \A_0 $ and acts on $ H_\bullet(\Sp_c) $.  Moreover $ H_\bullet^{\GO}(\mathcal R^{(k)}_{G,N}) $ provides a filtration of $  H_\bullet^{\GO}(\mathcal R^+_{G,N}) $ compatible with the filtration $ H_\bullet(\Sp_c^{(k)}) $ of $ H_\bullet(\Sp_c) $.  Thus we see that $ \Rees H_\bullet(\Sp_c) $ defines a coherent sheaf on $\Proj \Rees H_\bullet^{\GO}(\mathcal R^+_{G,N}) $.  We believe that $\Proj \Rees H_\bullet^{\GO}(\mathcal R^+_{G,N}) $ should be isomorphic to the compactified Zastava space (see \cite[Remark 3.7]{BFN2} for a related result) and thus  $ \Rees H_\bullet(\Sp_c) $ defines a coherent sheaf supported on the closure of the intersection of opposite semi-infinite orbits $ \overline{ S^{\lambda - \mu} \cap S_-^0}$.  This would be bring us to the full conjecture from \cite{BKK} (again only for those preprojective algebra modules which come from a $ \C Q$-module).

\end{Remark}

\subsection{Highest weights for truncated shifted Yangians}
\label{sec: hwtftsy}

The papers \cite{moncrystals}, \cite{klrpaper} study category $\OO$ for the truncated shifted Yangian $Y_\mu^\lambda$, and combinatorial relationships with quiver varieties.  We will relate this combinatorics to our BFN Springer fibres, based on the discussion in Section \ref{se:Hikita}.

For each node $i \in I$, choose integers $r_{i,1},\ldots, r_{i,w_i}$.  We organize them into multisets $\mathbf{R}_i = \left\{ r_{i,1},\ldots,r_{i,w_i}\right\}$, and write $\mathbf{R} = (\mathbf{R}_i)_{i\in I}$.  We also consider the cocharacters $\rho_i:\C^\times \rightarrow (\C^\times)^{w_i}$, defined by $s \mapsto (s^{r_{i,1}},\ldots, s^{r_{i,w_i}})$.  Now, consider the homomorphism
\begin{equation}
\label{eq: moncrystals 1}
\C^\times  \longrightarrow F \times \C^\times, \qquad s \longmapsto \big( (s^{-1} \rho_i(s))_{i\in I}, s^2 \big)
\end{equation}
There is an induced homomorphism $H_{F\times \C^\times}^\bullet(pt) \rightarrow H_{\C^\times}^\bullet(pt) = \C[x]$.  Setting $x = 1$, we obtain a homomorphism $\xi: H_{F\times \C^\times}^\bullet(pt) \rightarrow \C$.  The specialized algebra $\A_\xi = \A \otimes_{H_{F\times \C^\times}^\bullet(pt)} \C$ is precisely the truncated shifted Yangian $Y_\mu^\lambda(\mathbf{R})$ studied in \cite[Section 4]{klrpaper}.

Via the embedding (\ref{eq: moncrystals 1}), we obtain an action of $\C^\times$ on $N \sslash_\chi G$, and on the Nakajima quiver variety $Y = T^\ast N \sssslash_\chi G$ more generally.  In \cite[Section 8.3]{moncrystals} we conjectured the following weak form of the Hikita-Nakajima Conjecture \ref{hikita conjecture}, which we later proved:

\begin{Theorem}[\mbox{\cite[Theorem 1.5]{klrpaper}}]
The surjective Kirwan maps $ B \big( Y_\mu^\lambda(\mathbf{R})\big) \leftarrow H_G^\bullet(pt) \rightarrow H^\bullet( Y^{\C^\times} )$ induce bijections on maximal ideals:
$$
\operatorname{MaxSpec} B\big( Y_\mu^\lambda(\mathbf{R}) \big)  \ \longleftrightarrow \ \pi_0( Y^{\C^\times} )
$$
\end{Theorem}

Given $[c] \in (N\sslash G)^{C^\times}$,  we can produce a ``Verma-like" module $M_c$ for $Y_\mu^\lambda(\mathbf{R})$, by the construction of section \ref{se:Verma}.  By the explanation in section \ref{se:Hikita}, this gives us a geometric realization of part of the bijection in Theorem 8.10.  By considering other $G$--invariant Lagrangians $L \subset T^\ast N$ and repeating the same construction, we can obtain the following subset of $\pi_0( Y^{\C^\times})$:
\begin{equation}
\label{eq: achievable highest weights}
\bigcup_{L} \operatorname{Im} \left( \pi_0\big( (L\sslash_\chi G)^{\C^\times} \big) \longrightarrow \pi_0 \big( Y^{\C^\times} \big) \right),
\end{equation}
There are interesting cases where (\ref{eq: achievable highest weights}) is equal to all of $\pi_0( Y^{\C^\times})$:
\begin{Proposition}
\label{prop: W algebras}
Fix integers $1 \leq m_1 < \ldots < m_{n-1} < M $, and consider the following $A_{n-1}$--quiver gauge theory:
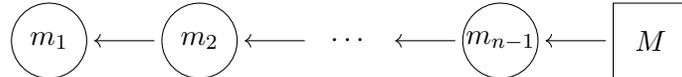
\begin{equation*}
\begin{tikzpicture}[scale=0.5]
\draw  (0,0) circle [radius=1];
\draw (4,0) circle[radius=1];
\draw (12,0) circle [radius=1];
\draw [->]  (2.8,0) -- (1.2,0);
\draw [->] (6.8,0) -- (5.2,0);
\draw [->] (10.8,0) -- (9.2,0);
\draw [->] (14.8,0) -- (13.2,0);
\node at (0,0) {$m_1$};
\node at (4,0) {$m_2$};
\node at (12,0) {$m_{n-1}$};
\node at (8,0) {$\cdots$};
\draw (15,-1) rectangle (17,1);
\node at (16,0) {$M$};
\end{tikzpicture}
\end{equation*}
Then $\pi_0\big( (N \sslash_\chi G)^{\C^\times}\big) \rightarrow \pi_0( Y^{\C^\times})$ is surjective, where $N$ is the Lagrangian given by the above orientation.  
\end{Proposition}
For the above quiver data, there are isomorphisms $\A_\gamma \cong Y_\mu^{M \varpi_1^\vee}(\mathbf{R}) \cong  W(\pi)_{\mathbf{R}}$, where $W(\pi)_{\mathbf{R}}$ is a  finite W-algebra for $\mathfrak{gl}_M$ (or rather, its central quotient corresponding to $\mathbf{R}$) by \cite[Theorem 4.3(a)]{WWY}.

\begin{Remark}
Note that in this section we continue to work with stability condition $\chi(g) = \prod_i \det(g_i)$.  With our conventions, this corresponds to the ``usual'' highest weight theory for the Yangian, where elements $E_i^{(p)}$ act locally nilpotently.  Meanwhile, taking the (more ubiquitous) opposite stability condition and the opposite orientation of the quiver, we obtain the lowest weight theory for the Yangian.
\end{Remark}

\begin{proof}[Proof of Proposition \ref{prop: W algebras}]
$N\sslash_\chi G$ is a partial flag variety (in $\C^M$), and $Y$ is its cotangent bundle.  Every component of $Y^{\C^\times}$ meets the base $N \sslash_\chi G$, which proves the claim. 
\end{proof}

\begin{Example}
In general (\ref{eq: achievable highest weights}) is a proper subset of $\pi_0(Y^{\C^\times})$, and it can even be empty. Consider the Nakajima quiver variety of type $A_3$, for the dimension vectors $w = (0,1,0)$ and $v= (1,2,1)$.  This quiver variety $Y$ consists of a single point, corresponding to the following $\chi$-stable framed representation:
$$
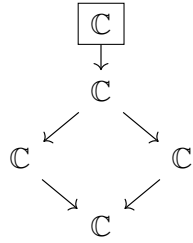
\begin{tikzcd}[sep=small]
& |[draw=black, rectangle]| \C \ar[d] & \\
& \C \ar[dl] \ar[dr] & \\
\C \ar[dr] & & \C \ar[dl] \\
& \C & 
\end{tikzcd}
$$
The quotient $L \sslash_\chi G$ is empty for all Lagrangians $L$, so the set (\ref{eq: achievable highest weights}) is empty.  Correspondingly, there is a unique simple module in category $\OO$ for the algebra $Y_{-\varpi_2^\vee}^{\varpi_2^\vee}(\mathbf{R})$, but we cannot realize it (or anything else in $\OO$) using BFN Springer fibres. 
\end{Example}

\subsection{A closed orbit example}
\label{section: closed orbit quiver}

In this section we revisit Example \ref{ex: closed orbit quiver}, where the Springer fibre had components 
$$
\Sp_c(r)  \cong \{ 0 \subseteq U \subseteq \C^m : \dim \C^m / U = r, \ XU \subset U \},
$$ 
with $X$ is a nilpotent operator on $\C^m$ with Jordan type $\mu$.  We have representations of the truncated Yangian $ Y_0^{2n}(\ssl_2)$ on the equivariant homology of these big Spaltenstein varieties, which we can try to relate the well-known theory of finite-dimensional modules for $Y(\mathfrak{sl}_2)$.

We keep our notation from Example \ref{ex: closed orbit quiver}, but for simplicity we identify $V = W_1 =W_2 = \C^n$ and $ T(W_1) = T(W_2)$ with the diagonal maximal torus $T \subset G = GL(n)$.  Following \cite[Appendix B]{BFN2}, we denote $H_{T(W_1)}^\bullet(pt) = \C[z_1,\ldots, z_n]$, $H_{T(W_2)}^\bullet(pt) = \C[z_{n+1},\ldots,z_{2n}]$, and $H_G^\bullet(pt) = \C[w_1,\ldots, w_n]^{S_n}$.

Take $c_0 = (I, I)$ to consist of identity matrices.  	 Then the stabilizer $\eqstab$ is the image of the embedding
		\begin{align*}
		T_\OO \rtimes \C^\times & \longrightarrow   (GL(n)_\OO \times T_\OO\times T_\OO) \rtimes \C^\times = \ttGO \rtimes  C^\times, \\
		(t, s) & \longmapsto  (s^{1/2} t, t, s^{\mu+1} t, s)
	\end{align*}
Recall the surjective homomorphism $\overline{\Phi}_0^{2n}:Y_\hbar(\mathfrak{sl}_2)[z_1,\ldots,z_{2n}] \twoheadrightarrow \A$ from \cite[Theorem B.18]{BFN2}. The generators $E^{(p)}$ map to homology classes over the Schubert variety for the coweight $\varpi_1^\ast = (0,...,0,-1)$, and thus land in degree $\A(1) \subset \A$. The maximal degree component $\Sp_c(m)$ therefore gives a highest weight space of $M_c$. Note that $\Sp_c(m)$ consists of a single point, corresponding to the subspace $U = 0$ in $\C^m$. 	
  
Similarly to the proof of Theorem \ref{thm: GT bases}, we find that $H_{\ttGO\rtimes \C^\times}^\bullet(pt)$ acts on $H^{\eqstab}_\bullet( \Sp_c(m)) \cong H_{\eqstab}^\bullet(pt)$ via the homomorphism
\begin{align*}
H_{\ttGO\rtimes \C^\times}^\bullet(pt) & \longrightarrow H_{\eqstab}^\bullet(pt), \\
z_1,\ldots,z_n, \hbar & \longmapsto z_1,\ldots,z_n, \hbar\\
e_i(w_1,\ldots,w_n) & \longmapsto e_i\big(z_1+(\mu_1+\tfrac{1}{2})\hbar,\ldots,z_n+(\mu_n+\tfrac{1}{2})\hbar), \\
z_{n+1},\ldots, z_{2n} & \longmapsto z_1+(\mu_1+1)\hbar, \ldots, z_n+(\mu_n+1)\hbar
\end{align*} 
Combining the above with \cite[(B.14)]{BFN2}, we find that the Yangian's generating series $H(u)$ acts on the highest weight space by the rational function
\begin{equation}
\label{eq: action of H}
\overline{\Phi}_0^{2n}\big(H(u)\big) = \frac{\prod_{i=1}^{2n}(u-z_i - \tfrac{1}{2}\hbar)}{\prod_{i=1}^n(u-w_i) (u-w_i - \hbar)} \longmapsto \prod_{i=1}^n \frac{u-z_i - \tfrac{1}{2}\hbar}{u-z_i - (\mu_i+\tfrac{1}{2})\hbar}
\end{equation}

Suppose that we choose complex numbers $r_1,\ldots,r_n$, and we specialize our module $M_c$ at $\hbar \mapsto 1$ and $z_i \mapsto r_i - \tfrac{1}{2}$ for $i=1,\ldots,n$. This induces a map $\xi: H_{F\times \Cx}^\bullet(pt) \rightarrow \C$, and we obtain a module $M_{c, r_1,\ldots,r_n}$ over the specialized algebra $\A_\xi$.

\begin{Remark}
$\A_\xi \cong Y_0^{2n}(\mathbf{R})$ with multiset $\mathbf{R} = \left\{r_1,\ldots,r_n, r_1+\mu_1+1,\ldots,r_n+\mu_n+1\right\}$.
\end{Remark}
By (\ref{eq: action of H}), after specialization  the series $\overline{\Phi}_0^{2n}\big(H(u)\big)$ acts on the highest weight space by 
$$
\prod_{i=1}^n \frac{u-r_i}{u-r_i - \mu_i} = \prod_{i=1}^n \frac{(u-r_i) (u-r_i-1)\cdots (u-r_i -\mu_i+1)}{(u-r_i-1)\cdots(u-r_i-\mu_i)}  \stackrel{def}{=}\frac{P(u+1)}{P(u)}
$$
The polynomial $P(u)$ is called the {\bf Drinfeld polynomial} of $M_{c, r_1,\ldots,r_n}$.  Following \cite[Definition 3.3.5]{Molev} we say that its root strings 
\begin{equation}
\label{eq: root strings}
S_i = \left\{r_i+1,r_i+2,\ldots,r_i+\mu_i\right\}
\end{equation} 
are in {\em general position} if $S_i \cap S_j \neq \emptyset$ implies $S_i \subset S_j$ or $S_j \subset S_i$.  

\begin{Proposition}
If the $S_i$ are in general position then $M_{c, r_1,\ldots,r_n}$ is an irreducible module for $Y(\mathfrak{sl}_2)$, and thus for its quotient $Y_0^{2n}(\mathbf{R})$. It is isomorphic to a tensor product of evaluation modules for $Y(\mathfrak{sl}_2)$, with underyling $\mathfrak{sl}_2$ module
$$
V(\mu_1)\otimes \cdots \otimes V(\mu_n)
$$

Moreover, any finite-dimenisonal irreducible module for $Y(\mathfrak{sl}_2)$ can be realized as $M_{c, r_1,\ldots,r_n}$ for appropriate $n, \mu$ and $r_1,\ldots,r_n$.
\end{Proposition}
\begin{proof}
$M_{c,r_1,\ldots,r_n}$ contains a highest weight vector $[\Sp_c(m)]$, with Drinfeld polynomial $P(u)$ as above.  There is a corresponding simple $Y(\mathfrak{sl}_2)$ module $L_P$, with the same highest weight, which must appear in the composition series of $M_{c,r_1,\ldots,r_n}$.  Now, by the odd cohomology vanishing proven below in Lemma \ref{lemma: odd coh van}, we can compute dimensions by counting torus fixed points: 
$$
\dim_\C M_{c,r_1,\ldots,r_n} = \chi ( \Sp_c) = \chi( \Sp_c^T) = \# \Sp_c^T = \prod_i (\mu_i+1)
$$ 
By \cite[Corollary 3.3.6]{Molev}, if the strings $S_i$ are in general position then $\dim_\C L_P$ is also $\prod_i (\mu_i+1)$, and $L_P$ is a tensor product of evaluation modules of the type claimed. This equality of dimensions implies that $L_P \cong M_{c,r_1,\ldots,r_n}$, proving the first claim.

For the second claim, recall that finite-dimensional irreducible $Y(\mathfrak{sl}_2)$ modules $L_P$ are in bijection with (monic) Drinfeld polynomials $P(u)$.  The roots of $P(u)$ can be factored uniquely into root strings $S_1,\ldots, S_n$ in general position, for some $n$ \cite[Prop 3.3.7]{Molev}.  We can choose $\mu_1,\ldots,\mu_n\geq 0$ and $r_1,\ldots,r_n$ to write these strings in the form (\ref{eq: root strings}).  This proves the claim.
\end{proof}

	\begin{Lemma}
	\label{lemma: odd coh van}
	$H^{i}( \Sp_c(r)) = 0$ for $i$ odd.  In particular, $\Sp_c(r)$ is equivariantly formal for $\eqstab$.
	\end{Lemma}
	The following proof is based on suggestions of Dinakar Muthiah.
	\begin{proof}
This is analogous to the proof of \cite[Cor.~1 to Thm.~2]{Springer}: considering $\Sp_c(r)$ as a variety over a finite field, we show that $H^i(\Sp_c(r), \mathbb{Q}_\ell)$ is pure, and that $\Sp_c(r)$ has polynomial point count.  
	
To show purity, note that $\Sp_c(r)$ is a fibre of a Grassmannian (i.e.~parabolic) version of the Grothendieck-Springer alteration. Consider instead the preimage $\widetilde{\Sp}_c(r)$ of the corresponding Slodowy slice.  Then $\widetilde{\Sp}_c(r)$ is a smooth variety, and there is an isomorphism
	$$ H^\ast(\widetilde{\Sp}_c(r), \mathbb{Q}_\ell)  \longrightarrow H^\ast({\Sp}_c(r), \mathbb{Q}_\ell) $$
	compatible with the action of Frobenius, as in \cite[Lemma 2]{Springer}. $\Sp_c(r)$ is projective so $H^i(\Sp_c(r),\mathbb{Q}_\ell)$ has weight $\leq i$, while $\widetilde{\Sp}_c(r)$ is smooth so $H^i(\widetilde{\Sp}_c(r),\mathbb{Q}_\ell)$ has weight $\geq i$.  This proves purity. 
	
	Finally, $\Sp_c(r)$ has polynomial point count by \cite[Example 8, \S III.6]{Macdonald}.  
	\end{proof}
	
\begin{Remark}
There is a geometrically defined action of the Yangian on the direct sum (over $r$) of the homology of fibres $T^\ast \Gr(m-r, m) \rightarrow \mathcal N$, by \cite[Remark 8.7]{GV}.  Our varieties $\Sp_c$ are big Spaltenstein versions of these fibres, and so one expects an analogous geometrically defined Yangian action on their homology, cf.~\cite[Section 5.2.4]{Muthiah}.  The precise relationship between these two module structures on the homology of $\Sp_c$ is unclear at the moment.
\end{Remark}

\section{Relation to quasimap spaces}
We will now relate Springer fibres to various spaces of quasimaps into stacks.

\subsection{Various spaces of quasimaps} \label{se:various}
Let $ Y $ be a scheme and let $ G $ be a reductive group.  The space of quasimaps from $ \PP^1 $ to the stack $[Y/G] $ parametrizes pairs $(P, s) $ where $P $ is principal $ G $-bundle on $ \PP^1 $ and $ s : P \rightarrow Y $ is a $G$-equivariant morphism.  More precisely, we define $\QM(\PP^1, [Y/G]) $ to be the mapping stack $ \underline{\operatorname{Map}}(\PP^1, [Y/G]) $.  This works out to the following definition.

\begin{Definition}
	$\QM(\PP^1, [Y/G]) $ is the stack defined by
	\begin{align*}
	\QM&(\PP^1, [Y/G])(S) = \\
	&\{ (P, s): \text{ $P $ is a principal $G$-bundle on $ \PP^1 \times S $, $ s : P \rightarrow Y$ is $G$-equivariant}\}
	\end{align*}
	where $ S $ is an affine scheme.

	Let $ c \in Y $.  Assume that $ c $ has trivial stabilizer in $ G $.
	
	We define the stack of \textbf{based quasimaps} $\QM_c(\PP^1, [Y/G]) $ to be
	\begin{align*}
	\QM_c&(\PP^1, [Y/G])(S) = \{ (P, s, \psi): \text{ $P $ is a principal $G$-bundle on $ \PP^1 \times S $,} \\ 
	&\text{$ s : P \rightarrow Y$ is $G$-equivariant,} \\
	&\text{ $\psi : S \times G \rightarrow P|_{\{\infty\} \times S} $ is a trivialization such that $ s(\psi(S,1)) = c $.}\}
	\end{align*}
	
	Finally, we say that a quasimap $ (P, s) $ has degree $ \sigma \in \pi $ if $ P $ has degree $ \sigma $, and we write $  \QM^\sigma_c(\PP^1, [X/G])$ for the stack of based quasimaps of degree $ \sigma $. 
\end{Definition}
Note that since $ c $ has a trivial stabilizer in $ G $, the trivialization $ \psi $ is unique.

Now, let $N, G, c $ be as in previous sections.  Let $ G' $ be the commutator subgroup of $ G $ and let $ H = G/G' $.

For the remainder of this section, we will make the following assumptions.
\begin{enumerate}
	\item $ c $ has trivial stabilizer in $ G $.
	\item The scheme theoretic fibre of $ N \rightarrow N \sslash G' $ over the point $ [c] $ is the orbit $ G' c $.
\end{enumerate}

In this section, we will study three spaces of quasimaps $$ \QM_c(\PP^1, [\Phi^{-1}(0)/G]),\  \QM_c(\PP^1, [N/G]), \ \QM_c(\PP^1, [N\sslash G' / H ])$$  
We will show that there are maps
$$
\QM_c(\PP^1, [N\sslash G' / H ]) \leftarrow \QM_c(\PP^1, [N/G]) \rightarrow  \QM_c(\PP^1, [\Phi^{-1}(0)/G])
$$
and we will prove the following theorem.

\begin{Theorem} \label{th:QMcompare}
	Let $ c $ be as above.
	\begin{enumerate}
		\item $\QM_c(\PP^1, [N/G]) \rightarrow  \QM_c(\PP^1, [\Phi^{-1}(0)/G])$ is a homotopy equivalence.
		\item For each $ \sigma$,  $ \QM^\sigma_c(\PP^1, [N\sslash G' / H ]) $ has a contracting $ \C^\times $ action.
		\item $ \Sp_c(\sigma) $ is the central fibre of $$ \QM^\sigma_c(\PP^1, [N/G]) \rightarrow \QM^\sigma_c(\PP^1, [N\sslash G' / H ]) $$ and $ H_\bullet(\Sp_c(\sigma)) \cong H_\bullet\big( \QM^\sigma_c(\PP^1, [N/G])\big) $
	\end{enumerate}
\end{Theorem}

\begin{Remark}
	This theorem admits an obvious equivariant upgrade.  First, the loop rotation action of $ \Cx $ on $ \Sp_c $ extends to an action on the quasimap spaces, coming from the usual $ \Cx $ action on $ \PP^1 $.  All the maps involved are equivariant and the isomorphisms of homology extend to $ \Cx$-equivariant homology. Second, suppose as usual that we have chosen a flavour group $ F $ and that the stabilizer of $ c $ in $ \tG $ maps isomorphically onto $ F $.  Then we get an action of the flavour group on the quasimap spaces.  Similarly, we see that we obtain isomorphisms of $F$-equivariant homology.
\end{Remark}

\begin{proof}[Proof of Theorem \ref{th:QMcompare} (1)]
Since $ N \subset \Phi^{-1}(0)$, we have an obvious embedding $$\QM_c(\PP^1, [N/G]) \rightarrow \QM_c(\PP^1, [\Phi^{-1}(0)/G]) $$ given by taking a point $ (P, s, \psi) $ and extending $s$ by 0.

Define a $\C^\times$-action on $ T^* N $ by scaling on the $ N^* $ factor.  This gives us an action of $ \C^\times $ on $\QM_c(\PP^1, [\Phi^{-1}(0)/G]) $ which contracts $\QM_c(\PP^1, [\Phi^{-1}(0)/G]) $  to $\QM_c(\PP^1, [N/G])$.  
\end{proof}

\subsection{Torus quotients} \label{se:QMtorus}

In this section, we will study quasimaps into stacks of the form $ [Y / T] $ where $ Y = \Spec A $ is a finite-type affine scheme and $ T $ is a torus.  As before choose $ c \in Y(\C)$ with trivial stabilizer.

Because $ T $ is a torus, $ \pi $ is just the coweight lattice of $ T $, and so $ \sigma $ is a coweight.  For the purposes of this section, let $ X $ denote the weight lattice of $ T $.  For any $ \lambda \in X $, $ \sigma(\lambda) $ is an integer.   For any integer $ n $, we write $ \C[x,y]_n $ for the space of homogeneous polynomials in two variables of degree $ n $.  Of course if $ n < 0 $, then this space is $ 0$.  We have the usual multiplication map $ \C[x,y]_{n_1} \otimes \C[x,y]_{n_2} \rightarrow \C[x,y]_{n_1 + n_2} $ and dually a coproduct
$$
\Delta : \C[x,y]_{n_1 + n_2}^* \rightarrow \C[x,y]_{n_1}^* \otimes \C[x,y]_{n_2}^*
$$
which we will write as $ \Delta(\beta) = \sum \beta_1 \otimes \beta_2 $ (as is customary, we are suppressing the index of summation).  Also, we define $ \beta_\infty \in \C[x,y]_n^* $ by $ \beta_\infty(p) = p(0,1) $ for any $ p \in \C[x,y]_n $.

Let $ P $ be a principal $T$-bundle on $ \PP^1 $.  We will be interested in the algebra $ \C[P] $ of global functions on $ P $.  The variety $ P $ (and thus the algebra $ \C[P] $) has actions of $ \C^\times $ (by its usual action on $ \PP^1$) and $ T $ (since it is a principal $T$-bundle).

\begin{Lemma} \label{le:functionsP}
	Let $ P $ be a principal $T$-bundle on $ \PP^1 $ of degree $ \sigma $.  Then the algebra $ \C[P] $ can be described as
	$$
	\C[P] = \bigoplus_{\lambda \in X} \C[x,y]_{\sigma(\lambda)}
	$$
	where $ T $ acts on $ \C[x,y]_{\sigma(\lambda)} $ with weight $ \lambda $ and where the algebra structure comes from the usual multiplication of polynomials.
	
	Moreover, this principal $T$-bundle has a natural trivialization at $ \infty $ which is given by $ \beta_\infty: \C[P] \rightarrow \C $.
\end{Lemma}

\begin{proof}
	We have
	$$
	P = (\C^2 \smallsetminus \{0\} ) \times T / \C^\times
	$$
	where $ \C^\times $ acts on $ \C^2 \smallsetminus \{0\}  $ by inverse scaling and on $ T$ by the map $ \sigma : \C^\times \rightarrow T$.  Thus $ \C[P] = (\C[x,y] \times \C[T])^{\C^\times} $ which is easily reduced to the above formula.
\end{proof}

Now, let $ A $ be a finitely-generated algebra with an action of $ T $ and let $ c \in \Spec A $.  Thus we get an algebra morphism $ A \rightarrow \C $, written as $ a \mapsto a(c) $.  We use $ \sigma $ to collapse the $ X $-grading on $ A $ to a $ \Z $-grading, by setting $ A_n = \displaystyle{\oplus_{\sigma(\lambda) = n}} \A_\lambda $.  

We assume that $ c, \sigma $ are compatible in the following sense:
\begin{equation}
\label{eq:ConditionNonEmpty}
\text{ if $ \sigma(\lambda) < 0 $ and $ a \in \A_\lambda $, then $ a(c) = 0 $ }
\end{equation}
This condition is equivalent to $ c $ lying the attracting set for the $ \Cx $ action defined by $\sigma$.  It is a necessary condition for the non-emptyness of the quasimap space $ \QM^\sigma_c(\PP^1,[Y/T]) $.

Assuming the above condition, we define
$$
A^\sigma := \Sym( \bigoplus_{n\in \Z} A_n \otimes \C[x,y]^*_n ) / J
$$
where $ J $ is the ideal generated by
$$ (a_1 a_2) \otimes \beta - \sum (a_1 \otimes \beta_1) (a_2 \otimes \beta_2) 
$$
for $ a_1 \in \A_{n_1}, a_2 \in \A_{n_2}, \beta \in \C[x,y]^*_{n_1 + n_2} $, and $ n_1, n_2 \in \Z $.

Note that $ A^\sigma $ is a finitely generated algebra.  Indeed, if we have homogeneous generators $ a_1, \dots, a_m $ for $A $ of degrees $ \lambda_1, \dots, \lambda_m$, then $ \{ a_k \otimes (x^p y^q)^* \}$ generates $ A^\sigma $, for $ 1 \le k \le m $ and $ p+q = \sigma(\lambda_k) $.

Then we define a further quotient
$$
A^\sigma_c := A^\sigma / ( a \otimes \beta_\infty - a(c) : a \in \A_\lambda)
$$

\begin{Example} \label{eg:Zastava}
	The following example was our motivation for the above definition of $ A^\sigma_c$.
	
	Let $ G $ be a semisimple group and let $ Y = G \sslash U $ be the base affine space and let $ T $ act on $ Y $ by right multiplication.  Then $ Y = \Spec A $, where  $ A = \oplus V(\lambda)^* $ is the direct sum of all duals of the irreducible representation of $ G $ with multiplication defined by
	$$ \iota^* : V(\lambda_1)^* \otimes V(\lambda_2)^* \rightarrow V(\lambda_1 + \lambda_2)^* \text{ dual to } \iota : V(\lambda_1 + \lambda_2) \rightarrow V(\lambda_1) \otimes V(\lambda_2) $$
	and where $ T $ acts on $ A $ by acting by weight $ \lambda $ on $ V(\lambda)^* $.
	
	We let $ c : A \rightarrow \C $ be given by the highest weight vector in each $ V(\lambda) $.
	
	In this case $ A^\sigma = \Sym( \oplus_\lambda V(\lambda)^* \otimes \C[x,y]^*_{\sigma(\lambda)} ) / J $.  Thus an algebra homomorphism $ \phi : A^\sigma \rightarrow \C $ gives us an element $ \phi_\lambda \in V(\lambda) \otimes \C[x,y]_{\sigma(\lambda)}  $.  These elements $ \phi_\lambda $ must satisfy the relations in $ J $, which is equivalent to the equations
	$$ \phi_{\lambda_1}  \phi_{\lambda_2} = \iota(\phi_{\lambda_1 + \lambda_2}) \text{ in } V(\lambda_1) \otimes V(\lambda_2) \otimes\C[x,y]_{\sigma(\lambda_1 + \lambda_2)} .$$
	
	Moreover, if $ \phi : A^\sigma_c \rightarrow \C $, then we see that we have the additional condition that $ \phi_\lambda(0,1) = v_\lambda $, the highest weight vector.
	
	Comparing with Definition 5.2 in \cite{FMSemiinfinite}, we conclude that $ \Spec A^\sigma_c $ equals the Zastava space $ Z^\sigma $, if  $\sigma $ lies in the cone of positive coroots, and is empty otherwise (this is equivalent to the condition (\ref{eq:ConditionNonEmpty})).
\end{Example}

Returning to the general case, we will now prove that the based quasimap space is an affine scheme.

\begin{Theorem} \label{th:QMtorus}
	Let $ A $ be as above and let $ Y = \Spec A $.  The stack $ \QM^\sigma_c(\PP^1,[Y/T]) $ is isomorphic to $ \Spec A^\sigma_c$.
\end{Theorem}
\begin{proof}
	We will prove that for $ S = \Spec R $, we have
	$$
	\QM^\sigma_c(\PP^1,[Y/T])(S) = \Hom_{alg}(A^\sigma_c, R)
	$$
	
	Now suppose that we are given $ \phi \in \Hom_{alg}(A^\sigma_c, R) $. Then since $ A^\sigma_c $ is a quotient of a symmetric algebra, for each $ \lambda$, we obtain a linear map $ \A_\lambda \otimes \C[x,y]_{\sigma(\lambda)}^* \rightarrow R $ and thus dualizing we obtain a linear map $ \phi_\lambda : \A_\lambda \rightarrow \C[x,y]_{\sigma(\lambda)} \otimes R $.  These linear maps $ \phi_\lambda $ must satisfy the conditions of the ideal $ J$ defining $ A^\sigma $ and also the further relation defining $ A^\sigma_c $ and so we see that
	$$
	\phi_\lambda(a_1 a_2) = \phi_{\lambda_1}(a_1) \phi_{\lambda_2}(a_2) \quad \beta_\infty(\phi_\lambda(a)) = a(c)
	$$
	
	On the other hand, consider a principal bundle $ P $ on $ \PP^1 \times S $ of degree $ \sigma $.  Since $ Y = \Spec A $, a $T$-equivariant map $ s : P \rightarrow Y $ is equivalent to an $X$-graded algebra map $ A \rightarrow \C[P] $.  Now, by Lemma \ref{le:functionsP}, we have
	$$
	\C[P] = \oplus_{\lambda} \C[x,y]_{\sigma(\lambda)} \otimes R
	$$
	
	We conclude that the collection $ (\phi_\lambda)_{\lambda \in X}$ defines am $X$-graded algebra map $ A \rightarrow \C[P] $ compatible with the map $ \C[P] \rightarrow \C $ and hence a $T$-equivariant map $ s : P \rightarrow Y $ which restricts to $ c $ on the trivialization.  Thus, from $ \phi \in \Hom_{alg}(A^\sigma_c, R) $ we obtain an element of $ \QM^\sigma_c(\PP^1,[Y/T])(S) $.  It is easy to see that this defines a bijection.
	
\end{proof}

The space $  \Spec A^\sigma_c$ has a distinguished point, also denoted $ c $, given by the map $ A^\sigma_c \rightarrow \C $ defined by $ a \otimes \beta \mapsto a(c) \beta(y^n) $ for $ a \otimes \beta \in   A_n \otimes \C[x,y]^*_n $.  Under the above isomorphism, this distinguished point can be viewed as the following quasimap.  Recall that 
$$ 
P = (\C^2 \smallsetminus \{0\}) \times T / \Cx
$$
where $ \Cx $ acts by inverse scaling on  $\C^2 \smallsetminus \{0\}$ and on $ T $ by the coweight $ \sigma$.  We define a map 
$$ (\C^2 \smallsetminus \C \times \{0\}) \times T \rightarrow Y \quad (x,y,t) \mapsto \sigma(y)t c $$
This is invariant for the action of $ \Cx$ and extends to a map $ P \rightarrow Y $ if and only if the condition (\ref{eq:ConditionNonEmpty}) holds.  Comparing with the proof of Theorem \ref{th:QMtorus}, we see that this quasimap corresponds to the point $ c \in \Spec A^\sigma_c $.

We define an action of $ \Cx $ on $ A^\sigma_c $ by setting the variable $ x $ to have degree $ 1 $ and the variable $ y $ degree $0$.

\begin{Proposition}
	This $ \Cx $ action on $ \QM^\sigma_c(\PP^1, [Y/T]) $ contracts this space to the point $ c $. 
\end{Proposition}

\begin{proof}
	It is easily see that with respect to the above action $ A^\sigma_c $ is non-negatively graded and has $  (A^\sigma_c)^\Cx $ one-dimensional with augmentation morphism given by $ c$.
\end{proof}

\subsection{Comparing BFN Springer fibres and quasimaps}
We return to our general setup of a reductive group $ G$, a representation $ N $, and a $\chi$-stable point $ c \in N $ satisfying our assumptions from section \ref{se:various}.

We will now apply the setup of the previous section to the affine scheme $ Y = N \sslash G' $ with the action of the torus $ H := G/G'$.  Note that we have a morphism $ N \rightarrow N\sslash G' $ which is invariant for the action of $ G'$ and we write $ [c] $ for the image of $ c $ under this morphism.

Applying Theorem \ref{th:QMtorus}, we deduce the following result which implies Theorem \ref{th:QMcompare} (2).

\begin{Corollary} \label{pr:ContractQM}
	$$\QM^\sigma_c(\PP^1, [N\sslash G' / H ]) = \Spec A^\sigma_{[c]}$$
	where $ A = \C[N]^{G'} $.  In particular, the action of $ \Cx $  on $ \QM^\sigma_{[c]}(\PP^1, [N\sslash G' / H ]) $ contracts this space to the point $ [c] \in N \sslash G'$.
\end{Corollary}

There is a morphism of stacks $ \pi : [N/G] \rightarrow [N\sslash G' / H] $ which leads to a morphism 
$$
\pi_\sigma : \QM^\sigma_c(\PP^1, [N/G])\rightarrow \QM^\sigma_{[c]}(\PP^1, [N\sslash G' / H ])
$$
which can be described as follows. Given a point $ (P, s, \psi) \in  \QM^\sigma_c(\PP^1, [N/G])(S) $, we consider the quotient $ P/G' $ which is a principal $H = G/G'$ bundle.  The map $ s : P \rightarrow N $ descends to a map $ s/G' : P/G' \rightarrow N \sslash G' $ and similarly the trivialization descends as well.  Thus we obtain the point 
$$\pi_{\sigma}(P,s,\psi) = (P/G', s/G', \psi/G') \in \QM^\sigma_{[c]}(\PP^1, [N\sslash G' / H ])(S)
$$

Now, we relate the central fibre of $ \pi $ to the Springer fibre.  We write $ \pi^{-1}([c]):= \sqcup_{\sigma} \pi_\sigma^{-1}{[c]} $.

Define a new moduli space $ \widehat \Sp_c $ by
\begin{align*}
\widehat \Sp_c(S) 
= &\{ (P, s, \psi): \text{ $P $ is a principal $G$-bundle on $ \PP^1 \times S $,} \\ 
&\text{$ s : P \rightarrow N$ is $G$-equivariant,} \\
&\text{ $\psi : \PP^1 \smallsetminus \{0\} \times S \times G \rightarrow P|_{ \PP^1 \smallsetminus \{0\} \times S} $ is a trivialization, such that } \\
&\text{ $ s\circ \psi : \PP^1 \smallsetminus \{0\} \times S \times \{1\} \rightarrow N $ is the constant morphism (with image $c$).}\} 
\end{align*}
As usual, we write $ \widehat \Sp_c(\sigma) $ for the locus where $ P $ has degree $ \sigma $.

The following result completes the proof of Theorem \ref{th:QMcompare} (3).

\begin{Proposition} \label{pr:CompareFcQM}
	\begin{enumerate}
		\item	The map $ (P,s,\psi) \mapsto (P|_{D \times S}, \psi|_{D\times S}) $ defines an isomorphism $ \widehat \Sp_c \rightarrow \Sp_c $.
		\item The map $ (P, s, \psi) \mapsto (P, s, \psi|_{\{\infty\} \times S}) $ gives an isomorphism $ \widehat \Sp_c(\sigma) \cong \pi_\sigma^{-1}([c])$.
	\end{enumerate}
\end{Proposition}

\begin{proof}
	\begin{enumerate}
		\item First, we note that the equation $ s(\psi(x,S,1)) = c $ determines $ s $ on the restriction of $ P $ over $  \PP^1 \smallsetminus \{0\} \times S$, a dense subset of $ P $.  Thus, given $ (P, \psi) $ if $ s $ exists, it is unique.  The result thus follows from a theorem of Beauville-Laszlo \cite{BL}.
		\item 
		First, we note that for any $ (P,s,\phi) \in \widehat \Sp_c(S) $, the composition $ s : P \rightarrow N \rightarrow N \sslash G' $ descends to a map $ s /G' :  P / G' \rightarrow N \sslash G' $.  The trivialization of $ P $ over $ \PP^1 \smallsetminus \{0\} \times S $ descends to a trivialization of $ P/G' $ over this locus.  With respect to this trivialization, the map $ s / G' $ is constant with image $ [c]$ and thus gives the central point $ [c] \in QM_{[c]}(\PP^1, [N \sslash G' / H]) $.  This shows that $ (P, s, \psi|_{\{\infty\} \times S}) \in \pi^{-1}([c])(S) $.
		
		Conversely, let $ (P,s,\psi_\infty) \in \pi^{-1}([c])$. Thus there exists a trivialization $ \psi_H $ of $ P/G' $ on $ \PP^1 \smallsetminus \{0\} \times S $ and such that $s/G' $ is given by 
		$$ P/G'|_{\PP^1 \smallsetminus \{0\} \times S}\cong H \times \PP^1 \smallsetminus \{0\} \times S \rightarrow N \sslash G' \quad (1,x) \mapsto [c] $$
		Choose a trivialization $ \psi$ of $ P $ on $ \PP^1 \smallsetminus \{0\} \times S $ extending $ \psi_\infty $ and $ \psi_H $, so we obtain the commutative diagram
		\begin{equation*}
		\begin{tikzcd}
		\PP^1 \smallsetminus \{0\} \times S \times G \arrow{r}{\psi} \arrow{d} & P|_{\PP^1 \smallsetminus \{0\} \times S} \arrow{r}{s} \arrow{d} & N \arrow{d} \\
		\PP^1 \smallsetminus \{0\} \times S \times H \arrow{r}{\psi_H} & P/G' |_{\PP^1 \smallsetminus \{0\} \times S} \arrow{r}{s/G'} & N \sslash G'
		\end{tikzcd}
		\end{equation*}
		Take a point $ (x,1) \in \PP^1 \smallsetminus \{0\} \times S \times G $. As the path down and right gives the point $ [c] $, we see that the path right must land in the fibre over $ [c] $.  By the assumption that this fibre is scheme-theoretically isomorphic to $ G' $, we see that there is a morphism $ g : \PP^1 \smallsetminus \{0\} \times S  \rightarrow G $ such that $ s \circ \psi(1,x) = g(x)c $.  We change the trivialization $ \psi $ by this morphism $ g $ and then we see that under the new trivialization, the section $s $ is constant (and equal to $c $).  Thus it defines a point in $ \widehat \Sp_c(S) $.   
		
		This gives the inverse morphism $ \pi^{-1}([c]) \rightarrow \widehat \Sp_c $ and thus we have the desired isomorphism $ \Sp_c(\sigma) \cong \pi_\sigma^{-1}([c])$.
	\end{enumerate}
\end{proof}

\subsection{Laumon's quasiflags spaces} \label{se:Laumon}

As in Examples \ref{eg:sln} and \ref{eg:sln2}, we consider the case in which $ G = \prod_{i = 1}^{n-1} GL_i $ and $ N = \oplus_{i = 1}^{n-1} \Hom(\C^i, \C^{i+1}) $, choose $ \chi : G \rightarrow \Cx $ to be given by $\chi(g) = \prod_i \det(g_i)^{-1}$ and take $ c \in N $ to be the point corresponding to the standard embeddings of $ \C^i $ into $ \C^{i+1} $.

In this case the $\chi$-stable points of $N $ are the injective homomorphisms and $ N\sslash_\chi G \cong FL_n $, the variety of full flags in $ \C^n$.  We also have $ G' = \prod_{i=1}^{n-1} SL_i$ and $ G/G' = (\C^\times)^{n-1} $.

We can identify $ \QM_c(\PP^1, [N/G])$ with Laumon's based quasiflag space $ \La_n$.
\begin{align*}
\La_n := \{ &0 = \sF_0 \subset \sF_1 \subset \cdots \subset \sF_{n-1} \subset \sF_n = \OO_{\PP^1} \otimes \C^n : \\
&\text{ $ \sF_i $ is a locally free sheaf of rank $ i $,\ $ \sF_i|_\infty = \C^i $} \}
\end{align*}
In this language the Springer fibre $ \Sp_c $ is identified with the following locus.

$$
\Sp_c = \{ \sF_\bullet : \sF_i|_{\PP^1 \smallsetminus \{0 \}} = \OO_{\PP^1 \smallsetminus \{0 \}} \otimes \C^i \} 
$$

In \cite{Kuz}, Kuznetsov studied a map $ \pi : \La_n^\sigma \rightarrow \Za_n^\sigma $, where $ \Za^\sigma $ denotes  Drinfeld's Zastava space.  As discussed in Example \ref{eg:Zastava}, $ \Za^\sigma_n = \QM_c^\sigma(\PP^1,[SL_n \sslash U / T])$ where $ T $ denotes the torus of $ SL_n $.  

In fact, we have the following.
\begin{Lemma}
	There is an isomorphism $ SL_n \sslash U = N \sslash G' $ compatible with the action of $ T = (\C^\times)^{n-1} $ on both sides.
\end{Lemma}

\begin{proof}
	It suffices to show that $\C[SL_n]^U = \C[N]^{G'} $.  We already know the left hand side is given by the direct of all irreducible representations of $ SL_n$.
	
	Given  any dominant weight $ \lambda = \sum \lambda_i \omega_i $, we can define $ \theta : G \rightarrow \Cx $ by $\theta(g_1, \dots, g_{n-1}) = \prod \det(g_i)^{\lambda_i}$.  So we must show that for this $ \theta$, we have an isomorphism of $ SL_n $ representations $ \C[N]^{G, \theta} = V(\lambda) $. 
	
	Using Howe duality, we have
	\begin{gather*}
	\C[N] = \Sym(N^*) = \Sym \bigl(\bigoplus_{i=1}^{n-1} \C^i \otimes (\C^{i+1})^* \bigr) \\
	= \bigoplus_{\mu^1, \dots, \mu^{n-1}} \bigotimes_{i=1}^{n-1} V_{GL_i}(\mu^i) \otimes V_{GL_{i+1}}(\mu^i)^* \\
	=\bigoplus_{\mu^1, \dots, \mu^{n-1}} \bigotimes_{i=1}^n V_{GL_i}(\mu^{i-1})^* \otimes V_{GL_i}(\mu^i)
	\end{gather*}
	where $ \mu^i \in \mathbb N^i $.  Also if $ j \le i $ and $ \mu \in \mathbb \N^j$, then $V_{GL_i}(\mu) $ denotes the irreducible representation of $ GL_i $ of highest weight $ (\mu,0, \dots,0) $.  (Also we make the convention that $ \mu^n = 0 $, so $V_{GL_n}(\mu^n) = \C$.)
	
	Now we are looking inside this space for those vectors where $ G = \prod_{i=1}^{n=1} GL_i$ acts by $ \theta $.  Now
	$$ (V_{GL_i}(\mu^{i-1})^* \otimes V_{GL_i}(\mu^i))^{GL_i, \det^{\lambda_i}} = \begin{cases} \C \text{ if $ -\mu^{i-1} + \mu^i = \lambda_i \omega_i$ } \\
	0 \text{ otherwise}
	\end{cases}
	$$
	Combining these equations together gives the desired result.
	
\end{proof}

\begin{Corollary} \label{th:LaumonQM}
	We have an isomorphism $$\QM_c^\sigma(\PP^1,[SL_n \sslash U / T]) \cong   \QM_c^\sigma(\PP^1,[N\sslash G' /  (\C^\times)^{n-1}])$$  
	Under this isomorphism, the map $  \La_n^\sigma \rightarrow \Za_n^\sigma $ studied by Kuznetsov coincides with our map $ \pi $ above.
\end{Corollary} 

\begin{proof}
	The first claim follows from the Lemma.  The second follows by tracing through the definitions.
\end{proof}

Thus we see that the central fibre of $ \pi_\sigma $ studied in section 2 by Kuznetsov \cite{Kuz} coincides with our BFN Springer fibre $ \Sp_c(\sigma) $.  In particular, Kuznetsov computes the dimension of $ \Sp_c(\sigma) $ and its Poincar\'e polynomial. 

\subsection*{Declaration of Funding and Conflicts of Interest}
	
J.K. was supported by an NSERC Discovery Grant and a Simons Fellowship. J.H. is part of the Simons Collaboration on Homological Mirror Symmetry supported by Simons Grant 390287. 

This research was supported in part by Perimeter Institute for Theoretical Physics. Research at Perimeter Institute is supported by the Government of Canada through the Department of Innovation, Science and Economic Development Canada and by the Province of Ontario through the Ministry of Research, Innovation and Science. 

The authors have no relevant financial or non-financial interests to disclose.


\begin{thebibliography}{KTWWY2}
\bibitem[AHR]{AHR}
P. Achar, A. Henderson, S. Riche, Geometric Satake, Springer correspondence, and small representations II, \textit{Represent. Theory} \textbf{19} (2015), 94--166.

\bibitem[BKK]{BKK}
P. Baumann, J. Kamnitzer, A. Knutson, The Mirkovic-Vilonen basis and Duistermaat-Heckman measures, \textit{Acta Math.} \textbf{227} (2021), 1--101.

\bibitem[BL]{BL}
A.~Beauville, Y.~Laszlo, Un lemme de descente, \textit{C.~R.~Acad.~Sci.~Pari.~S\'er.~I Math.} \textbf{320} (1995), no.~3, 335--340.

\bibitem[BKV]{BKV} 
A. Bouthier, D. Kazhdan, Y. Varshavsky, Perverse sheaves on infinite-dimensional stacks, and affine Springer theory; \arxiv{2003.01428}.

\bibitem[BLPW1]{BLPW1}
T. Braden, T. Licata, N. Proudfoot, B. Webster, Hypertoric cateogry $\mathcal{O}$, \textit{Advances in Mathematics}, no.~3-4, (2012), 1487–1545 

\bibitem[BLPW2]{BLPW}
T. Braden, T. Licata, N. Proudfoot, B. Webster, Quantizations of conical symplectic resolutions II: category $\OO$ and symplectic duality, \textit{Ast\'erisque} \textbf{384} (2016), 75--179.

\bibitem[B]{B}
A. Braverman, Instanton counting via affine Lie algebras. I. Equivariant J-functions of (affine) flag manifolds and Whittaker vectors, \textit{Algebraic Structures and Moduli Spaces} (2004), 113-132.



\bibitem[BFFR]{BFFR}
A. Braverman, B. Feigin, M. Finkelberg, L. Rybnikov, A finite analog of the AGT relation I: finite W-algebras and quasimaps' spaces, \textit{Comm. Math Phys.} \textbf{308} (2011), no.~2, 457--478.

\bibitem[BF]{BF}
A. Braverman, M. Finkelberg, Coulomb branches of 3-dimensional gauge theories and related structures,  in \textit{Geometric Representation Theory and Gauge Theory}, Lecture Notes in Mathematics {\bf 2248} (2019). 

\bibitem[BFN1]{BFN1}
A. Braverman, M. Finkelberg, H. Nakajima, Towards a mathematical definition of Coulomb branches of $3$-dimensional $\mathcal N=4$ gauge theories, II, \textit{Adv. Theor. Math. Phys.} {\bf 22} (2019), no.~5, 1071--1147. 

\bibitem[BFN2]{BFN2}
A. Braverman, M. Finkelberg, H. Nakajima, Coulomb branches of $3d$ $\mathcal N=4$ quiver gauge theories and slices in the affine Grassmannian (with appendices by Alexander Braverman, Michael Finkelberg, Joel Kamnitzer, Ryosuke Kodera, Hiraku Nakajima, Ben Webster, and Alex Weekes), \textit{Adv. Theor. Math. Phys.} {\bf 23} (2019), no.~1, 75--166.

\bibitem[BDG]{BDG}
M. Bullimore, T. Dimofte, D. Gaiotto, The Coulomb Branch of 3d $\mathcal{N}= 4$ Theories, \textit{Comm. in Math. Physics}, textbf{2} (2017), 671–751.

\bibitem[BDGH]{BDGH}{
M. Bullimore, T. Dimofte, D. Gaiotto, J. Hilburn, Boundaries, mirror symmetry, and symplectic duality in $3d$ $\mathcal{N} = 4$ gauge theories, \textit{JHEP} \textbf{10} (2016), front matter + 191 pp.}

\bibitem[BDGHK]{BDGHK}{
M. Bullimore, T. Dimofte, D. Gaiotto, J. Hilburn, H. Kim, Vortices and Vermas, \textit{Advances in Theoretical Physics}, \textbf{22} (2020)
}


\bibitem[ChG]{CG}
N. Chriss, V. Ginzburg, \textit{Representation Theory and Complex Geometry}, Birk\"auser, 2010.


\bibitem[CoG]{CG2}
K. Costello, D. Gaiotto, Vertex operators and $3$d $\mathcal{N}=4$ theories, \textit{JHEP}, \textbf{5} (2019).

\bibitem[CCG]{CCG}
K. Costello, T. Creutzig, D. Gaiotto, Higgs and Coulomb branches from vertex operator algebras, \textit{JHEP}, \textbf{3} (2019).

\bibitem[DGGH]{DGGH}{
T. Dimofte, N. Garner, M. Geracie, J. Hilburn, Mirror symmetry and line operators, \textit{JHEP}, \textbf{2} (2020).}


\bibitem[FFFR]{FFFR}
B. Feigin, M. Finkelberg, I. Frenkel, and L. Rybnikov, Gelfand–Tsetlin algebras and cohomology rings of Laumon spaces, \textit{Selecta Math.} \textbf{17} (2011), 337--361. 

\bibitem[FM]{FMSemiinfinite}
M. Finkelberg, I. Mirkovic, Semiinfinite Flags I. Case of global curve $\mathbb P^1$; \arxiv{alg-geom/9707010}.

\bibitem[FT]{FT}
M. Finkelberg, A.Tsymbaliuk, Shifted quantum affine algebras: integral forms in type $A$ (with appendices by Alexander Tsymbaliuk and Alex Weekes), \textit{Arnold Math. J.} \textbf 5 (2019), no. 2, 197--283. 

\bibitem[GK]{GK}
N. Garner, O. Kivinen, Hilbert schemes on planar curve singularities are generalized affine Springer fibres; \textit{IMRN}, 2023, no. 8, 6402--6460.




\bibitem[GKLO]{GKLO}
A.~Gerasimov, S.~Kharchev, D.~Lebedev, and S.~Oblezin,  On a class of
representations of the Yangian and moduli space of monopoles, \textit{Comm. Math. Phys.} \textbf{260} (2005). 511--525 
arxiv:math/0409031.

\bibitem[Gi]{Ginzburg}
V. Ginzburg, Lectures on Nakajima's quiver varieties, \textit{Geometric methods in representation theory. I}, 149--219, S\'emin.~Congr., 24-I, Soc.~Math.~France, 2012.

\bibitem[GV]{GV}
V. Ginzburg, E. Vasserot, Langlands reciprocity for affine quantum groups of type $A_n$, \textit{IMRN} (1993), no.~3, 67--85.

\bibitem[GKM]{GKM}
M. Goresky, R. Kottwitz, R. MacPherson, Purity of equivalued affine Springer fibers, \textit{Represent. Theory} \textbf{10} (2006), 130--146.

\bibitem[Hik]{Hikita}
T. Hikita, An Algebro-Geometric Realization of the Cohomology Ring of Hilbert Scheme of Points in the Affine Plane, \textit{IMRN} \textbf{8} (2017), 2538--2561.

\bibitem[Hil]{H}
J. Hilburn, Hypergeometric systems and projective modules in hypertoric category $\mathcal{O}$,  \textit{University of Oregon Dissertation}, (2016).

\bibitem[HY]{HY}
J. Hilburn, P. Yoo, Symplectic duality and Geometric Langlands,  \textit{In preparation}.

\bibitem[KMW]{KMW}
J. Kamnitzer, D. Muthiah, A. Weekes, On a reducedness conjecture for spherical Schubert varieties and slices in the affine Grassmannian, \textit{Transform. Groups} \textbf{23} (2018)
no. 3,  707--722.

\bibitem[KMP]{KMP}
 J. Kamnitzer, M. McBreen, and N. Proudfoot, The quantum Hikita conjecture,  \textit{Adv. Math.} \textbf{390} (2021), 45 pages.

\bibitem[KTWWY1]{moncrystals}
 J. Kamnitzer, P. Tingley, B. Webster, A. Weekes, and O. Yacobi, Highest weights for truncated shifted Yangians and product monomial crystals, \textit{J. Combinatorial Algebra} \textbf{3} (2019), 237--303.
 
 \bibitem[KTWWY2]{klrpaper}
  J. Kamnitzer, P. Tingley, B. Webster, A. Weekes, and O. Yacobi, On category $\cO$ for affine Grassmannian slices and categorified tensor products, \textit{Proc. London Math. Soc.} \textbf{119} (2019), issue 5, 1179--1233.

\bibitem[Kr]{Krylov}
V. Krylov, Integrable Crystals and Restriction to Levi Subgroups Via Generalized Slices in the Affine Grassmannian, \textit{Funktsional. Anal. i Prilozhen.}, \textbf{52} (2018), 40--65.

\bibitem[Ku]{Kuz}
A. Kuznetsov,  Laumon's resolution of Drinfeld's compactification is small, \textit{Math. Res. Lett.} \textbf 4 (1997), 349--364.

\bibitem[Ma]{Macdonald}
I. Macdonald, \textit{Symmetric functions and Hall polynomials}, 2nd Ed., Oxford University Press, 1995.

\bibitem[Mo]{Molev}
A. Molev, Yangians and classical Lie algebras, Mathematical Surveys and Monographs {\bf 143}, \textit{Amer. Math. Soc} (2007), xviii+400. 

\bibitem[Mu]{Muthiah}
D. Muthiah, Weyl group action on weight zero Mirkovi\'c-Vilonen basis and equivariant multiplicities, \textit{Adv.~Math.} \textbf{385} (2021), 40 pp.

\bibitem[N1]{NakKM}
H. Nakajima, Quiver varieties and Kac-Moody algebras, \textit{Duke Math. J.} \textbf{91} (1998), no.~3, 515--560.

\bibitem[N2]{NakHandsaw}
H. Nakajima, Handsaw quiver varieties and finite W--algebras, \textit{Mosc. Math. J.} \textbf{12} (2012), vol.~3, 633--666.

\bibitem[OY]{OY}
A. Oblomkov, Z. Yun, Geometric representations of graded and rational Cherednik algebras,  \textit{Advances in Mathematics} \textbf{292} (2016), 601-706.

\bibitem[R]{Raskin}
S.~Raskin, D-modules on infinite dimensional varieties; \href{https://web.ma.utexas.edu/users/sraskin/dmod.pdf}{https://web.ma.utexas.edu/users/sraskin/dmod.pdf}.

\bibitem[S]{Springer}
T.A. Springer, A purity result for fixed point varieties in flag manifolds, \textit{J. Fac. Sci. Univ. Tokyo Sect. IA Math} \textbf{31} (1984), no.~2, 271--282.

\bibitem[T]{T}
C. Teleman, The role of Coulomb branches in $2$D gauge theory, \arxiv{1801.1012} (2018)


\bibitem[VV]{VV}
M. Varagnolo, E. Vasserot, Finite-dimensional representations of DAHA and affine Springer fibers: The spherical case, \textit{Duke Math. J.}
\textbf{147}, no. 3 (2009), 439--540.

\bibitem[WWY]{WWY}
B. Webster, A. Weekes, O. Yacobi, A quantum Mirkovic-Vilonen isomorphism, \textit{Represent. Theory} \textbf{24} (2020), 38--84.

\bibitem[Web]{Web}
B. Webster, Koszul duality between Higgs and Coulomb categories
		$\mathcal{O}$, \arxiv{1611.06541}.

\bibitem[Wee]{Alex}
A. Weekes, Generators for Coulomb branches of quiver gauge theories; \arxiv{1903.07734}.

\bibitem[Y]{Zhiwei}
Z. Yun, Lectures on Springer theories and orbital integrals, \textit{Geometry of Moduli Spaces and Representation Theory}, IAS/Park City Mathematics Series \textbf{24} (2017).



\end{thebibliography}
\end{document}